\documentclass[10pt,a4paper]{amsart}
\pdfoutput=1
\usepackage[subsec,skipfonts]{enrabbrev}

\DeclareSymbolFont{bbold}{U}{bbold}{m}{n}
\DeclareSymbolFontAlphabet{\mathbbold}{bbold}
\newcommand{\bbone}{\mathbbold{1}}

\usepackage{eucal}

\title{$\infty$-Operads via Symmetric Sequences}

\newcommand{\igpd}{$\infty$-groupoid}
\newcommand{\igpds}{$\infty$-groupoids}
\newcommand{\Dop}{\simp^{\op}}

\newcommand{\DF}{\bbDelta_{\mathbb{F}}}
\newcommand{\DFop}{\DF^{\op}}

\newcommand{\hDFSop}{\widehat{\bbDelta}_{\mathbb{F},\mathcal{S}}^{\op}}
\newcommand{\hDFUop}{\widehat{\bbDelta}_{\mathbb{F},\mathcal{U}}^{\op}}
\newcommand{\hDFVop}{\widehat{\bbDelta}_{\mathbb{F},\mathcal{V}}^{\op}}

\newcommand{\DFX}{\bbDelta_{\mathbb{F},X}}
\newcommand{\DFXop}{\DFX^{\op}}
\newcommand{\xF}{\mathbb{F}}
\newcommand{\xFeq}{\xF^{\simeq}}
\newcommand{\Finj}{\mathbb{F}_{\txt{inj}}}

\newcommand{\DFV}{\DF^{\mathcal{V}}}
\newcommand{\DFVop}{\DF^{\mathcal{V},\op}}

\newcommand{\Fact}{\txt{Fact}}

\newcommand{\Span}{\txt{Span}}

\DeclareMathOperator{\Tw}{Tw}
\newcommand{\OSpan}{\overline{\txt{Span}}}

\newcommand{\icoopds}{$\infty$-cooperads}
\newcommand{\Coopd}{\txt{Coopd}}
\newcommand{\OpdA}{\Opd^{\txt{aug}}}
\newcommand{\opd}{\txt{opd}}
\newcommand{\CoopdA}{\Coopd^{\txt{coaug}}}

\newcommand{\COLL}{\txt{COLL}}
\newcommand{\Coll}{\txt{Coll}}

\newcommand{\ncsquare}[8]{%
  \begin{tikzcd}%
     #1 \arrow{r}{#5} \arrow{d}{#6} \pgfmatrixnextcell #2 %
      \arrow{d}{#7} \\%
     #3 \arrow{r}{#8} \pgfmatrixnextcell #4%
   \end{tikzcd}%
 }
 \newcommand{\nlcsquare}[4]{\ncsquare{#1}{#2}{#3}{#4}{}{}{}{}}

\newcommand{\eg}{e.g.\@}
\newcommand{\ie}{i.e.\@}
\newcommand{\cf}{cf.\@}

\newcommand{\bbS}{\bbSigma}
\newcommand{\bbL}{\bbLambda}
\newcommand{\hbbS}{\widehat{\bbS}}
\newcommand{\bbSM}{\bbS\mathcal{M}}
\newcommand{\Unf}{\txt{Unf}}

\newcommand{\Algd}{\txt{Algd}}

\begin{document}

\begin{abstract}
  We construct a generalization of the Day convolution tensor product
  of presheaves that works for certain double
  $\infty$-categories. Using this construction, we obtain an \icatl{}
  version of the well-known description of (one-object) operads as
  associative algebras in symmetric sequences; more generally, we show
  that (enriched) \iopds{} with varying spaces of objects can be
  described as associative algebras in a double \icat{} of symmetric
  collections.
\end{abstract}

\maketitle

\tableofcontents

\section{Introduction}
The theory of operads is a general framework for encoding and working
with algebraic structures, first introduced in the early 70s in order
to describe homotopy-coherent algebraic operations on topological
spaces~\cite{MayGeomIter,BoardmanVogt}. Since then, the theory has
found many applications in diverse areas of mathematics --- aside from
algebraic topology, where operads in topological spaces, simplicial
sets, and spectra have numerous uses (see for example
\cite{MayGeomIter,BoardmanVogt,ChingThesis,SmirnovOperads,DwyerHessKnots},
among many others), operads in vector spaces and chain complexes (also
known as \emph{linear operads} and \emph{dg-operads}, respectively)
are by now a well-studied topic in algebra (see for instance
\cite{GinzburgKapranov,LV}), with applications in mathematical physics
(\cf{} \cite{MarklShniderStasheff}) and algebraic geometry (\eg{}
\cite{KrizMay}), while operads in sets have become a standard tool in
combinatorics (\cf{} \cite{GiraudoHab,MendezSetOpds}).

Classically, an operad $\mathbf{O}$ in a symmetric monoidal category
$\mathbf{C}$ consists of a sequence $\mathbf{O}(n)$ of objects of
$\mathbf{C}$, where the symmetric group $\Sigma_{n}$ acts on
$\mathbf{O}(n)$ (this data is called a \emph{symmetric sequence})
together with a unital and associative composition operation. If
$\mathbf{C}$ has colimits indexed by groupoids and the tensor product
preserves these, then this data can be conveniently encoded using the
\emph{composition product} of symmetric sequences. This is a monoidal
structure on symmetric sequences, given by the formula\footnote{Note
  that this is the \emph{reverse} ordering of the product compared to many
  references; this convention corresponds to the one that naturally
  appears in our \icatl{}  construction.}
\[(X \circ Y)(n) \cong \coprod_{k =0}^{\infty} \left(
  \coprod_{i_{1}+\cdots+i_{k} = n} (X(i_{1}) \otimes \cdots \otimes
  X(i_{k})) \times_{\Sigma_{i_{1}} \times \cdots \times
    \Sigma_{i_{k}}} \Sigma_{n} \right) \otimes_{\Sigma_{k}} Y(k);\]
the unit is the symmetric sequence
\[ \bbone(n) =
\begin{cases}
  \emptyset, & n \neq 1 \\
  \bbone, & n = 1
\end{cases},\]
where $\bbone$ is the unit in $\mathbf{C}$. As first observed by
Kelly~\cite{KellyOpd}, an operad is then precisely an associative
algebra with respect to $\circ$: the multiplication map
$\mathbf{O} \circ \mathbf{O} \to \mathbf{O}$ is given by a family of
equivariant maps
\[ \mathbf{O}(k) \otimes \mathbf{O}(i_{1}) \otimes \cdots \otimes
\mathbf{O}(i_{k}) \to \mathbf{O}(i_{1} + \cdots + i_{k}),\]
supplying the operadic composition maps, and similarly the unit map
$\bbone \to \mathbf{O}$ corresponds to a unit
$\bbone \to \mathbf{O}(1)$.

In homotopical settings, this classical notion of operads has
a number of shortcomings, analogous to those afflicting topological or
simplicial categories when we want to work with them only up to
homotopy (\ie{} consider them as models for \icats{}). This motivates
the introduction of a fully homotopy-coherent version of operads,
known as \emph{$\infty$-operads}. Just as in the case of \icats{}, there
are several useful models for \iopds{}, including those of
Lurie~\cite{HA} (which is currently by far the best-developed),
Moerdijk--Weiss~\cite{MoerdijkWeiss},
Cisinki--Moerdijk~\cite{CisinskiMoerdijkDendSeg}, and
Barwick~\cite{BarwickOpCat}. These authors only consider \iopds{} in
spaces, but the formalism has recently been extended to cover \iopds{}
in other symmetric monoidal \icats{} in \cite{enropd}.

The goal of the present paper is to provide another point of view on
(enriched) \iopds{}, by extending to the higher-categorical setting
the description of operads as associative algebras in symmetric
sequences:
\begin{thm*}
  Let $\xFeq$ be the groupoid of finite sets and bijections. If
  $\mathcal{V}$ is a  symmetric
  monoidal \icat{} compatible with colimits indexed by
  \igpds{}\footnote{By this we mean that the underlying \icat{}
    $\mathcal{V}$ has colimits indexed by \igpds{}, and the tensor
    product preserves such colimits in each variable.}, then there exists a monoidal structure on the
  \icat{} $\Fun(\xFeq, \mathcal{V})$ of symmetric sequences
  such that associative algebras are $\mathcal{V}$-enriched
  \iopds{}. Moreover, the tensor product is described by the same
  formula as above.
\end{thm*}

More precisely, this gives a description of \iopds{} with a single
object. It is often convenient to consider the more general notion of
operads with many objects (also known as \emph{coloured operads} or
\emph{symmetric multicategories}), and the term \iopd{} typically
refers to the higher-categorical version of these generalized objects,
which also have a description as associative algebras:
For a set $S$, let \[\xFeq_{S} := \coprod_{n = 0}^{\infty} S^{\times
  n}_{h\Sigma_{n}}\] denote the groupoid with
objects lists $(s_{1},\ldots,s_{n})$ ($s_{i} \in S$) and with a
morphism $(s_{1},\ldots,s_{n}) \to (s'_{1},\ldots,s'_{m})$
given by a bijection
$\sigma \colon \{1,\ldots,n\} \isoto \{1,\ldots,m\}$ in
$\xFeq$ such that $s_{i} = s'_{\sigma(i)}$. Then a (symmetric) \emph{$S$-collection} (or
\emph{$S$-coloured symmetric sequence}) in $\mathbf{V}$ is a functor
$\xFeq_{S}\times S \to \mathbf{V}$. The category
$\Fun(\xFeq_{S}\times S, \mathbf{V})$ again has a composition
product $\circ$, given by a more complicated version of the formula we
gave above, such that an operad with $S$ as its set of objects is
precisely an associative algebra for this monoidal structure. Our work
also gives an \icatl{} version of this many-object composition
product.

More generally, we can describe operads with varying spaces of objects
as associative algebras in a \emph{double category}. We
will call a functor $\xFeq_{S}\times T \to \mathbf{V}$ an
\emph{$(S,T)$-collection} in $\mathbf{V}$. Then we can define a double category
$\COLL(\mathbf{V})$ as follows:
\begin{itemize}
\item Objects are sets, and vertical morphisms are maps of sets.
\item Horizontal morphisms from $S$ to $T$ are $(S,T)$-collections.
\item Composition of horizontal morphisms is given by a version of
  the composition product.
\end{itemize}
An associative algebra in $\COLL(\mathbf{V})$ consists of a set $S$
together with an associative algebra in the category of horizontal
endomorphisms of $S$ with composition as monoidal structure, \ie{} an
associative algebra in $S$-collections with the composition
product. Thus associative algebras are precisely operads, and moreover
morphisms of algebras in $\COLL(\mathbf{V})$ are precisely functors of
operads. We will produce an \icatl{} version of this structure:
\begin{thm*}
  For any symmetric monoidal \icat{} $\mathcal{V}$ compatible with
  colimits indexed by \igpds{}
  there is a double \icat{} $\COLL(\mathcal{V})$ such that
  $\Alg(\COLL(\mathcal{V}))$ is the \icat{} of $\mathcal{V}$-enriched
  \iopds{}.\footnote{In this paper we are considering \iopds{} as
    \emph{algebraic} objects, \ie{} we are not inverting the fully
    faithful and essentially surjective objects or imposing a
    completeness condition. In the terminology of
    Ayala--Francis~\cite{AyalaFrancisFlagged}, $\Alg(\COLL(\mathcal{V}))$ is the \icat{} of
    ``flagged $\mathcal{V}$-\iopds{}''.}
\end{thm*}
The double \icat{} $\COLL(\mathcal{V})$ admits the same description as
its analogue for ordinary categories, except with \igpds{} as
objects.

In a sequel to this paper \cite{opdalg} we apply this description of
\iopds{} to study \emph{algebras} over enriched \iopds{}. In addition,
we hope that it can serve as a starting point for a better
understanding of \emph{bar-cobar} (or \emph{Koszul}) \emph{duality}
for \iopds{}. Over a field of characteristic zero,
Koszul duality for dg-operads was introduced by Ginzburg and Kapranov
\cite{GinzburgKapranov}, and is by now well understood using
model-categorical methods (see \eg{}
\cite{GetzlerJones,FresseKoszul,FresseEn,LV,ValletteHtpyAlgs}).  As a
first step towards an \icatl{} approach to Koszul duality, here we
construct a bar-cobar adjunction between \iopds{} and
\icoopds{}. Following the approach proposed by Francis and
Gaitsgory~\cite{FrancisGaitsgory}, we obtain this as the bar-cobar
adjunction between associative algebras and coassociative coalgebras
(constructed in great generality in \cite{HA}*{\S 5.2.2}) applied to
our monoidal \icat{} of symmetric sequences. This seems likely to 
agree with existing constructions not only in chain complexes over a field of
characteristic $0$, but also in other settings such as spectra \cite{ChingBar,ChingHarper}, where it is
closely related to Goodwillie calculus \cite{ChingThesis},
as well as in $K(n)$-local
spectra, where bar-cobar duality has been constructed and applied
in work of Heuts~\cite{HeutsVn}.

\subsection{Overview of Results}\label{subsec:overview}
Let us now give a more detailed overview of the results of this
paper. The starting point for our construction is the
``coordinate-free'' definition of the composition product due to
Dwyer and Hess~\cite{DwyerHess}*{\S A.1}. They observe that, if
$\xF^{[1],\simeq}$ denotes the groupoid of morphisms of finite
sets and $\xF^{[2],\simeq}$ denotes the groupoid of composable
pairs of morphisms of finite sets, then:
\begin{itemize}
\item Symmetric sequences in $\Set$ are the same thing as symmetric
  monoidal functors $\xF^{[1],\simeq} \to \Set$, with respect to
  the disjoint union in $\xF^{[1],\simeq}$ and the cartesian
  product of sets.
\item Under this identification the composition product of $X$ and $Y$
  corresponds (by \cite{DwyerHess}*{Lemma A.4}) to the left Kan extension, along the functor $\xF^{[2],\simeq} \to \xF^{[1],\simeq}$ given by composition,
  of the restriction of $X \times Y$ from $\xF^{[1],\simeq}
  \times \xF^{[1],\simeq}$ to $\xF^{[2],\simeq}$. In other
  words,
  \[ (X \circ Y)(A \to C) \cong \colim_{(A \to B \to C) \in
      \xF^{[2],\simeq}_{(A \to C)}} X(A \to B) \times Y(B \to C).\]
  If $C \cong *$, then the groupoid $\xF^{[2],\simeq}_{(A \to *)}$ of
  factorizations of $A \to *$ is the groupoid of maps $A \to B$ and
  isomorphisms $B \isoto B'$ under $A$. An isomorphism class of such
  objects corresponds to a decomposition $|A| = i_{1}+\cdots+i_{k}$
  where $k = |B|$, with a division of $A$ into subsets of size
  $i_{j}$. This can be rewritten to recover the previous formula (with
  the division of $A$ corresponding to the product with $\Sigma_{n}$
  for a given partition of $n= |A|$).
\end{itemize}
After a slight reformulation this description is closely related to
Barwick's indexing category $\DF$ for \iopds{}, introduced in \cite{BarwickOpCat}. This is the
category with objects sequences $S_{0} \to S_{1} \to \cdots \to S_{n}$
of morphisms of finite sets, with a map
$(S_{0} \to \cdots \to S_{n}) \to (T_{0} \to \cdots \to T_{m})$ given
by a map $\phi \colon [n] \to [m]$ in $\simp$ and injective morphisms
$S_{i} \to T_{\phi(i)}$ such that the squares
\nolabelcsquare{S_i}{S_j}{T_{\phi(i)}}{T_{\phi(j)}} are cartesian. If
$(\DF)_{[n]}$ denotes the fibre at $[n]$ of the obvious projection
$\DF \to \simp$, then:
\begin{itemize}
\item Symmetric sequences in $\Set$ are the same thing as functors
  $X \colon  (\DF)_{[1]}^{\op} \to \Set$ such that for every object $S
  \to T$ the map
  \[ X(S \to T) \to \prod_{i \in T} X(S_{i} \to *),\]
  induced by the morphisms
  \[
  \begin{tikzcd}
    S_{i} \arrow{r} \arrow{d} \arrow[phantom]{dr}[very near
    start]{\lrcorner} & \{i\} \arrow{d} \\
    S \arrow{r} & T,
  \end{tikzcd}
  \]
  is an isomorphism.
\item Under this identification the composition product 
  of $X$ and $Y$
  corresponds to the left Kan extension, along the functor
  $(\DF)_{[2]}^{\op} \to (\DF)_{[1]}^{\op}$ corresponding to $d_{1}
  \colon [1] \to [2]$,
  of the restriction of $X \times Y$ along the functor
  $(\DF)_{[2]}^{\op} \to (\DF)_{[1]}^{\op} \times (\DF)_{[1]}^{\op}$
  corresponding to $(d_{2},d_{0})$. In other words,
  \[ (X \circ Y)(A \to C) \cong \colim_{(A' \to B \to C') \in
      ((\DF)_{[2]}^{\op})_{/(A \to C)}} X(A' \to B) \times Y(B \to C').\]
  This is equivalent to the previous description since the inclusion
  $\xF^{[2],\simeq}_{(A \to C)}  \to (\DF)_{[2],/(A \to C)}^{\op}$ is
  cofinal: given an object $\xi$ in the target, which is a diagram
  \[
    \begin{tikzcd}
      A \arrow{d} \arrow{rr} & & C \arrow{d} \\
      A' \arrow{r} & B \arrow{r} & C',
    \end{tikzcd}
  \]
  where the square is cartesian,
  the category $(\xF^{[2],\simeq}_{(A \to C)})_{\xi/}$ is a
  contractible groupoid with the single object given by the factorization
  $A \to B \times_{C'}C \to C$.
\end{itemize}
The projection $\DF \to \simp$ is a Grothendieck fibration, and the
corresponding functor $\Phi \colon \Dop \to \Cat$ is a double
category, in the sense that it satisfies the Segal condition
\[ \Phi_{n} \isoto \Phi_{1} \times_{\Phi_{0}} \cdots \times_{\Phi_{0}}
\Phi_{1}.\]
We will obtain the composition product by applying to this double
category a general construction of monoidal structures on functor
categories arising from certain double \icats{}. In fact, our
construction will produce a canonical double \icat{} of which this
monoidal \icat{} is a piece, with the full double \icat{} describing
\iopds{} with varying spaces of objects.

The construction of this double \icat{} can be seen a variation of the \emph{Day
  convolution}~\cite{DayConv} structure on functor
categories: If $\mathbf{C}$ is a small monoidal category and
$\mathbf{V}$ is a monoidal category compatible with colimits, then the functor
category $\Fun(\mathbf{C}, \mathbf{V})$ has a tensor product, given
for functors $F$ and $G$ as the left Kan extension along
$\otimes \colon \mathbf{C} \times \mathbf{C} \to \mathbf{C}$ of the
composite
\[ \mathbf{C} \times \mathbf{C} \xto{F \times G} \mathbf{V} \times
\mathbf{V} \xto{\otimes} \mathbf{V}.\]
This monoidal structure has the property that an associative algebra
in $\Fun(\mathbf{C}, \mathbf{V})$ is the same thing as a lax monoidal
functor $\mathbf{C} \to \mathbf{V}$; more generally, the Day
convolution has a universal property for algebras over non-symmetric
operads.

Day convolution (in the symmetric monoidal setting) was implemented in the \icatl{} context by
Glasman~\cite{GlasmanDay}.\footnote{More recently, Lurie has also given a more
general account~\cite{HA}*{\S 2.2.6}.} In this paper we extend this to
a construction of Day convolution for a class of double \icats{}:
\begin{thm}
  Suppose $\mathcal{M} \to \Dop$ is a suitable double \icat{}. Then
  there is a double \icat{} $\widehat{\mathcal{M}}_{\mathcal{S}}$ such
  that
  for any generalized non-symmetric \iopd{} $\mathcal{O}$ we have a natural
  equivalence\footnote{The right-hand side is the \icat{} of ``Segal
    $\mathcal{O} \times_{\Dop} \mathcal{M}$-spaces'', which are
    functors $\mathcal{O} \times_{\Dop} \mathcal{M} \to \mathcal{S}$
    satisfying certain Segal-type limit conditions; see
    \cref{defn:SegOb} for the precise definition.}
  \[ \Alg_{\mathcal{O}}(\widehat{\mathcal{M}}_{\mathcal{S}}) \simeq
    \Seg_{\mathcal{O} \times_{\Dop} \mathcal{M}}(\mathcal{S}).\]
  The objects of $\widehat{\mathcal{M}}_{\mathcal{S}}$ are functors
  $\mathcal{M}_{0} \to \mathcal{S}$ and the vertical morphisms are
  natural transformations of such functors. A horizontal morphism from
  $F$ to $G$ is a functor $\mathcal{M}_{1,F,G} \to \mathcal{S}$, where
  $\mathcal{M}_{1,F,G} \to \mathcal{M}_{1}$ is the left fibration for
  the composite functor
  \[ \mathcal{M}_{1} \xto{(d_{1,!}, d_{0,!})} \mathcal{M}_{0} \times
    \mathcal{M}_{0} \xto{F \times G} \mathcal{S}.\]
\end{thm}
This theorem summarizes the results of \S\ref{sec:dayconv}: We
construct these double \icats{} in \S\ref{sec:day} using an unfolding
construction introduced in \S\ref{sec:unfold}, and prove the universal
property in \S\ref{sec:dayunivp}. Note that the precise meaning of ``suitable'' we need is quite
restrictive. We also show in \S\ref{sec:daymon}
that we can extract from $\widehat{\mathcal{M}}_{\mathcal{S}}$ a
family of monoidal \icats{} and lax monoidal functors which suffices
to describe associative algebras in
$\widehat{\mathcal{M}}_{\mathcal{S}}$.
Moreover, we consider enriched versions of the theorem, with more
general targets than $\mathcal{S}$, in \S\ref{sec:enrdc}.

To obtain our double \icats{} we use results on
\icats{} of spans due to Barwick~\cite{BarwickMackey}, and
\S\ref{sec:spans} is devoted to a review of this work, with some
slight variations, together with a brief review of non-symmetric
\iopds{} and related structures. 

In \S\ref{sec:appl} we apply our results on Day convolution to
\iopds{}. In \S\ref{subsec:opd} we describe non-enriched \iopds{} as
associative algebras in a double \icats{} of collections in
$\mathcal{S}$, and in \S\ref{subsec:enropd} we extend this to a
description of enriched \iopds{}. More precisely, we obtain an equivalence between
associative algebras in a double \icat{} of collections and \iopds{} in the
sense of Barwick~\cite{BarwickOpCat}, as generalized to enriched
\iopds{} in \cite{enropd}. In \S\ref{subsec:bar} we then apply this
description of \iopds{} to obtain the bar-cobar adjunction between
\iopds{} and \icoopds{}.

As a warm-up to this description of \iopds{}, in \S\ref{subsec:enrcat}
we also consider an additional application of our Day convolution
construction, by showing that enriched \icats{} can be described as
associative algebras.

\subsection{Related Work}
There are at least two other approaches to constructing the
composition product on symmetric sequences $\infty$-categorically:

\subsubsection*{Composition Product from Free Presentably Symmetric
  Monoidal Categories} An alternative approach to defining the
composition product of $S$-coloured symmetric sequences in $\Set$
starts with the observation that
$\Fun(\coprod_{n = 0}^{\infty} S^{n}_{h\Sigma_{n}}, \Set)$ is the free
presentably symmetric monoidal category generated by $S$. If
$\mathbf{C}$ is a presentably symmetric monoidal category we therefore
have a natural equivalence
\[ \Fun(S, \mathbf{C}) \simeq \Fun^{L,\otimes}(\Fun(\coprod_{n = 0}^{\infty} S^{n}_{h\Sigma_{n}},
\Set), \mathbf{C}),\]
where the right-hand side denotes the category of colimit-preserving
symmetric monoidal functors. Taking $\mathbf{C}$ to be
\[\Fun(\xFeq_{S}, \Set) \cong \coprod_{n = 0}^{\infty}
  S^{n}_{h\Sigma_{n}}, \Set),\]
we get a natural equivalence
\[ \Fun(\xFeq_{S} \times S, \Set) \simeq \Fun\left(S,
  \Fun(\xFeq_{S}, \Set)\right)
\simeq \Fun^{L,\otimes}\left(\Fun(\xFeq_{S}, \Set), \Fun(\xFeq_{S}, \Set)\right).\]
Here the right-hand side has an obvious monoidal structure given by
composition of functors, and this corresponds under the equivalence to
the composition product of $S$-coloured symmetric sequences. This
construction is described in \cite{BaezDolanHDA3}*{\S 2.3}. The
one-object variant is much better known; it is attributed to Carboni
in the ``Author's Note'' for \cite{KellyOpd}, and it is also found in
Trimble's preprint \cite{TrimbleLie}. There is also an enriched
version of this construction, for (coloured) symmetric sequences in a
presentably symmetric monoidal category. More recently, this approach
has been further developed in
\cite{GambinoJoyal,FioreGambinoHylandWinskel} where it is shown to
arise from a 2-categorical construction that produces a 2-category of
operads with varying sets of objects (but with \emph{bimodules} of
operads as morphisms rather than functors).

In the \icatl{} setting, it is not hard to see that
$\Fun(\coprod_{n=0}^{\infty} X^{n}_{h \Sigma_{n}}, \mathcal{S})$ is again the
free presentably symmetric monoidal \icat{} generated by a space
$X$. One can thus take the same route to obtain a composition product
on $X$-coloured symmetric sequences in the \icat{} of spaces. In the
one-object case this approach (including its enriched variant) is
worked out in Brantner's thesis~\cite{BrantnerThesis}*{\S 4.1.2}. 
However, this approach has not yet been compared to any of the
established models for \iopds{}.

\subsubsection*{Polynomial Monads}
In \cite{polynomial} we show that \iopds{} with a fixed space of
objects $X$ are equivalent to analytic monads on the slice \icat{}
$\mathcal{S}_{/X}$. These analytic monads can be viewed as associative
algebras under composition in an \icat{} of analytic endofunctors of
$\mathcal{S}_{/X}$. The latter can be identified with $X$-coloured
symmetric sequences in $\mathcal{S}$, so this gives an alternative
description of \iopds{} as associative algebras for the composition
product. Compared to our approach here, this has a number of
advantages:
\begin{itemize}
\item it makes it clear that an \iopd{} can be recovered from its free
  algebra monad,
\item it clarifies the relation between \iopds{} and trees (because
  free analytic monads can be described in terms of trees).
\end{itemize}
It also seems likely that versions of polynomial monads in other
$\infty$-topoi can be used to define operad-like structures that occur
in equivariant and motivic homotopy theory. On the other hand,
polynomial monads do not seem to extend usefully to give a description
of \emph{enriched} \iopds{}.

\subsection{Acknowledgments}
I thank Ben Knudsen for suggesting that the ``coordinate-free''
definition of Dwyer and Hess should lead to an \icatl{} construction
of the composition product. I also thank Lukas Brantner, Hongyi Chu,
Joachim Kock, and Thomas Nikolaus for helpful discussions related to
this project, and Gijs Heuts for pointing out a serious misconception
about bar-cobar duality in the first versions of the paper. This paper
was substantially revised while the author was employed by the IBS
Center for Geometry and Physics in a position funded by grant
IBS-R003-D1 of the Institute for Basic Science of the Republic of
Korea.

\section{Background on Spans and Non-Symmetric
  $\infty$-Operads}\label{sec:spans}

In this section we first review non-symmetric \iopds{} and related
structures in \S\ref{subsec:nsopd}, and then recall some definitions
and results regarding spans from \cite{BarwickMackey}, with some minor
variations to get the generality we need in the next section.

\subsection{Review of Non-Symmetric $\infty$-Operads}\label{subsec:nsopd}
For the reader's convenience, we will briefly review some definitions
and results related to non-symmetric \iopds{} that we will make
frequent use of below. For more details, as well as motivation, we
refer the reader to \cite{enr,nmorita,HA}.
\begin{notation}
  $\simp$ denotes the standard simplicial indexing category, \ie{} the
  category of ordered sets $[n] = \{0,1,\ldots,n\}$ and
  order-preserving maps. We say a map $\phi \colon [n] \to [m]$ is
  \emph{inert} if it is the inclusion of a subinterval, \ie{} $\phi(i)
  = \phi(0)+i$ for all $i$, and \emph{active} if it preserves the
  end-points, \ie{} $\phi(0) = 0$, $\phi(n) = m$. The active and inert
  maps form a factorization system on $\simp$ --- every morphism
  factors uniquely as an active map followed by an inert map. We write
  $\simp^{\txt{int}}$ for the subcategory of $\simp$ containing only
  the inert maps, and $\simp^{\txt{el}}$ for the full subcategory of
  $\simp^{\txt{int}}$ containing only the objects $[0]$ and $[1]$;
  we also use the notation \[\simp^{\txt{el}}_{/[n]} :=
    \simp^{\txt{el}} \times_{\simp^{\txt{int}}}
    \simp^{\txt{int}}_{/[n]}\]
  for the category of inert maps to $[n]$ from $[0]$ and $[1]$
\end{notation}

\begin{defn}
  For $0 \leq i \leq j \leq n$ we write $\rho_{ij}$ for the inclusion
  $[j-i] \cong \{i,i+1,\ldots,j\} \hookrightarrow [n]$. If
  $\mathcal{C}$ is an \icat{} with products, then an \emph{associative
    monoid} in $\mathcal{C}$ is a functor
  $A \colon \Dop \to \mathcal{C}$ such that for every $n$ the map
  $A_{n} \to \prod_{i = 1}^{n} A_{1}$, induced by the maps
  $\rho_{(i-1)i}\colon [1] \to [n]$, is an equivalence.
\end{defn}

\begin{defn}
If $\mathcal{C}$ is an
  \icat{} with finite limits, then a \emph{category object} in
  $\mathcal{C}$ is a functor $X \colon \Dop \to \mathcal{C}$ such that
  for every $n$ the map
  \[ X_{n} \to X_{1} \times_{X_{0}} \cdots \times_{X_{0}} X_{1}\]
  induced by the maps $\rho_{(i-1)i}$ and $\rho_{ii}$, is an equivalence.
\end{defn}

\begin{remark}
  A category object in the \icat{} $\mathcal{S}$ of spaces is a
  \emph{Segal space} in the sense of Rezk~\cite{RezkCSS}. The
  structure of a Segal space describes precisely the ``algebraic''
  structure of an \icat{}, \ie{} a homotopy-coherent composition with
  identities, but to capture the correct equivalences between \icats{}
  one must invert the fully faithful and essentially surjective maps
  between Segal spaces, or equivalently restrict to the full
  subcategory of \emph{complete} Segal spaces.
\end{remark}

\begin{defn}
  A \emph{monoidal \icat{}} is a cocartesian fibration $\mathcal{C}^{\otimes}
  \to \Dop$ such that the corresponding functor $\Dop \to \CatI$ is an
  associative monoid. Similarly, a \emph{double \icat{}} is a
  cocartesian fibration $\mathcal{M} \to \Dop$ such that the
  corresponding functor $\Dop \to \CatI$ is a category object.
\end{defn}

\begin{notation}
  We will use the following terminology to describe double \icats{}
  $\mathcal{M} \to \Dop$:
  \begin{itemize}
  \item an object of $\mathcal{M}_{0}$ is an \emph{object} of the
    double \icat{},
  \item a morphism of $\mathcal{M}_{1}$ is a \emph{vertical morphism}
    of the double \icat{}
  \item an object of $\mathcal{M}_{1}$ is a \emph{horizontal
      morphism},
  \item a morphism in $\mathcal{M}_{1}$ is a \emph{square},
  \item composition of vertical morphisms is composition in the
    \icat{} $\mathcal{M}_{0}$,
  \item vertical composition of squares is composition in $\mathcal{M}_{1}$
  \item composition of horizontal morphisms, as well as horizontal
    composition of squares, is given by the functor
    \[ \mathcal{M}_{1} \times_{\mathcal{M}_{0}} \mathcal{M}_{1}
      \isofrom \mathcal{M}_{2} \xto{d_{1,!}} \mathcal{M}_{1}.\]
  \end{itemize}
\end{notation}

\begin{notation}\label{not:odotcomp}
  Given objects $X,Y \in \mathcal{M}_{0}$, we write $\mathcal{M}(X,Y)$
  for the fibre of $\mathcal{M}_{1} \xto{(d_{1,!},d_{0,!})}
  \mathcal{M}_{0} \times \mathcal{M}_{0}$ at $(X,Y)$, and call this
  the \icat{} of horizontal morphisms from $X$ to $Y$. Given its
  simplicial origin, it is usually less confusing to write composition
  of horizontal morphisms in the non-standard order, and we denote it
  \[ \blank \odot_{Y} \blank \colon \mathcal{M}(X,Y) \times \mathcal{M}(Y,Z)
    \to \mathcal{M}(X,Z).\]
  We will also write $\bbone_{X} \in \mathcal{M}(X,X)$ for the
  horizontal identity.
\end{notation}

\begin{defn}\label{defn:gnsiopd}
  A \emph{generalized non-symmetric \iopd{}} is a functor $p \colon \mathcal{O} \to \Dop$
  such that
  \begin{enumerate}[(i)]
  \item $\mathcal{O}$ has $p$-cocartesian morphisms over all inert
    maps in $\Dop$,
  \item for every $n$ the functor
    $\mathcal{O}_{[n]} \to \mathcal{O}_{[1]}
    \times_{\mathcal{O}_{[0]}} \cdots \times_{\mathcal{O}_{[0]}}
    \mathcal{O}_{[1]}$,
    induced by the cocartesian morphisms over the maps $\rho_{(i-1)i}$ and
    $\rho_{ii}$, is an equivalence,
  \item for every $X \in \mathcal{O}_{[n]}$, $Y \in \mathcal{O}_{[m]}$
    and $\phi \colon [m] \to [n]$ in $\simp$, the map
    \[ \Map^{\phi}_{\mathcal{O}}(X, Y) \to
    \Map^{\rho_{01}\phi}_{\mathcal{O}}(\rho_{01,!}X, Y)
    \times_{\Map^{\rho_{11}\phi}(\rho_{11,!}X, Y)} \cdots
    \times_{\Map^{\rho_{(n-1)(n-1)}\phi}(\rho_{(n-1)(n-1),!}X, Y)}
    \Map^{\rho_{(n-1)n}\phi}_{\mathcal{O}}(\rho_{(n-1)n,!}X, Y)\]
    is an equivalence, where $X \to \rho_{ij,!}X$ is the
    $p$-cocartesian morphism over the inert map $\rho_{ij}$
 and $\Map^{\phi}_{\mathcal{O}}(X,Y)$
    denotes the fibre at $\phi$ of the map $\Map_{\mathcal{O}}(X,Y)
    \to \Map_{\Dop}([n],[m])$.
  \end{enumerate}
  We refer to the cocartesian morphisms over inert morphisms in $\Dop$
  as \emph{inert} morphisms in $\mathcal{O}$.
  A morphism of \gnsiopds{} is a functor over $\Dop$ that preserves
  inert morphisms; we also refer
  to a morphism of \gnsiopds{} $\mathcal{O} \to \mathcal{P}$ as an
  \emph{$\mathcal{O}$-algebra} in $\mathcal{P}$ and write
  $\Alg_{\mathcal{O}}(\mathcal{P})$ for the \icat{} of these. More
  generally, if $\mathcal{O}$ and $\mathcal{P}$ are \gnsiopds{} over
  $\mathcal{Q}$ we write $\Alg_{\mathcal{O}/\mathcal{Q}}(\mathcal{P})$
  for the analogous \icat{} of commutative triangles of morphisms of
  \gnsiopds{}
  \[
    \begin{tikzcd}
      \mathcal{O} \arrow{rr} \arrow{dr} & & \mathcal{P} \arrow{dl} \\
       & \mathcal{Q}.
    \end{tikzcd}
    \]
\end{defn}

\begin{defn}
  A \emph{\nsiopd{}} is a \gnsiopd{} $\mathcal{O}$ such that
  $\mathcal{O}_{[0]} \simeq *$.
\end{defn}

\begin{notation}
  If $\mathcal{O}$ is a \gnsiopd{} and $x$ is an object of
  $\mathcal{O}_{n}$, we will often write $x \to x_{ij}$ for the
  cocartesian morphism over $\rho_{ij}$ for $0 \leq i \leq j \leq n$.
\end{notation}

\begin{lemma}\label{lem:subgnsiopd}
  Suppose $\mathcal{O}$ is a \gnsiopd{}. Let $\mathcal{O}'_{0}$ be a
  full subcategory of $\mathcal{O}_{0}$ and $\mathcal{O}'_{1}$ be a
  full subcategory of $\mathcal{O}_{1}$ such that for $x \in
  \mathcal{O}'_{1}$ the objects $x_{00}$ and $x_{11}$ are in
  $\mathcal{O}'_{0}$. If $\mathcal{O}'$ denotes the full subcategory
  of $\mathcal{O}$ spanned by objects $x$ such that $x_{ii} \in
  \mathcal{O}'_{0}$ and $x_{(i-1)i}$ is in $\mathcal{O}'_{1}$ for all
  $i$, then
  \begin{enumerate}[(i)]
  \item  $\mathcal{O}'$ is also a \gnsiopd{},
  \item the inclusion
  $j \colon \mathcal{O}' \hookrightarrow \mathcal{O}$ preserves inert
  morphisms,
\item for any \gnsiopd{} $\mathcal{P}$ the functor
  \[ j_{*} \colon \Alg_{\mathcal{P}}(\mathcal{O}') \to
    \Alg_{\mathcal{P}}(\mathcal{O}) \]
  given by composition with $i$ is fully faithful, with image the
  algebras $\mathcal{P} \to \mathcal{O}$ such that the restrictions
  $\mathcal{P}_{i} \to \mathcal{O}_{i}$ factor through
  $\mathcal{O}'_{i}$ for $i = 0,1$.
  \end{enumerate}
\end{lemma}
\begin{proof}
 By definition, we have pullback
  squares
  \[
    \begin{tikzcd}
      \mathcal{O}'_{n} \arrow{d} \arrow{r} & \mathcal{O}_{n}
      \arrow{d}{\sim} \\
      \mathcal{O}'_{1} \times_{\mathcal{O}'_{0}} \cdots
      \times_{\mathcal{O}'_{0}} \mathcal{O}'_{1} \arrow{r} &
      \mathcal{O}_{1} \times_{\mathcal{O}_{0}} \cdots
      \times_{\mathcal{O}_{0}} \mathcal{O}_{1},
    \end{tikzcd}
  \]
  so that the left vertical map is an equivalence. Condition (iii) in
  \cref{defn:gnsiopd} is also immediate from $\mathcal{O}'$ being a
  full subcategory. If $x$ is in $\mathcal{O}'$ and $x \to y$ is an
  inert morphism in $\mathcal{O}$, then by the definition of
  $\mathcal{O}'$ the object $y$ is also in $\mathcal{O}'$,
  so $\mathcal{O}'$ inherits cocartesian morphisms over inert
  morphisms from $\mathcal{O}$. This proves (i) and (ii), and (iii) is
  immediate from the definition of $\Alg_{\mathcal{P}}(\mathcal{O}')$
  as a full subcategory of $\Fun_{/\Dop}(\mathcal{P}, \mathcal{O}')$.
\end{proof}

\begin{defn}  
  If $\mathcal{C}$ is an \icat{} with finite products and
  $\mathcal{O}$ is a \gnsiopd{}, then an \emph{$\mathcal{O}$-monoid} in
  $\mathcal{C}$ is a functor $M \colon \mathcal{O} \to \mathcal{C}$ such that
  for every $x \in \mathcal{O}_{[n]}$, the map $M(x) \to
  \prod_{i=1}^{m} M(x_{(i-1)i})$ induced by the cocartesian morphisms
  $x \to x_{(i-1)i}$ over $\rho_{(i-1)i}$, is an equivalence. We write
  $\txt{Mon}_{\mathcal{O}}(\mathcal{C})$ for the \icat{} of
  $\mathcal{O}$-monoids in $\mathcal{C}$, a full subcategory of
  $\Fun(\mathcal{O}, \mathcal{C})$.
\end{defn}

\begin{defn}
  Let $\mathcal{O}$ be a \gnsiopd{}. An \emph{$\mathcal{O}$-monoidal
    \icat{}} is a cocartesian fibration $\mathcal{U}^{\otimes} \to
  \mathcal{O}$ such that the corresponding functor $\mathcal{O} \to
  \CatI$ is an $\mathcal{O}$-monoid; for $X \in \mathcal{O}_{[1]}$ we
  often write $\mathcal{U}_{X}$ for the fibre of
  $\mathcal{U}^{\otimes}$ at $X$. Note that the composite
  $\mathcal{U}^{\otimes} \to \mathcal{O} \to \Dop$ is again a
  \gnsiopd{} (and a double \icat{} if $\mathcal{O}$ is one). We call a
  morphism of \gnsiopds{} over $\mathcal{O}$ between
  $\mathcal{O}$-monoidal \icats{} a \emph{lax $\mathcal{O}$-monoidal
    functor}, and say that it is \emph{$\mathcal{O}$-monoidal} if it
  preserves all cocartesian morphisms over $\mathcal{O}$.
\end{defn}

\begin{defn}
  If $\mathcal{V}^{\otimes} \to \mathcal{O}$ is an $\mathcal{O}$-monoidal \icat{}, we write
  $\mathcal{V}_{\otimes} \to \mathcal{O}^{\op}$ for the corresponding cartesian
  fibration. Then
  $\mathcal{V}^{\op,\otimes} := (\mathcal{V}_{\otimes})^{\op} \to
  \mathcal{O}$ is again an $\mathcal{O}$-monoidal \icat{}; this
  describes the $\mathcal{O}$-monoidal 
  structure on $\mathcal{V}_{X}^{\op}$ ($X \in \mathcal{O}_{[1]}$)
  given by the same operations as
  those on $\mathcal{V}_{X}$.
\end{defn}

\begin{propn}\label{propn:algismon}
  If $\mathcal{M}$ is a \gnsiopd{} and $\mathcal{C}$ is an \icat{}
  with products, then there is a natural equivalence
  \[ \Alg_{\mathcal{M}}(\mathcal{C}) \simeq \Mon_{\mathcal{M}}(\mathcal{C}).\]
\end{propn}
\begin{proof}
  This is a special case of \cite{patterns2}*{Proposition 5.1} (which
  generalizes the version for symmetric \iopds{}, \cite{HA}*{Proposition 2.4.1.7}).
\end{proof}

The \icatl{} analogue of Day convolution was first constructed by
Glasman~\cite{GlasmanDay} for symmetric monoidal \icats{}. It was
generalized by Lurie~\cite{HA}*{\S 2.2.6} to $\mathcal{O}$-monoidal
\icats{} where $\mathcal{O}$ is a (symmetric) \iopd{} and further
extended by Hinich to \emph{flat} \iopds{} \cite{HinichYoneda}. The
following is a special case of another generalization, proved in \cite{patterns2}:
\begin{propn}\label{propn:dayconvo}
  Let $\mathcal{O}$ be a \gnsiopd{} and $\mathcal{U}^{\otimes} \to
  \mathcal{O}$ an $\mathcal{O}$-monoidal \icat{}. There exists an
  $\mathcal{O}$-monoidal \icat{} $\mathcal{U}_{\mathcal{S}}^{\otimes}
  \to \mathcal{O}$, natural with respect to $\mathcal{O}$-monoidal
  functors, such that for $X \in \mathcal{O}_{[1]}$ we have 
  $(\mathcal{U}_{\mathcal{S}}^{\otimes})_{X} \simeq
  \Fun(\mathcal{U}^{\otimes}_{X}, \mathcal{S})$ and with the universal
  property that for every \gnsiopd{} $\mathcal{P}$ over $\mathcal{O}$
  we have a natural equivalence
  \[
    \Alg_{\mathcal{P}/\mathcal{O}}(\mathcal{U}_{\mathcal{S}}^{\otimes})
    \simeq \Alg_{\mathcal{P} \times_{\mathcal{O}}
      \mathcal{U}^{\otimes}}(\mathcal{S}) \simeq \Mon_{\mathcal{P}
      \times_{\mathcal{O}} \mathcal{U}^{\otimes}}(\mathcal{S}).\]
  Moreover, if $\mathcal{P}(\mathcal{U})^{\otimes}:=
  \mathcal{U}^{\op,\otimes}_{\mathcal{S}}$ then there is a fully
  faithful $\mathcal{O}$-monoidal functor
  \[\mathcal{U}^{\otimes} \hookrightarrow
    \mathcal{P}(\mathcal{U})^{\otimes},\]
  given over $X \in \mathcal{O}_{[1]}$ by the Yoneda embedding
  $\mathcal{U}^{\otimes}_{X} \hookrightarrow
  \mathcal{P}(\mathcal{U}^{\otimes}_{X})$.
\end{propn}
\begin{proof}
  This is a special case of \cite{patterns2}*{Proposition 6.16 and
    Corollary 6.21}.
\end{proof}

\begin{remark}\label{rmk:dayconv1}
  The functor $\mathcal{U}^{\otimes}_{1} \to \mathcal{O}_{1}$ is a
  cocartesian fibration. Let $U \colon \mathcal{O}_{1} \to \CatI$
  denote the corresponding functor. Then $\mathcal{U}^{\otimes}_{\mathcal{S},1}
  \to \mathcal{O}_{1}$ is the cartesian fibration for the functor
  $\Fun(U,\mathcal{S}) \colon \mathcal{O}_{1}^{\op} \to \CatI$ defined
  by composition, or
  equivalently the cocartesian fibration for the induced functor given by
  the left adjoints, \ie{} left Kan extensions along the functors
  $U(f)$ for $f$ in $\mathcal{U}^{\otimes}_{1}$.
\end{remark}

\begin{defn}\label{defn:SegOb}
  Let $\mathcal{C}$ be an \icat{} with pullbacks and $\mathcal{O}$ a
  \gnsiopd{}. A \emph{Segal $\mathcal{O}$-object} in $\mathcal{C}$ is
  a functor $F \colon \mathcal{O} \to \mathcal{C}$ such that for every
  object $X \in \mathcal{O}$ over $[n] \in \simp$, the morphism
  \[ F(X) \to F(X_{01}) \times_{F(X_{11})} \cdots \times_{F(X_{(n-1)(n-1)})}
    F(X_{(n-1)n}) \]
  is an equivalence. We write
  $\Seg_{\mathcal{O}}(\mathcal{C})$ for the full subcategory of
  $\Fun(\mathcal{O}, \mathcal{C})$ spanned by the Segal
  $\mathcal{O}$-objects.
\end{defn}

\begin{propn}\label{propn:Segfib}
  Let $\mathcal{O}$ be a \gnsiopd{}. The restriction functor
  \[ \Seg_{\mathcal{O}}(\mathcal{S}) \to \Fun(\mathcal{O}_{0},
    \mathcal{S}) \]
  is a cartesian fibration, and the fibre at $\Xi \colon
  \mathcal{O}_{0}\to \mathcal{S}$ is equivalent to
  $\Mon_{\mathcal{O}_{\Xi}}(\mathcal{S})$ where $\mathcal{O}_{\Xi} \to
  \mathcal{O}$ is the left fibration for the functor $\mathcal{O} \to
  \mathcal{S}$ obtained as the right Kan extension of $\Xi$ along the
  inclusion $\mathcal{O}_{0} \hookrightarrow \mathcal{O}$.
\end{propn}
\begin{proof}
  As \cite{enrcomp}*{Theorem 7.5}.
\end{proof}

\begin{defn}
  Let $\mathcal{O}$ be a \gnsiopd{} and let $\mathcal{U}^{\otimes}$ be
  an $\mathcal{O}$-monoidal \icat{}. We write
  \[\Algd_{\mathcal{O}}(\mathcal{U}) \to
    \Fun(\mathcal{O}_{0},\mathcal{S})\]
  for the cartesian fibration corresponding to the functor $X \mapsto
  \Alg_{\mathcal{O}_{X}/\mathcal{O}}(\mathcal{U})$ and refer to its
  objects as \emph{$\mathcal{O}$-algebroids} in $\mathcal{U}$. 
\end{defn}

\begin{ex}\label{ex:enrcat}
  $\Dop$-algebroids in a monoidal \icat{} $\mathcal{V}$ are algebras
  in $\mathcal{V}$ for the family $\Dop_{X}$ ($X \in \mathcal{S}$) of
  \gnsiopds{}. These were called \emph{categorical algebras} in
  \cite{enr}, where they were used to model \icats{} enriched in $\mathcal{V}$.
\end{ex}

\begin{remark}
  \cref{propn:Segfib} and \cref{propn:algismon} identify
  $\Algd_{\mathcal{O}}(\mathcal{S})$ with
  $\Seg_{\mathcal{O}}(\mathcal{S})$. If $\mathcal{U}^{\otimes}$ is a
  small $\mathcal{O}$-monoidal \icat{} then the natural equivalence of
  \cref{propn:dayconvo} gives an equivalence
  \[
    \Alg_{\mathcal{O}_{\Xi}/\mathcal{O}}(\mathcal{U}^{\otimes}_{\mathcal{S}})
    \simeq \Alg_{\mathcal{U}^{\otimes}_{\Xi}}(\mathcal{S}),
  \]
  natural in $\Xi$, and so an equivalence    
  \[ \Algd_{\mathcal{O}}(\mathcal{U}^{\otimes}_{\mathcal{S}}) \simeq
    \Algd_{\mathcal{U}^{\otimes}}(\mathcal{S}) \simeq
    \Seg_{\mathcal{U}^{\otimes}}(\mathcal{S}).\]
  Combined with the $\mathcal{O}$-monoidal Yoneda embedding, we get:
\end{remark}
\begin{cor}
  Let $\mathcal{O}$ be a \gnsiopd{} and $\mathcal{U}^{\otimes}$ a
  small $\mathcal{O}$-monoidal \icat{}. Then there is a fully faithful functor
  \[ \Algd_{\mathcal{O}}(\mathcal{U}) \hookrightarrow
    \Seg_{\mathcal{U}^{\op,\otimes}}(\mathcal{S})\]
  with image those Segal $\mathcal{U}^{\op,\otimes}$-spaces $\Phi$ such
  that for every $x \in \mathcal{O}_{[1]}$, $p \in \Phi(x_{00})$, and $q
  \in \Phi(x_{11})$ the presheaf 
  \[ \Phi_{x,p,q} \colon (\mathcal{U}^{\otimes}_{x})^{\op} \simeq
    \mathcal{U}^{\op,\otimes}_{x} \xto{\Phi} \mathcal{S}_{/\Phi(x_{00})
      \times \Phi(x_{11})} \xto{(\blank)_{(p,q)}} \mathcal{S},\]
  obtained by taking fibres at $(p,q)$, is representable.
\end{cor}

\begin{defn}
  Let $K$ be a collection of \icats{}. Following \cite{HA}*{Definition
    3.1.1.18} we say that an
  $\mathcal{O}$-monoidal \icat{} $\mathcal{V}^{\otimes}$ is
  \emph{compatible with $K$-colimits} if
  \begin{itemize}
  \item the \icat{} $\mathcal{V}_{X}$ has $K$-colimits for every
    object $X \in \mathcal{O}_{1}$,
  \item for every active morphism $f \colon X \to Y$ in $\mathcal{O}$
    with $X \in\mathcal{O}_{n}$ and $Y \in \mathcal{O}_{1}$, the
    functor
    \[ \prod_{i=1}^{n} \mathcal{V}_{X_{(i-1)i}} \simeq
      \mathcal{V}^{\otimes}_{X} \xto{f_{!}} \mathcal{V}_{Y},\] induced
    by the cocartesian morphisms over $f$, preserves $K$-colimits in
    each variable.
  \end{itemize}
\end{defn}

\begin{lemma}\label{lem:loccoc}
  Let $\pi \colon \mathcal{O} \to \Dop$ be a \gnsiopd{}.
  \begin{enumerate}[(i)]
  \item If for every active morphism $\phi \colon [1] \to [n]$ in
    $\simp$ and every $X \in \mathcal{O}_{n}$, there is a locally
    $\pi$-cocartesian morphism $X \to \phi_{!}X$ in $\mathcal{O}$,
    then $\pi$ is a locally cocartesian morphism.
  \item If in addition for every active map $\phi \colon [2] \to [n]$
    and $X \in \mathcal{O}_{n}$, the canonical map
    \[(\phi d_{1})_{!}X \to d_{1,!}\phi_{!}X\] is an equivalence, then
    $\pi$ is a cocartesian fibration.
  \end{enumerate}
\end{lemma}
\begin{proof}
  We first prove that $\mathcal{O}$ has locally cocartesian morphisms
  over any active map $\alpha \colon [n] \to [m]$ in $\simp$. Given $x
  \in \mathcal{O}_{m}$ and $y \in \mathcal{O}_{n}$, we have
  \[ \Map_{\mathcal{O}}^{\alpha}(x, y) \simeq \lim_{\rho_{ij} \in
      \simp^{\txt{el},\op}_{/[n]}}
    \Map_{\mathcal{O}}^{\alpha_{ij}}(x_{\alpha(i)\alpha(j)}, y_{ij})\]
  where $\alpha_{ij}$ is the active part of $\alpha \circ
  \rho_{ij}$. By assumption we have locally cocartesian morphisms
  $x_{\alpha(i)\alpha(j)} \to \alpha_{ij,!}x_{\alpha(i)\alpha(j)}$ (if
  $i = j$ this is just the identity), so we can rewrite this as
  \[ \lim_{\rho_{ij} \in \simp^{\txt{el},\op}_{/[n]}}
    \Map_{\mathcal{O}_{j-i}}(\alpha_{ij,!}x_{\alpha(i)\alpha(j)},
    y_{ij}) \simeq \Map_{\mathcal{O}_{n}}(\alpha_{!}x, y),\] where
  $\alpha_{!}x$ is the object of
  $\mathcal{O}_{n} \simeq \lim_{\rho_{ij} \in
    \simp^{\txt{el},\op}_{/[n]}} \mathcal{O}_{j-i}$ corresponding to
  the family of objects $\alpha_{ij,!}x_{\alpha(i)\alpha(j)}$. Thus we
  have a locally cocartesian morphism $x \to \alpha_{!}x$.

  Next, suppose $\phi \colon [n] \to [m]$ is an arbitrary map in $\simp$, and let
  $[n] \xto{\alpha} [k] \xto{\iota} [m]$ be its active-inert
  factorization. Then for $x \in \mathcal{O}_{m}$ the composite $x \to
  \iota_{!}x \to \alpha_{!}\iota_{!}x$ is locally cocartesian over
  $\phi$, where the first map is cocartesian over $\iota$ and the
  second is locally cocartesian over $\alpha$: for $y \in
  \mathcal{O}_{n}$ we have an
  equivalence
  \[ \Map_{\mathcal{O}}^{\phi}(x,y) \simeq
    \Map^{\alpha}_{\mathcal{O}}(\iota_{!}x, y) \simeq
    \Map_{\mathcal{O}_{n}}(\alpha_{!}\iota_{!}x, y),\]
  since $x \to \iota_{!}x$ is cocartesian. This shows that
  $\mathcal{O} \to \Dop$ is a locally cocartesian fibration.

  Before we prove part (ii), we make a further observation
  in the general case: Suppose
  $\alpha \colon [n] \to [m]$ is active, $\iota \colon [l] \to [n]$
  is inert, $x$ is an object of $\mathcal{O}_{m}$, $x \to \alpha_{!}x$
  is locally cocartesian, and $\alpha_{!}x \to \iota_{!}\alpha_{!}x$
  is cocartesian. Then it follows from the decomposition above of
  $\alpha_{!}$ in terms of locally cocartesian morphisms over the
  unique active maps $[1] \to
  [n]$ that
  $x \to \iota_{!}\alpha_{!}x$ is locally
  cocartesian over $\phi := \alpha \iota$.

  It remains to prove (ii), for which we have to check that the
  assumption implies that locally
  cocartesian morphisms over active maps compose, \ie{} for active
  morphisms
  \[ [m] \xto{\alpha} [n] \xto{\beta} [k] \]
  the natural map $(\beta\alpha)_{!}X \to \alpha_{!}\beta_{!}X$ is an
  equivalence for $X \in \mathcal{O}_{k}$. Using the
  decomposition of locally cocartesian morphisms above we can immediately reduce
  to the case where $m = 1$. Now if $\alpha$ is surjective, we must
  have $n = 0$ or $1$; if $n = 0$ then $\beta = \id_{[0]}$, while if
  $n = 1$ then $\alpha = \id_{[1]}$ --- in either case the claim is
  trivially true. We can therefore assume that $\alpha$ is not
  surjective, in which case we can find a factorization of $\alpha$ as
  \[ [1] \xto{d_{1}} [2] \xto{\alpha'} [n] \]
  where $\alpha'(1) \neq \alpha'(0),\alpha'(2)$;
  using this factorization we get for $X \in \mathcal{O}_{k}$ a
  commutative square
  \[
    \begin{tikzcd}
      (\beta\alpha' d_{1})_{!}X \arrow{r} \arrow{d} &
      d_{1,!}(\beta\alpha')_{!}X \arrow{d} \\
      (\alpha' d_{1})_{!}\beta_{!}X \arrow{r} & d_{1,!}\alpha'_{!}\beta_{!}X.
    \end{tikzcd}
  \]
  Here our assumption guarantees the horizontal maps are equivalences,
  and we want to show the left vertical map is an equivalence.  It
  thus suffices to show the right vertical map is an equivalence, for
  which it's enough to prove $(\beta\alpha')_{!}X \to
  \alpha'_{!}\beta_{!}X$ is an equivalence since $d_{1,!}$ is a
  functor. Our assumption on $\alpha'(1)$ means this decomposes as a
  pair of maps
  \[ [1] = \{0,1\} \to \{\alpha'(0),\ldots,\alpha'(1)\} \to
    \{\beta\alpha'(0),\ldots, \beta\alpha'(1)\}\]
  and similarly with $\{1,2\}$, where
  \[ \{\alpha'(0),\ldots, \alpha'(1)\},
    \{\alpha'(1),\ldots, \alpha'(2)\} < n. \]
  This means we can reduce to our assumption by inducting on $n$.
  Combined with our previous
  observations we have then shown that locally cocartesian morphisms
  compose in general, since it holds for all combinations of active
  and inert maps. Thus $\pi$ is a cocartesian fibration, as required.
\end{proof}

\subsection{$\infty$-Categories of Spans}
For compatibility with \cite{BarwickMackey} we will work with \icats{}
as \emph{quasicategories}, \ie{} simplicial sets satisfying the
horn-filling condition for inner horns, in this and the next subsections.

\begin{defn}
  Let $\epsilon \colon \simp \to \simp$ be the functor
  $[n] \mapsto [n] \star [n]^{\op}$. This induces a functor
  $\epsilon^{*} \colon \sSet \to \sSet$ given by composition with
  $\epsilon$; this functor is the \emph{edgewise subdivision} of
  simplicial sets. If $\mathcal{C}$ is an \icat{}, we will write
  $\Tw^{r} \mathcal{C} := \epsilon^{*}\mathcal{C}$ and refer to this as
  the \emph{twisted arrow \icat{}} of $\mathcal{C}$.
\end{defn}

\begin{remark}
  By \cite{HA}*{Proposition 5.2.1.3} the simplicial set
  $\Tw^{r} \mathcal{C}$ is an \icat{} if $\mathcal{C}$ is one, and the
  projection $\Tw^{r} \mathcal{C} \to \mathcal{C} \times
  \mathcal{C}^{\op}$ (induced by the inclusions $[n], [n]^{\op} \to [n]
  \star [n]^{\op}$) is a right fibration.
\end{remark}

\begin{remark}
  If $\mathbf{C}$ is an ordinary category, then it is easy to see that
  $\Tw^{r} \mathbf{C}$ can be identified with the \emph{twisted arrow
    category} of $\mathbf{C}$. This has morphisms $c \to d$ in
  $\mathbf{C}$ as objects, and diagrams
  \[
  \begin{tikzcd}
    c \arrow{r} \arrow{d} & d \\
    c' \arrow{r} & d' \arrow{u}
  \end{tikzcd}
  \]
  as morphisms from $c \to d$ to $c' \to d'$, with composition induced
  from composition in $\mathbf{C}$. Unwinding the definition of $\Tw^{r}
  \mathcal{C}$ for $\mathcal{C}$ an \icat{}, we see that its objects and morphisms
  admit the same description in terms of $\mathcal{C}$.
\end{remark}

\begin{ex}
  The twisted arrow category $\Tw^{r}(\Delta^{n})$ is the poset of pairs
  $(i,j)$ with $0 \leq i \leq j \leq n$ where $(i,j) \leq (i',j')$ if
  $i \leq i'$, $j' \leq j$.
\end{ex}

\begin{warning}
  There are two possible conventions for the definition of
  $\Tw^{r} \mathcal{C}$: Instead of the definition we have given we could
  instead consider $[n] \mapsto [n]^{\op} \star [n]$; let us call the
  resulting simplicial set $\Tw^{\ell} \mathcal{C}$ --- this is the
  definition of the twisted arrow \icat{} used in \cite{BarwickMackey}
  (there called $\widetilde{\mathcal{O}}(\mathcal{C})$). We clearly
  have $\Tw^{r} \mathcal{C} \cong (\Tw^{\ell} \mathcal{C})^{\op}$, which
  explains why op's appear in different places here compared to
  \cite{BarwickMackey}.
\end{warning}

\begin{defn}
  The functor $\epsilon^{*}$ has a right adjoint
  $\epsilon_{*} \colon \sSet \to \sSet$, given by right Kan
  extension. Explicitly, $\epsilon_{*}X$ is determined by
  $\Hom(\Delta^{n}, \epsilon_{*}X) \cong \Hom(\Tw^{r}(\Delta^{n}), X)$. If
  $\mathcal{C}$ is an \icat{}, we write $\OSpan(\mathcal{C})$ for the
  simplicial set $\epsilon_{*}\mathcal{C}$.
\end{defn}

\begin{defn}
  Let $\Tw^{r}(\Delta^{n})_{0}$ denote the full subcategory of
  $\Tw^{r}(\Delta^{n})$ spanned by the objects $(i,j)$ where $j-i \leq 1$.
  We say a simplex $\Delta^{n} \to \OSpan(\mathcal{C})$ is
  \emph{cartesian} if the corresponding functor
  $F \colon \Tw^{r}(\Delta^{n}) \to \mathcal{C}$ is the right Kan
  extension of its restriction to $\Tw^{r}(\Delta^{n})_{0}$, or
  equivalently if for all integers
  $0 \leq i \leq k \leq l \leq j \leq n$, the square
  \nolabelcsquare{F(i,j)}{F(k,j)}{F(i,l)}{F(k,l)} is cartesian. We
  write $\Span(\mathcal{C})$ for the simplicial subset of
  $\OSpan(\mathcal{C})$ containing only the cartesian simplices.
\end{defn}

\begin{remark}\label{rmk:spanprop}
  A morphism $\mathcal{I} \to \OSpan(\mathcal{C})$ corresponds to a
  functor $F \colon \Tw^{r}(\mathcal{I}) \to \mathcal{C}$. Unwinding the
  definitions, we see that the map to $\OSpan(\mathcal{C})$ takes $i
  \in \mathcal{I}$ to $F(i \xto{\id} i)$ and a morphism $f \colon i
  \to j$ to the value of $F$ at the span
  \[
    \begin{tikzcd}
      i  \arrow[equals]{d} \arrow[equals]{r} & i \arrow{r}{f}
      \arrow{d}{f} & j
      \arrow[equals]{d} \\
      i \arrow{r}{f} & j \arrow[equals]{r} & j.
    \end{tikzcd}
  \]
  A functor $\mathcal{I} \to \Span(\mathcal{C})$ then corresponds to a
  functor $\Tw^{r}(\mathcal{I}) \to \mathcal{C}$ such that for all
  composable morphisms $f \colon i \to j$, $g \colon j \to k$, the
  value of $F$ at the commutative square
  \[
    \begin{tikzcd}
      {} & i \arrow{dd} \arrow[equals]{dl} \arrow{drr} \\
      i \arrow[crossing over]{drr} \arrow{dd} & & & j \arrow{dd} \arrow[equals]{dl} \\
      & k \arrow[leftarrow]{dl} \arrow[equals]{drr}  & j  \\
      j \arrow[equals]{drr} & & & k \arrow[leftarrow]{dl} \\
       & & j \arrow[equals]{uu}
    \end{tikzcd}
    \]
  in $\Tw^{r}(\mathcal{I})$ is a cartesian square in $\mathcal{C}$.
\end{remark}

\begin{propn}[Barwick, \cite{BarwickMackey}*{Proposition 3.4}]
  If $\mathcal{C}$ is an \icat{} with pullbacks, then
  $\Span(\mathcal{C})$ is an \icat{}. \qed
\end{propn}

\begin{defn}
  Following Barwick~\cite{BarwickMackey}, we say a \emph{triple} is a
  list $(\mathcal{C}, \mathcal{C}^{F}, \mathcal{C}^{B})$ where
  $\mathcal{C}$ is an \icat{} and $\mathcal{C}^{B}$ and
  $\mathcal{C}^{F}$ are both subcategories of $\mathcal{C}$ containing
  all the equivalences. We will call the morphisms in
  $\mathcal{C}^{B}$ the \emph{backwards} morphisms and the morphisms
  in $\mathcal{C}^{F}$ the \emph{forwards} morphisms in the triple. We
  say a triple is \emph{adequate} if for every morphism $f \colon x
  \to y$ in $\mathcal{C}^{F}$ and $g \colon z \to y$ in
  $\mathcal{C}^{B}$, there is a pullback square
  \csquare{w}{z}{x}{y}{f'}{g'}{g}{f}
  where $f'$ is in $\mathcal{C}^{F}$ and $g'$ is in $\mathcal{C}^{B}$.
\end{defn}

\begin{ex}
  If $\mathcal{C}$ is any \icat{}, we have the triple
  $(\mathcal{C}, \mathcal{C}, \mathcal{C})$ where all morphisms are
  both forwards and backwards morphisms. We call this the
  \emph{maximal} triple on $\mathcal{C}$; it is adequate \IFF{}
  $\mathcal{C}$ has pullbacks.
\end{ex}

\begin{remark}
  In \cite{BarwickMackey}, the forwards morphisms are called
  \emph{ingressive} and the backwards morphisms are called
  \emph{egressive}.
\end{remark}

\begin{defn}
  Given a triple $(\mathcal{C}, \mathcal{C}^{F}, \mathcal{C}^{B})$ we
  define $\OSpan_{B,F}(\mathcal{C})$ to be the simplicial subset
  of $\OSpan(\mathcal{C})$ containing only those simplices that
  correspond to maps $\sigma \colon \Tw^{r}(\Delta^{n}) \to
  \mathcal{C}$ such that for all $i,j$ the map $\sigma(i,j) \to \sigma(i+1,j)$
  lies in $\mathcal{C}^{F}$ and the map $\sigma(i,j) \to \sigma(i, j-1)$
  lies in $\mathcal{C}^{B}$. We write $\Span_{B,F}(\mathcal{C})$ for
  the simplicial subset of $\OSpan_{B,F}(\mathcal{C})$ containing
  the cartesian simplices with this property.
\end{defn}

\begin{propn}[Barwick, \cite{BarwickMackey}*{Proposition 5.6}]
  If $(\mathcal{C}, \mathcal{C}^{F}, \mathcal{C}^{B})$ is an adequate
  triple, then $\Span_{B,F}(\mathcal{C})$ is an \icat{}. \qed
\end{propn}

\subsection{Spans and Fibrations}

\begin{defn}\label{defn:cartfibtriple}
  Given an adequate triple $(\mathcal{B}, \mathcal{B}^{F}, \mathcal{B}^{B})$ and
  an inner fibration $p \colon \mathcal{E} \to \mathcal{B}$ such that
  $\mathcal{E}$ has $p$-cartesian morphisms over morphisms in
  $\mathcal{B}^{B}$, we define a triple
  $(\mathcal{E}, \mathcal{E}^{F}, \mathcal{E}^{B})$ by taking
  $\mathcal{E}^{B}$ to consist of cartesian morphisms over morphisms
  in $\mathcal{B}^{B}$ and $\mathcal{E}^{F}$ to consist of all
  morphisms lying over morphisms in $\mathcal{B}^{F}$.
\end{defn}

\begin{propn}\label{propn:cartfibtriplead}
  In the situation of Definition~\ref{defn:cartfibtriple}, the triple
  $(\mathcal{E}, \mathcal{E}^{F}, \mathcal{E}^{B})$ is
  adequate. Moreover, we have a pullback square of simplicial sets
  \nolabelcsquare{\Span_{B,F}(\mathcal{E})}{\OSpan_{B,F}(\mathcal{E})}{\Span_{B,F}(\mathcal{B})}{\OSpan_{B,F}(\mathcal{B}).}
\end{propn}

This is a consequence of the following simple observation:
\begin{lemma}\label{lem:cartpb}
  Let $p \colon \mathcal{E} \to \mathcal{B}$ be an inner
  fibration, and suppose we have a pullback square
  \csquare{a'}{b'}{a}{b}{f'}{\alpha}{\beta}{f}
  in $\mathcal{B}$. If $\bar{b}' \to \bar{b}$ is a morphism in
  $\mathcal{E}$ over $\beta$ and there exist $p$-cartesian morphisms
  $\bar{a} \to \bar{b}$ over $f$ and $\bar{a}' \to \bar{b}'$ over
  $f'$, then the commutative square
  \nolabelcsquare{\bar{a}'}{\bar{b}'}{\bar{a}}{\bar{b}}
  (where the left vertical morphism is induced by the universal
  property of $\bar{a} \to \bar{b}$) is cartesian.
\end{lemma}
\begin{proof}
  For any $\bar{x}$ in $\mathcal{E}$ over $x \in \mathcal{B}$ we have
  a commutative cube
  \[
  \begin{tikzcd}
    {} & \Map_{\mathcal{E}}(\bar{x}, \bar{a}') \arrow{rr} \arrow{dd} \arrow{dl} & &
    \Map_{\mathcal{E}}(\bar{x}, \bar{b}') \arrow{dd} \arrow{dl}\\
    \Map_{\mathcal{E}}(\bar{x}, \bar{a}) \arrow{dd} \arrow[crossing over]{rr} & &
    \Map_{\mathcal{E}}(\bar{x}, \bar{b}) \\
    {} & \Map_{\mathcal{B}}(x, a') \arrow{rr}\arrow{dl} & &
    \Map_{\mathcal{B}}(x, b') \arrow{dl} \\
    \Map_{\mathcal{B}}(x, a) \arrow{rr} & & \Map_{\mathcal{B}}(x, b)
    \arrow[crossing over,leftarrow]{uu}
  \end{tikzcd}
\]
  in the \icat{} of spaces.
  Here the bottom face is cartesian since $a'$ is a pullback, and the
  front and back faces are cartesian since the morphisms $\bar{a} \to
  \bar{b}$ and $\bar{a}' \to \bar{b}'$ are $p$-cartesian. Therefore
  the top face is also cartesian. Since this holds for all $\bar{x}
  \in \mathcal{E}$ this means $\bar{a}'$ is the pullback $\bar{a}
  \times_{\bar{b}} \bar{b}'$, as required.
\end{proof}
\begin{proof}[Proof of Proposition~\ref{propn:cartfibtriplead}]
  Adequacy follows immediately from Lemma~\ref{lem:cartpb}. Moreover,
  this lemma also shows that an $n$-simplex of
  $\OSpan_{B,F}(\mathcal{E})$ lies in $\Span_{B,F}(\mathcal{E})$
  \IFF{} it maps to an $n$-simplex of $\Span_{B,F}(\mathcal{B})$,
  giving the pullback square.
\end{proof}

\begin{defn}
  For $K$ a simplicial set, let $\Tw^{r}_{B}(K)$ denote the marked
  simplicial set $(\Tw^{r}(K), B)$ where $B$ is the set of ``backwards''
  maps, \ie{} those lying in the image of $K^{\op} \to \Tw^{r}(K)$.
\end{defn}

In the remaining part of this subsection we give a reformulation of
the results of \cite{BarwickMackey}*{\S 12} that will be convenient
for us.

\begin{propn}\label{propn:Twmarkedanod}
  For $0 < k < n$, the map $\Tw^{r}_{B}(\Lambda^{n}_{k})^{\op} \to \Tw^{r}_{B}(\Delta^{n})^{\op}$ is marked anodyne in
  the sense of \cite{HTT}*{Definition 3.1.1.1}.
\end{propn}
\begin{proof}
  This follows from the filtration defined in \cite{BarwickMackey}*{\S
    12}, using \cite{BarwickMackey}*{Proposition 12.14}.
\end{proof}

\begin{cor}\label{cor:Span inner fib}
  If $\mathcal{E} \to \mathcal{B}$ is as in Definition~\ref{defn:cartfibtriple}, then
  \begin{enumerate}[(i)]
  \item   $\OSpan_{B,F}(\mathcal{E}) \to \OSpan_{B,F}(\mathcal{B})$ is
    an inner fibration.
  \item $\Span_{B,F}(\mathcal{E}) \to \Span_{B,F}(\mathcal{B})$ is an
    inner fibration.
  \end{enumerate}
\end{cor}
\begin{proof}
  To prove (i) we must show that there exists a lift in every
  commutative square
  \liftcsquare{\Lambda^n_k}{\OSpan_{B,F}(\mathcal{E})}{\Delta^n}{\OSpan_{B,F}(\mathcal{B})}{}{}{}{}{}
  with $0 < k < n$. This is equivalent to giving a lift in the corresponding commutative square
  \liftcsquare{\Tw^{r} \Lambda^n_k}{\mathcal{E}}{\Tw^{r}
    \Delta^n}{\mathcal{B}.}{}{}{}{}{}
  Here the lift exists by Proposition~\ref{propn:Twmarkedanod}, since
  by definition the backwards maps go to cartesian morphisms in
  $\mathcal{E}$. Now (ii) follows from the pullback square in
  Proposition~\ref{propn:cartfibtriplead}.  
\end{proof}

\begin{propn}\label{propn:SpanLocCoCart}
  Let $p \colon \mathcal{E} \to \mathcal{B}$ be as in
  Definition~\ref{defn:cartfibtriple}, and assume that in
  addition $\mathcal{E}$ has locally $p$-cocartesian edges over morphisms in
  $\mathcal{B}^{F}$. Then:
  \begin{enumerate}[(i)]
  \item $\OSpan_{B,F}(\mathcal{E}) \to \OSpan_{B,F}(\mathcal{B})$ is a
    locally cocartesian fibration,
  \item $\Span_{B,F}(\mathcal{E}) \to \Span_{B,F}(\mathcal{B})$ is
    a locally cocartesian fibration,
  \end{enumerate}
  A span $X \xfrom{f} Y \xto{g} Z$ in $\mathcal{E}$ is locally
  $p$-cocartesian \IFF{} $g$ is a locally $p$-cocartesian morphism in
  $\mathcal{E}$.
\end{propn}
\begin{proof}
  We first prove (i). Consider a 1-simplex $\phi$ of
  $\OSpan_{B,F}(\mathcal{B})$, which corresponds to a span
  $b \xfrom{f} b' \xto{g} b''$ in $\mathcal{B}$. We wish to show that
  the pullback $\phi^{*}\OSpan_{B,F}(\mathcal{E}) \to \Delta^{1}$ is a
  cocartesian fibration. Pick an object $e$ of $\mathcal{E}$ lying
  over $b$. Then a 1-simplex of $\OSpan_{B,F}(\mathcal{E})$ with
  source $e$ lying over $\phi$ is a span
  $e \xfrom{\bar{f}} e' \xto{\bar{g}} e''$ where $\bar{f}$ is a
  cartesian morphism over $f$ and $\bar{g}$ is any morphism over $g$. The space
  of maps from $e$ to $e''$ in $\phi^{*}\OSpan_{B,F}(\mathcal{E})$ can
  therefore be identified with the space
  $\Map_{\mathcal{E}}(e', e'')_{g}$ of maps in $\mathcal{E}$ lying
  over $g$. From this it follows immediately that if $\bar{g} \colon
  e' \to e''$ is a locally cocartesian morphism from $e'$ over $g$
  then the span $e \xfrom{\bar{f}} e' \xto{\bar{g}} e''$ is locally
  cocartesian, as required. This proves (i),
  from which (ii) follows by the pullback square of
  Proposition~\ref{propn:cartfibtriplead}.
\end{proof}

\begin{cor}\label{cor:spancocartfib}
  Let $p \colon \mathcal{E} \to \mathcal{B}$ be as in
  Definition~\ref{defn:cartfibtriple}, and assume in
  addition:
  \begin{enumerate}[(1)]
  \item $\mathcal{E}$ has $p$-cocartesian edges over morphisms in
    $\mathcal{B}^{F}$.
  \item Consider a pullback square
    \csquare{a'}{b'}{a}{b}{f'}{\alpha}{\beta}{f}
    in $\mathcal{B}$ with $\alpha,\beta$ in $\mathcal{B}^{F}$ and
    $f',f$ in $\mathcal{B}^{B}$. Let $\bar{b}'$ be an object of
    $\mathcal{E}$ over $b'$,
    and suppose $\bar{b}' \xto{\bar{\beta}} \bar{b}$ is a $p$-cocartesian morphism over
    $\beta$ and $\bar{a} \xto{\bar{f}} \bar{b}$ and $\bar{a}' \xto{\bar{f'}} \bar{b}'$ are
    $p$-cartesian morphisms over $f$ and $f'$. Then in the commutative
    square
    \csquare{\bar{a}'}{\bar{b}'}{\bar{a}}{\bar{b}}{\bar{f}'}{\bar{\alpha}}{\bar{\beta}}{\bar{f}}
    induced by the universal property of $\bar{f}$, the morphism
    $\bar{\alpha}$ is again $p$-cocartesian.
  \end{enumerate}
  Then $\Span_{B,F}(\mathcal{E}) \to \Span_{B,F}(\mathcal{B})$ is a
  cocartesian fibration.
\end{cor}
\begin{proof}
  We know from Proposition~\ref{propn:SpanLocCoCart} that
  $\Span_{B,F}(\mathcal{E}) \to \Span_{B,F}(\mathcal{B})$ is a locally
  cocartesian fibration. By \cite{HTT}*{Proposition 2.4.2.8} it
  therefore suffices to show that the locally cocartesian morphisms
  are closed under composition. Lemma~\ref{lem:cartpb} shows that this
  is indeed the case under the given assumptions.
\end{proof}

\section{Day Convolution for Double
  $\infty$-Categories}\label{sec:dayconv}
In this section we carry out the main technical construction of this
paper: We show that for a certain class of double \icats{}
$\mathcal{M}$, there exists a \emph{Day convolution} double \icat{}
$\widehat{\mathcal{M}}_{\mathcal{S}}$ such that for any non-symmetric \iopd{}
$\mathcal{O}$ we have a natural equivalence
\[ \Alg_{\mathcal{O}}(\widehat{\mathcal{M}}_{\mathcal{S}}) \simeq
  \Seg_{\mathcal{O} \times_{\Dop} \mathcal{M}}(\mathcal{S}).\] In
\S\ref{sec:unfold} we introduce an ``unfolding'' construction that we
use to define $\widehat{\mathcal{M}}_{\mathcal{S}}$ in \S\ref{sec:day}; we then
establish the universal property in \S\ref{sec:dayunivp}. Next we
prove in \S\ref{sec:daymon} that we may view associative algebras in
$\widehat{\mathcal{M}}$ as algebras in a family of monoidal
\icats{}. We also consider enriched variants of the Day convolution
construction in \S\ref{sec:enrdc}, and in \S\ref{subsec:enrcat} we
illustrate the theory by discussing the example of enriched \icats{}.

\subsection{An Unfolding Construction}\label{sec:unfold}
Suppose we have a cocartesian fibration
$p \colon \mathcal{E} \to \mathcal{U}$ and a cartesian fibration
$q \colon \mathcal{U} \to \mathcal{B}$. Our goal in this subsection
is to construct for every cocomplete \icat{} $\mathcal{X}$ a
cocartesian fibration $\widetilde{\mathcal{E}}_{\mathcal{X}} \to
\mathcal{B}$ with the universal property that for any functor
$\mathcal{C} \to \mathcal{B}$ there is a natural equivalence
\[ \Map_{/\mathcal{B}}(\mathcal{C},
  \widetilde{\mathcal{E}}_{\mathcal{X}}) \isoto \Map(\mathcal{C}
  \times_{\mathcal{B}} \mathcal{E}, \mathcal{X}).\]

\begin{remark}
  Recall that a functor of \icats{} $f$ is called an
  \emph{exponentiable}, \emph{flat}, or \emph{Conduch\'e fibration} if the functor $f^{*}$ given by
  pullback along $f$ has a right adjoint $f_{*}$. Both cartesian and
  cocartesian fibrations are examples of exponentiable fibrations,
  hence the composite $qp \colon \mathcal{E} \to \mathcal{B}$ is an
  exponentiable fibration. The universal property of
  $\widetilde{\mathcal{E}}_{\mathcal{X}}$ is that of $(qp)_{*}(\mathcal{X}
  \times \mathcal{E})$, but it is not clear from the latter that
  $\widetilde{\mathcal{E}}_{\mathcal{X}}$ will be a
  cocartesian fibration if $\mathcal{X}$ is cocomplete.
\end{remark}

To define $\widetilde{\mathcal{E}}_{\mathcal{X}}$ we first introduce
an ``unfolding construction'' that uses $p$ and $q$ to  
 construct a functor
$\mathcal{B} \to \Span(\CatI)$ that takes $b \in \mathcal{B}$ to the
fibre $\mathcal{E}_{b}$ of the composite
$\mathcal{E} \to \mathcal{B}$, and takes a morphism
$f \colon b \to b'$ to the top row in the diagram
\[
  \begin{tikzcd}
    \mathcal{E}_{b} \arrow{d}{p_{b}}& f^{*}\mathcal{E}_{b} \arrow{l}
    \arrow{r}{f_{!}} \arrow[phantom]{dl}[very near start]{\llcorner} \arrow{d}&
    \mathcal{E}_{b'} \arrow{d}{p_{b'}} \\
    \mathcal{U}_{b} & \mathcal{U}_{b'} \arrow{l}{f^{*}} \arrow[equals]{r} &
    \mathcal{U}_{b'},
  \end{tikzcd}
\]
where $f^{*} \colon \mathcal{U}_{b'} \to \mathcal{U}_{b}$ is the functor
given by the cartesian morphisms over $f$ and the left square is a
pullback; an object of $f^{*}\mathcal{E}_{b}$ then corresponds to a pair
$(x \in \mathcal{E}_{b}, u \in \mathcal{U}_{b'})$ such that $p(x)
\simeq f^{*}u$ in $\mathcal{U}_{b}$, and the top right morphism
takes $(x,u)$ to the cocartesian pushforward $\bar{f}_{!}x$ where
$\bar{f} \colon f^{*}u \to u$ is the $q$-cartesian morphism over $f$.

\begin{construction}
Let $q^{\vee}\colon \mathcal{U}^{\vee} \to \mathcal{B}^{\op}$ be the cocartesian
fibration dual to $q \colon \mathcal{U} \to \mathcal{B}$ (\ie{} the cocartesian
fibration corresponding to the same functor as $q$).
By \cite{freepres}*{Theorem 4.5}, the \emph{free} cocartesian 
fibration on $q^{\vee}$ is
$\mathcal{U}^{\vee} \times_{\mathcal{B}^{\op}}
(\mathcal{B}^{\op})^{\Delta^{1}} \to \mathcal{B}^{\op}$, where the
fibre product uses $q^{\vee}$ and evaluation at $0$ and the functor to
$\mathcal{B}^{\op}$ uses evaluation at $1$. Since $q^{\vee}$ is a
cocartesian fibration, the identity induces a functor
\[ \mathcal{U}^{\vee} \times_{\mathcal{B}^{\op}}
  (\mathcal{B}^{\op})^{\Delta^{1}} \to \mathcal{U}^{\vee} \]
over $\mathcal{B}^{\op}$ that preserves cocartesian morphisms. Dualizing
again, we obtain a morphism of cartesian fibrations
\[ (\mathcal{U}^{\vee} \times_{\mathcal{B}^{\op}}
  (\mathcal{B}^{\op})^{\Delta^{1}})^{\vee} \to \mathcal{U}\]
that preserves cartesian morphisms. The following lemma identifies the
source of this functor with $\mathcal{U}^{\vee}
\times_{\mathcal{B}^{\op}} \Tw^{r}(\mathcal{B}^{\op})$:
\end{construction}

\begin{lemma}
  For any functor $f \colon \mathcal{C} \to \mathcal{B}$, the
  cartesian fibration
  $(\mathcal{C} \times_{\mathcal{B}} \mathcal{B}^{\Delta^{1}})^{\vee}
  \to \mathcal{B}^{\op}$ dual to the free cocartesian fibration on $f$
  is equivalent to
\[ \mathcal{C} \times_{\mathcal{B}} \Tw^{r}(\mathcal{B}) \to \mathcal{B}^{\op}.\]
\end{lemma}
\begin{proof}
  We can write the cocartesian fibration $\mathcal{C}
  \times_{\mathcal{B}} \mathcal{B}^{\Delta^{1}} \to \mathcal{B}$ as
  the fibre product
  \[ (\mathcal{C} \times \mathcal{B}) \times_{(\mathcal{B} \times
      \mathcal{B})} \mathcal{B}^{\Delta^{1}} \]
  of cocartesian fibrations over $\mathcal{B}$. Since dualization of
  fibrations is an equivalence of \icats{}, it preserves fibre
  products, hence we obtain an equivalence
  \[ (\mathcal{C} \times_{\mathcal{B}} \mathcal{B}^{\Delta^{1}})^{\vee}
    \simeq (\mathcal{C} \times \mathcal{B}^{\op}) \times_{(\mathcal{B}
      \times \mathcal{B}^{\op})} (\mathcal{B}^{\Delta^{1}})^{\vee}
    \simeq \mathcal{C} \times_{\mathcal{B}}(\mathcal{B}^{\Delta^{1}})^{\vee}.\]
  By \cite{cois}*{Proposition A.2.4} the dual of $\mathcal{B}^{\Delta^{1}}
  \to \mathcal{B}$ is  $\Tw^{r}(\mathcal{B})
  \to \mathcal{B}^{\op}$, which completes the proof.
\end{proof}

\begin{defn}
  Given a cartesian fibration $\mathcal{U} \to \mathcal{B}$, we have
  constructed a canonical functor
  \[ \mathfrak{c}_{\mathcal{U}} \colon
    \mathcal{U}^{\vee}\times_{\mathcal{B}^{\op}} \Tw^{r}(\mathcal{B})
    \to \mathcal{U}.\] For $\mathcal{E} \to \mathcal{U}$ a cocartesian
  fibration, we define the \emph{unfolding} $\txt{Unf}(\mathcal{E})$ as the fibre product
  $(\mathcal{U}^{\vee} \times_{\mathcal{B}^{\op}}
  \Tw^{r}(\mathcal{B})) \times_{\mathcal{U}} \mathcal{E}$, using
  $\mathfrak{c}_{\mathcal{U}}$, and write
  $\overline{\mathfrak{c}}_{\mathcal{U}}$ for the induced map
  $\Unf(\mathcal{E}) \to \mathcal{E}$ over $\mathfrak{c}_{\mathcal{U}}$. The projection
  $\txt{Unf}(\mathcal{E}) \to \Tw^{r}(\mathcal{B})$ is then a
  cocartesian fibration, since it decomposes as a composite
  \[ (\mathcal{U}^{\vee} \times_{\mathcal{B}^{\op}}
  \Tw^{r}(\mathcal{B})) \times_{\mathcal{U}} \mathcal{E} \to
  \mathcal{U}^{\vee} \times _{\mathcal{B}^{\op}}
  \Tw^{r}(\mathcal{B}) \to \Tw^{r}(\mathcal{B}),\]
where the first map is a pullback of the cocartesian fibration
$\mathcal{E} \to \mathcal{U}$ and the second is a pullback of the
cocartesian fibration $\mathcal{U}^{\vee} \to \mathcal{B}^{\op}$.
\end{defn}

\begin{remark}\label{unfpb}
  Given a functor $\mathcal{C} \to \mathcal{B}$, we have a commutative
  diagram
  \[
    \begin{tikzcd}
      \txt{Unf}(\mathcal{C} \times_{\mathcal{B}}\mathcal{E})
      \arrow{rr} \arrow{dr} \arrow{dd} & & \mathcal{C}
      \times_{\mathcal{B}}\mathcal{E} \arrow{dr} \arrow{dd} \\
      & \txt{Unf}(\mathcal{E})   \arrow[crossing over]{rr} & &
      \mathcal{E}  \arrow{dd}\\
      (\mathcal{C} \times_{\mathcal{B}}\mathcal{U})^{\vee}
      \times_{\mathcal{C}^{\op}} \Tw^{r}(\mathcal{C})\arrow{rr} \arrow{dr} \arrow{dd} & &
      \mathcal{C} \times_{\mathcal{B}} \mathcal{U} \arrow{dr} \arrow{dd} \\
       & \mathcal{U}^{\vee}
       \times_{\mathcal{B}^{\op}}\Tw^{r}(\mathcal{B})
       \arrow[leftarrow,crossing over]{uu}
       \arrow[crossing over]{rr} & & \mathcal{U}
       \arrow{dd} \\
       \Tw^{r}(\mathcal{C}) \arrow{rr} \arrow{dr} & & \mathcal{C} \arrow{dr}\\
       & \Tw^{r}(\mathcal{B}) \arrow[leftarrow,crossing over]{uu} \arrow{rr} & & \mathcal{B}.
    \end{tikzcd}
  \]
  In the top cube, the back and front faces are cartesian by
  definition of unfolding, and the right face is cartesian since the
  bottom right and right composite squares are cartesian. This implies
  that the left square in the top cube
  is cartesian.  Moreover, since dualization of fibrations is
  compatible with pullbacks we have
  $(\mathcal{C} \times_{\mathcal{B}} \mathcal{U})^{\vee} \simeq
  \mathcal{U}^{\vee} \times_{\mathcal{B}^{\op}} \mathcal{C}^{\op}$,
  and hence
  \[ (\mathcal{C} \times_{\mathcal{B}} \mathcal{U})^{\vee}
    \times_{\mathcal{C}^{\op}} \Tw^{r}(\mathcal{C}) \simeq
    \mathcal{U}^{\vee} \times_{\mathcal{B}^{\op}} \Tw^{r}(\mathcal{C})
    \simeq \left(\mathcal{U}^{\vee} \times_{\mathcal{B}^{\op}}
      \Tw^{r}(\mathcal{B})\right) \times_{\Tw^{r}(\mathcal{B})}
    \Tw^{r}(\mathcal{C}).
  \]
  The bottom left face in the diagram is therefore cartesian, and so
  the left composite square is a pullback. Thus unfolding is
  compatible with base change, in the sense that we have a natural
  equivalence
  \[ \Unf(\mathcal{C} \times_{\mathcal{B}} \mathcal{E}) \isoto
    \Tw^{r}(\mathcal{C}) \times_{\Tw^{r}(\mathcal{B})} \Unf(\mathcal{E}).\]
\end{remark}

\begin{lemma}
  The cocartesian fibration $\txt{Unf}(\mathcal{E}) \to
  \Tw^{r}(\mathcal{B})$ corresponds to a functor $\mathfrak{U}_{\mathcal{E}} \colon \mathcal{B} \to \Span(\CatI)$.
\end{lemma}
\begin{proof}
  Given morphisms $a \xto{f} b \xto{g} c$ in $\mathcal{B}$, we have
  the commutative diagram of \icats{}
  \[
    \begin{tikzcd}
       {} & & g^{*}f^{*}\mathcal{E}_{a} \arrow{drr} \arrow{dl} \arrow{dd}
   \\
       & f^{*}\mathcal{E}_{a} \arrow{dd}\arrow{dl} \arrow{drr} & & & g^{*}\mathcal{E}_{b} \arrow{dd}\arrow{drr} \arrow{dl} \\
       \mathcal{E}_{a} \arrow{dd}& & \mathcal{U}_{c} \arrow{dl}
       \arrow[equals]{drr} & \mathcal{E}_{b}\arrow{dd} & & &
       \mathcal{E}_{c}\arrow{dd} \\
       & \mathcal{U}_{b} \arrow{dl} \arrow[equals]{drr} & & & \mathcal{U}_{c} \arrow[equals]{drr} \arrow{dl} \\
       \mathcal{U}_{a} & & & \mathcal{U}_{b} & & & \mathcal{U}_{c}       
    \end{tikzcd}
    \]
    and by Remark~\ref{rmk:spanprop} we must show that the commutative
    square in the top level is cartesian. But in the commutative cube
    the bottom, back left, and front right faces are cartesian, hence
    so is the top face.
\end{proof}

We can now define $\widetilde{\mathcal{E}}_{\mathcal{X}}$ using the
following construction:
\begin{defn}
  For any \icat{} $\mathcal{X}$, let
  $p_{\mathcal{X}} \colon \mathcal{F}_{\mathcal{X}} \to \CatI$ denote the cartesian fibration
  correspoding to the functor
  $\Fun(\blank, \mathcal{X}) \colon \CatI^{\op} \to \LCatI$. If
  $\mathcal{X}$ is cocomplete, then this is also a cocartesian
  fibration (with cocartesian morphisms given by left Kan
  extensions). We then have a locally cocartesian fibration
  $\Span_{B,F}(\mathcal{F}_{\mathcal{X}}) \to \Span(\CatI)$ by
  Proposition~\ref{propn:SpanLocCoCart}, where
  $\mathcal{F}_{\mathcal{X}}$ is equipped with the triple structure
  from \cref{defn:cartfibtriple}.
\end{defn}

\begin{defn}
  Given a cocartesian fibration $\mathcal{E} \to \mathcal{U}$ and a
  cartesian fibration $\mathcal{U} \to \mathcal{B}$, we let
  $\widetilde{\mathcal{E}}_{\mathcal{X}}$ for a cocomplete \icat{} $\mathcal{X}$
  be defined by the pullback
  \[
    \ncsquare{\widetilde{\mathcal{E}}_{\mathcal{X}}}{\Span_{B,F}(\mathcal{F}_{\mathcal{X}})}{\mathcal{B}}{\Span(\CatI).}{}{}{p_{\mathcal{X}}}{\mathfrak{U}_{\mathcal{E}}}
  \]
  Then $\widetilde{\mathcal{E}}_{\mathcal{X}} \to \mathcal{B}$ is a
  locally cocartesian fibration.
\end{defn}

\begin{lemma}
  Let $\mathcal{X}$ be a cocomplete \icat{}.  The locally cocartesian
  fibration $\widetilde{\mathcal{E}}_{\mathcal{X}} \to \mathcal{B}$ is
  a cocartesian fibration. 
\end{lemma}
\begin{proof}
  We must show
  that the locally cocartesian morphisms are closed under composition.
  For morphisms $a \xto{f} b \xto{g} c$ in $\mathcal{B}$, we have the
  cartesian square
  \[
    \ncsquare{g^{*}f^{*}\mathcal{E}_{a}}{g^{*}\mathcal{E}_{b}}{f^{*}\mathcal{E}_{a}}{\mathcal{E}_{b}}{F'}{G'}{G}{F}
  \]
  as above, and we must show that the mate transformation
  \[ F'_{!}G'^{*} \to G^{*}F_{!} \]
  of functors $\Fun(f^{*}\mathcal{E}_{a}, \mathcal{X}) \to
  \Fun(g^{*}\mathcal{E}_{b}, \mathcal{X})$
  is an equivalence. At $\phi \in \Fun(f^{*}\mathcal{E}_{a},
  \mathcal{X})$ and $x \in g^{*}\mathcal{E}_{b}$, the mate
  transformation evaluates to the natural map of colimits
 \[ \colim_{y \in (g^{*}f^{*}\mathcal{E}_{a})_{/x}} \phi(G'y) \to
    \colim_{z \in (f^{*}\mathcal{E}_{a})_{/Gx}} \phi(z) \]
  arising from the functor $(g^{*}f^{*}\mathcal{E}_{a})_{/x} \to
  (f^{*}\mathcal{E}_{a})_{/Gx}$ induced by $G$. It thus suffices to
  show that this functor is cofinal.

  By definition, $(g^{*}f^{*}\mathcal{E}_{a})_{/x}$ is the pullback
  $g^{*}f^{*}\mathcal{E}_{a} \times_{g^{*}\mathcal{E}_{b}}
  (g^{*}\mathcal{E}_{b})_{/x}$. Since $g^{*}\mathcal{E}_{b}$ and
  $g^{*}f^{*}\mathcal{E}_{a}$ are pulled back along $\mathcal{U}_{c}
  \to \mathcal{U}_{b}$, we can rewrite this to see that there is a
  natural pullback square
  \[
    \nlcsquare{(g^{*}f^{*}\mathcal{E}_{a})_{/x}}{(f^{*}\mathcal{E}_{a})_{/Gx}}{(\mathcal{U}_{c})_{/\pi_{c}x}}{(\mathcal{U}_{b})_{/\pi_{b}Gx},}
  \]
  where $\pi_{t}$ denotes the projection $\mathcal{E}_{t}\to
  \mathcal{U}_{t}$. In this square the right vertical functor is a
  cocartesian fibration, and the bottom horizontal functor is cofinal
  since both $(\mathcal{U}_{c})_{/\pi_{c}x}$ and
  $(\mathcal{U}_{b})_{/\pi_{b}Gx}$ have a terminal object, which is
  preserved by this functor. It follows by  \cite{HTT}*{Proposition
    4.1.2.15} that the top horizontal functor is also cofinal, as
  required.
\end{proof}

\begin{remark}
  The cocartesian fibration $\widetilde{\mathcal{E}}_{\mathcal{X}} \to
  \mathcal{B}$ corresponds to a functor $\mathcal{B} \to \CatI$. This
  takes $b \in \mathcal{B}$ to $\Fun(\mathcal{E}_{b}, \mathcal{X})$
  and a morphism $f \colon b \to b'$ to the composite functor
  \[ \Fun(\mathcal{E}_{b}, \mathcal{X}) \to \Fun(f^{*}\mathcal{E}_{b},
    \mathcal{X}) \to \Fun(\mathcal{E}_{b'}, \mathcal{X})\]
  where the first functor is given by composition with
  $f^{*}\mathcal{E}_{b} \to \mathcal{E}_{b}$ and the second by left
  Kan extension along $f^{*}\mathcal{E}_{b} \to
  \mathcal{E}_{b'}$. Both $f^{*}\mathcal{E}_{b}$ and
  $\mathcal{E}_{b'}$ are cocartesian fibrations over
  $\mathcal{U}_{b'}$, and the functor $f_{!} \colon
  f^{*}\mathcal{E}_{b} \to \mathcal{E}_{b'}$ preserves cocartesian morphisms. The
  following lemma therefore implies that the left Kan extension along
  $f_{!}$ can be computed fibrewise, \ie{} for $\Phi \colon
  f^{*}\mathcal{E}_{b} \to \mathcal{X}$ and $x \in
  \mathcal{E}_{b'}$ over $u \in \mathcal{U}_{b'}$ we have
  \[ (f_{!})_{!}\Phi(x) \simeq \colim_{(\mathcal{E}_{b,f^{*}u})_{/x}}
    \Phi, \] where we have used the equivalence
  $(f^{*}\mathcal{E}_{b})_{u} \simeq \mathcal{E}_{b,f^{*}u}$, and the
  slice
  \[(\mathcal{E}_{b,f^{*}u})_{/x}:= \mathcal{E}_{b,f^{*}u}
  \times_{\mathcal{E}_{b',u}} (\mathcal{E}_{b',u})_{/x}\] is defined
  using the functor $\mathcal{E}_{b,f^{*}u} \to \mathcal{E}_{b',u}$
  given by cocartesian pushforward along the cartesian morphism $\bar{f} \colon f^{*}u \to
  u$. We obtain the following description of the functor
  $\Fun(\mathcal{E}_{b}, \mathcal{X}) \to \Fun(\mathcal{E}_{b'},
  \mathcal{X})$ arising from the cocartesian fibration
  $\widetilde{\mathcal{E}}_{\mathcal{X}} \to \mathcal{B}$: For $\Psi
  \colon \mathcal{E}_{b} \to \mathcal{X}$, its image is the functor
  $\mathcal{E}_{b'}\to \mathcal{X}$ that to $x \in \mathcal{E}_{b',u}$
  assigns
  \[ \colim_{(y, \bar{f}_{!}y \to x) \in
      (\mathcal{E}_{b,f^{*}u})_{/x}} \Psi(y).\]
\end{remark}

\begin{lemma}\label{lem:cocartlke}
  Consider a commutative triangle of \icats{}
  \[
    \begin{tikzcd}
      \mathcal{E} \arrow{rr}{f} \arrow{dr}[below left]{p} & & \mathcal{F}
      \arrow{dl}{q} \\
       & \mathcal{B},
    \end{tikzcd}
  \]
  where $p$ and $q$ are cocartesian fibrations and $f$ preserves
  cocartesian morphisms. Then for $x \in \mathcal{F}_{b}$ the
  inclusion
  \[ \mathcal{E}_{b/x} := \mathcal{E}_{b} \times_{\mathcal{F}_{b}}
    \mathcal{F}_{b/x} \to \mathcal{E} \times_{\mathcal{F}}
    \mathcal{F}_{/x} =: \mathcal{E}_{/x}\]
  is cofinal. In particular, if $\mathcal{C}$ is a cocomplete
  \icat{} then the left Kan extension $f_{!}F$ along $f$ of any functor $F
  \colon \mathcal{E} \to \mathcal{C}$ can be computed fibrewise over
  $\mathcal{B}$, \ie{} for $x \in \mathcal{F}_{b}$ we have
  \[ f_{!}F(x) \simeq \colim_{y \in \mathcal{E}_{b/x}} F(y).\]
\end{lemma}
\begin{proof}
  By \cite{HTT}*{Theorem 4.1.3.1} it suffices to check that for every
  object $\eta = (y, f(y) \xto{\phi} x)$ in $\mathcal{E}_{/x}$, the \icat{}
  $(\mathcal{E}_{b/x})_{\eta/}$ is weakly contractible. This \icat{}
  has as objects maps $y \to y'$ over $q(\phi)$ together with commutative triangles
  \[
    \begin{tikzcd}
      f(y) \arrow{dr}{\phi} \arrow{dd} \\
      & x \\
      f(y') \arrow{ur}{\phi'}
    \end{tikzcd}
  \]
  where $\phi'$ lies over $\id_{b}$. It therefore has an initial
  object, given by the cocartesian morphism $\psi \colon y \to y'$
  over $q(\phi)$ together with the canonical factorization
  $f(y) \xto{f(\psi)} f(y') \to x$ that exists since $f(\psi)$ is
  again a cocartesian morphism.
\end{proof}

\begin{remark}\label{tildeEsec}
  We can identify sections of the cocartesian fibration
  $\widetilde{\mathcal{E}}_{\mathcal{X}}$ as follows: For any functor
  $\phi \colon \mathcal{C} \to \mathcal{B}$, we have
  \[
      \Map_{/\mathcal{B}}(\mathcal{C},
      \widetilde{\mathcal{E}}_{\mathcal{X}})  \simeq \left\{
        \ncsquare{\mathcal{C}}{\Span_{B,F}(\mathcal{F}_{\mathcal{X}})}{\mathcal{B}}{\Span(\CatI)}{}{\phi}{\Span(p_{\mathcal{X}})}{\mathfrak{U}_{\mathcal{E}}}
      \right\}
      \subseteq
      \left\{
        \ncsquare{\Tw^{r}(\mathcal{C})}{\mathcal{F}_{\mathcal{X}}}{\Tw^{r}(\mathcal{B})}{\CatI}{}{\Tw^{r}(\phi)}{p_{\mathcal{X}}}{}
      \right\}.
  \]
  By the pullback square in Proposition~\ref{propn:cartfibtriplead},
  the only condition for a point of the right-hand $\infty$-groupoid
  to lie in the image of $\Map_{/\mathcal{B}}(\mathcal{C},
  \widetilde{\mathcal{E}}_{\mathcal{X}})$ is that the functor
  $\Tw^{r}(\mathcal{C}) \xto{\Phi} \mathcal{F}_{\mathcal{X}}$ takes morphisms
  in $\Tw^{r}(\mathcal{C})$ of the form
  \[
    \begin{tikzcd}
      c \arrow{d}{f} \arrow[equals]{r} & c \arrow[equals]{d} \\
      d & c \arrow{l}{f}
    \end{tikzcd}
  \]
  to cartesian morphisms in $\mathcal{F}_{\mathcal{X}}$. This amounts
  to the natural transformation 
  \[
    \begin{tikzcd}
      \phi(f)^{*}\mathcal{E}_{\phi(c)} \arrow{dd} \arrow[dr,
      ""{name=U,inner sep=1pt,below left}, "\Phi(f)"{above right}] \\
      & \mathcal{X} \\
      |[alias=D]| \mathcal{E}_{\phi(c)} \arrow{ur}[below right]{\Phi(\id_{c})}
      \arrow[phantom, from=U, to=D, ""{name=UU, near start},
      ""{name=DD, near end}]
      \arrow[Rightarrow,
      from=UU, to=DD]
    \end{tikzcd}
  \]
  being an equivalence. Since $\mathcal{F}_{\mathcal{X}}$ is by
  definition the cartesian fibration for the functor
  $\Fun(\blank, \mathcal{X})$ we can use \cite{freepres}*{Proposition
    7.3} and \cref{unfpb} to obtain an equivalence
  \[
      \left\{
        \ncsquare{\Tw^{r}(\mathcal{C})}{\mathcal{F}_{\mathcal{X}}}{\Tw^{r}(\mathcal{B})}{\CatI}{}{\Tw^{r}(\phi)}{p_{\mathcal{X}}}{}
      \right\}
      \simeq
      \Map(\Tw^{r}(\mathcal{C})\times_{\Tw^{r}(\mathcal{B})}\Unf(\mathcal{E}),
      \mathcal{X}) \simeq \Map(\Unf(\mathcal{C} \times_{\mathcal{B}}
      \mathcal{E}), \mathcal{X}),
  \]
  natural in $\mathcal{C}$, under which $\Map_{/\mathcal{B}}(\mathcal{C},
  \widetilde{\mathcal{E}}_{\mathcal{X}})$ is identified with the
  functors
  $\Tw^{r}(\mathcal{C})\times_{\Tw^{r}(\mathcal{B})}\txt{Unf}(\mathcal{E})
  \to \mathcal{X}$ that take morphisms
  \[ (f \colon c \to d, e \in \mathcal{E}_{\phi(c)}, u \in
    \mathcal{U}_{\phi(d)}, \pi_{\phi(c)}e \simeq \phi(f)^{*}u) \to
    (\id_{c}, e, \phi(f)^{*}u, \pi_{\phi(c)}e \simeq \phi(f)^{*}u),\] over the morphism in
  $\Tw^{r}(\mathcal{C})$ above, to equivalences in $\mathcal{X}$.
\end{remark}

\begin{notation}
  Let $W_{\mathcal{U}/\mathcal{B}}$ denote the class of morphisms in
  $\mathcal{U}^{\vee} \times_{\mathcal{B}^{\op}} \Tw^{r}(\mathcal{B})$
  (or $\Unf(\mathcal{U})$)
  of the form
    \[ (u \in \mathcal{U}_{b}, f \colon a \to b) \to (f^{*}u, \id_{a}),\]
  corresponding to the cocartesian morphism $u \to f^{*}u$ in
  $\mathcal{U}^{\vee}$ and the morphism
    \[
    \begin{tikzcd}
      a \arrow{d}{f} \arrow[equals]{r} & a \arrow[equals]{d} \\
      b & a \arrow{l}{f}
    \end{tikzcd}
  \]
  in $\Tw^{r}(\mathcal{B})$; this is the union  over $b \in
  \mathcal{B}$ of the classes
  $W_{\mathcal{U}/\mathcal{B}, b}$ of such morphisms that lie over $b$.
  Then let
  $\overline{W}_{\mathcal{E}/\mathcal{B}}$ denote the class of
  morphisms in $\Unf(\mathcal{E})$ consisting of cocartesian morphisms
  lying over the morphisms in $W_{\mathcal{U}}$.
\end{notation}

\begin{remark}\label{secisloc}
  Using this notation, \cref{tildeEsec} identifies the space
  $\Map_{/\mathcal{B}}(\mathcal{C},
  \widetilde{\mathcal{E}}_{\mathcal{X}})$ of sections with the space
  of functors $\Unf(\mathcal{C} \times_{\mathcal{B}}\mathcal{E}) \to
  \mathcal{X}$ that take the morphisms in $\overline{W}_{\mathcal{C}
    \times_{\mathcal{B}} \mathcal{E}/\mathcal{C}}$ to equivalences in
  $\mathcal{X}$, or equivalently the space
  \[ \Map(\Unf(\mathcal{C} \times_{\mathcal{B}}
    \mathcal{E})[\overline{W}_{\mathcal{C} \times_{\mathcal{B}}
      \mathcal{E}/\mathcal{C}}^{-1}], \mathcal{X})\] of functors from
  the localization at
  $\overline{W}_{\mathcal{C} \times_{\mathcal{B}}
    \mathcal{E}/\mathcal{C}}$. Our next goal is to identify this
  localization with $\mathcal{C} \times_{\mathcal{B}} \mathcal{E}$.
\end{remark}

\begin{propn}
  The functor $\mathfrak{c}_{\mathcal{U}} \colon \mathcal{U}^{\vee}
  \times_{\mathcal{B}^{\op}}
  \Tw^{r}(\mathcal{B}) \to \mathcal{U}$ exhibits $\mathcal{U}$ as the
  localization at the class $W_{\mathcal{U}/\mathcal{B}}$.
\end{propn}
\begin{proof}
  The functor $\mathfrak{c}_{\mathcal{U}}$ is by construction a map of
  cartesian
  fibrations over $\mathcal{B}$ that preserves cartesian morphisms. On
  fibres over $b \in \mathcal{B}$, the functor
  \[ \mathfrak{c}_{\mathcal{U},b} \colon \mathcal{U}^{\vee} \times_{\mathcal{B}^{\op}}
    (\mathcal{B}_{b/})^{\op}  \to \mathcal{U}_{b}\]
  takes $(u \in \mathcal{U}_{a}, f \colon b \to a)$ to $f^{*}u
  \in \mathcal{U}_{b}$. This has a canonical section $s_{b}$, taking
  $u \in \mathcal{U}_{b}$ to $(u, \id_{b})$. Moreover, the cocartesian
  morphisms in $\mathcal{U}^{\vee}$ determine a natural transformation
  $\id \to s_{b}\mathfrak{c}_{\mathcal{U},b}$, given for $(u, f \colon
  b \to a)$ by the cocartesian morphism $u \to f^{*}u$ and the
  triangle
  \[
    \begin{tikzcd}
      {} & b \arrow{dl}{f} \arrow[equals]{dr} \\
      a & & b \arrow{ll}{f}.
    \end{tikzcd}
  \]
  It follows that $\mathfrak{c}_{\mathcal{U},b}$ exhibits
  $\mathcal{U}_{b}$ as the localization of
  $\mathcal{U}^{\vee} \times_{\mathcal{B}^{\op}}
  (\mathcal{B}_{b/})^{\op}$ at the class $W_{\mathcal{U}/\mathcal{B},b}$.
  The same argument
  also shows that $\mathcal{U}_{b}$ is the localization at the larger
  class $W'_{\mathcal{U}/\mathcal{B},b}$ of morphisms
  \[ (u, f\colon b \to a)  \to (u', f' \colon b \to a') \]
  over
    \[
    \begin{tikzcd}
      {} & b \arrow{dl}{f} \arrow{dr}{f'} \\
      a & & a' \arrow{ll}{\alpha}.
    \end{tikzcd}
  \]
  and $\alpha^{*}u \to u'$,
  such that the induced morphism $f^{*}u \simeq f'^{*}\alpha^{*}u \to
  f'^{*}u'$ is an equivalence in $\mathcal{U}_{b}$. This class is
  compatible with cartesian pullback, in the sense that
  $\beta^{*}W'_{\mathcal{U}/\mathcal{B},b} \subseteq W'_{\mathcal{U}/\mathcal{B},b'}$ for
  $\beta \colon b' \to b$, and so \cite{HinichLoc}*{Proposition 2.1.4} 
  implies that $\mathfrak{c}_{\mathcal{U}}$ exhibits $\mathcal{U}$ as
  the localization of $\mathcal{U}^{\vee} \times_{\mathcal{B}}^{\op}
  \Tw^{r}(\mathcal{B})$ at $W'_{\mathcal{U}/\mathcal{B}} := \bigcup_{b}
  W'_{\mathcal{U}/\mathcal{B},b}$. It thus only remains to see that localizing at
  $W'_{\mathcal{U}/\mathcal{B}}$ is the same as localizing at
  $W_{\mathcal{U}/\mathcal{B}}$, which follows from applying the 2-for-3
  property of localizations using the diagram
  \[
    \begin{tikzcd}
      b \arrow[equals]{r} \arrow{d}{f} & b \arrow{d}{f'}
      \arrow[equals]{r} & b \arrow[equals]{d} \\
      a & a' \arrow{l}{\alpha} & b \arrow{l}{f'}
    \end{tikzcd}
  \]
  in $\Tw^{r}(\mathcal{B})$.
\end{proof}

\begin{cor}\label{unfloc}
  The functor $\overline{\mathfrak{c}}_{\mathcal{U}} \colon
  \Unf(\mathcal{E}) \to \mathcal{E}$ exhibits $\mathcal{E}$ as the
  localization of $\Unf(\mathcal{E})$ at the class
  $\overline{W}_{\mathcal{E}/\mathcal{B}}$.
\end{cor}
\begin{proof}
  It follows from \cite{HinichLoc}*{Proposition 2.1.4} and its proof
  that if $\mathcal{X} \to \mathcal{Y}$ is a cocartesian fibration and
  $\eta \colon \mathcal{Y}' \to \mathcal{Y}$ exhibits $\mathcal{Y}$ as the
  localization of $\mathcal{Y}'$ at the morphisms in $W$, then the
  canonical functor $\mathcal{X}' \to \mathcal{X}$ from the pullback
  $\mathcal{X}'$ of $\mathcal{X}$ along $\eta$ exhibits $\mathcal{X}$
  as the localization of $\mathcal{X}'$ at the cocartesian morphisms
  over $W$.
\end{proof}

From this the universal property of
$\widetilde{\mathcal{E}}_{\mathcal{X}}$ now follows using \cref{secisloc}:
\begin{cor}\label{cor:EXuniv}
  For any \icat{} $\mathcal{X}$, there is a natural equivalence
  \[ \Map_{/\mathcal{B}}(\mathcal{C},
    \widetilde{\mathcal{E}}_{\mathcal{X}}) \simeq \Map(\mathcal{C}
    \times_{\mathcal{B}} \mathcal{E}, \mathcal{X}),\]
  natural in $\mathcal{C} \in \Cat_{\infty/\mathcal{B}}$. \qed
\end{cor}

\begin{remark}\label{rmk:unfftr}
  We remark briefly on the naturality of the construction in the
  cocartesian fibration. Suppose then that we have a commutative
  triangle
  \[
    \begin{tikzcd}
      \mathcal{E}_{0} \arrow{rr}{\phi} \arrow{dr}[below left]{p_{0}} &
      & \mathcal{E}_{1} \arrow{dl}{p_{1}} \\
       & \mathcal{U},
    \end{tikzcd}
  \]
  and a cartesian fibration $q \colon \mathcal{U} \to \mathcal{B}$. We
  can replace the triangle by a cocartesian fibration $p \colon \mathcal{E} \to
  \mathcal{U} \times \Delta^{1}$ and then apply the construction to
  $p$ and $q \times \Delta^{1} \colon \mathcal{U} \times \Delta^{1}
  \to \mathcal{B} \times \Delta^{1}$ to obtain for any cocomplete
  \icat{} $\mathcal{X}$ a cocartesian fibration
  $\widetilde{\mathcal{E}}_{\mathcal{X}} \to \mathcal{B} \times
  \Delta^{1}$. By naturality of the construction the fibres at $i =
  0,1$ identify with $\widetilde{\mathcal{E}}_{i,\mathcal{X}} \to
  \mathcal{B}$ and so this cocartesian fibration corresponds to a
  commutative triangle
  \[
    \begin{tikzcd}
      \widetilde{\mathcal{E}}_{0,\mathcal{X}} \arrow{rr} \arrow{dr} &
      & \widetilde{\mathcal{E}}_{1,\mathcal{X}} \arrow{dl} \\
       & \mathcal{B},
    \end{tikzcd}
  \]
  where the horizontal functor preserves cocartesian morphisms. On the
  fibre over $b \in \mathcal{B}$, this is the functor
  $\Fun(\mathcal{E}_{0,b}, \mathcal{X}) \to \Fun(\mathcal{E}_{1,b},
  \mathcal{X})$ given by left Kan extension along
  $\phi_{b}\colon \mathcal{E}_{0,b} \to \mathcal{E}_{1,b}$. Making the
  same construction with $\Delta^{1}$ replaced by $\Delta^{n}$ for all
  $n$ it is easy to see that we get a functor
  $\widetilde{(\blank)}_{\mathcal{X}}$ from (small) cocartesian
  fibrations over $\mathcal{B}$ to (large) cocartesian fibrations over
  $\mathcal{B}$.
\end{remark}

\subsection{The Day Convolution Double
  $\infty$-Category}\label{sec:day}
We now apply the construction of the previous subsection to obtain
the Day convolution for a double \icat{}. First we need some notation:
\begin{defn}
Let $\bbS^{n}$ denote the partially ordered set of pairs of integers
$(i,j)$, $0 \leq i \leq j \leq n$, with $(i,j) \leq (i',j')$ if $i
\leq i' \leq j' \leq j$. This determines a functor $\bbS^{\bullet}
\colon \simp \to \Cat$ by taking $\phi \colon [n] \to [m]$ to the
functor $\bbS^{n} \to \bbS^{m}$ that sends $(i,j)$ to $(\phi(i),
\phi(j))$; we write $\hbbS \to \Dop$ for the cartesian fibration
corresponding to this functor. We can also define a functor $\Pi
\colon \hbbS \to \Dop$ by sending $([n], (i,j))$ to $[j-i]$, with a
map $([n], (i,j)) \to ([m], (i',j'))$, which corresponds to a map
$\phi \colon [m] \to [n]$ in $\simp$ such that $(i,j) \leq
(\phi(i'),\phi(j'))$ in $\bbS^{m}$, to the map $[j'-i'] \to [j-i]$
obtained by restricting $\phi$ to a map $\{i',i'+1,\ldots,j'\} \to
\{i,i+1,\ldots,j\}$. 
\end{defn}

\begin{remark}
  We can identify $\bbS^{n}$ with $\Tw^{r}(\Delta^{n})$ and with
  $\simp^{\txt{int},\op}_{/[n]}$.   
\end{remark}

\begin{defn}
  Let $\mathcal{M} \to \Dop$ be a double \icat{}. Then the base change
  \[\bbSM := \mathcal{M} \times_{\Dop} \hbbS \to
  \hbbS\] along the functor $\Pi$ is a cocartesian
  fibration.  Applying the unfolding construction of the previous
  subsection to this together with the cartesian fibration $\hbbS \to
  \Dop$, we get a new cocartesian fibration
  $\txt{Unf}(\bbSM) \to \Tw^{r}(\Dop)$, corresponding
  to a functor $\mathfrak{U}_{\bbSM} \colon \Dop \to
  \Span(\CatI)$, from which we obtain another cocartesian fibration
  \[ \widehat{\mathcal{M}}^{+}_{\mathcal{X}} :=
    \widetilde{\bbSM}_{\mathcal{X}} \to \Dop,\]
  where $\mathcal{X}$ is any cocomplete \icat{}.
\end{defn}

\begin{remark}\label{rmk:hMXdesc}
  The cocartesian fibration $\widehat{\mathcal{M}}^{+}_{\mathcal{X}}
  \to \Dop$ corresponds to a functor $\Dop \to \CatI$ that takes $[n]$
  to $\Fun(\mathcal{M} \times_{\Dop} \bbS^{n}, \mathcal{X})$ and a
  morphism $\phi \colon [m] \to [n]$ in $\simp$ to
  \[ \Fun(\mathcal{M} \times_{\Dop} \bbS^{n}, \mathcal{X}) \to
    \Fun(\phi^{*}(\mathcal{M} \times_{\Dop} \bbS^{n}), \mathcal{X})
    \to \Fun(\mathcal{M} \times_{\Dop} \bbS^{m}, \mathcal{X})\] where
  $\phi^{*}(\mathcal{M} \times_{\Dop} \bbS^{n})$ is equivalently the
  pullback $\mathcal{M} \times_{\Dop} \bbS^{m}$ along the composite
  $\bbS^{m} \xto{\phi} \bbS^{n} \to \Dop$, the
  first functor is given by composition with the induced functor
  $\mathcal{M} \times_{\Dop} \bbS^{n} \to \phi^{*}(\mathcal{M}
  \times_{\Dop} \bbS^{n})$, and the second by left Kan extension along
  $\phi^{*}(\mathcal{M} \times_{\Dop} \bbS^{n}) \to \mathcal{M}
  \times_{\Dop} \bbS^{m}$. Since both $\phi^{*}(\mathcal{M}
  \times_{\Dop} \bbS^{n})$ and $\mathcal{M}
  \times_{\Dop} \bbS^{m}$ are cocartesian fibrations over $\bbS^{m}$,
  and the functor preserves cocartesian morphisms, \cref{lem:cocartlke}
  implies that this left Kan extension is given at $(x \in
  \mathcal{M}_{j-i}, (i,j) \in \bbS^{m})$ by the colimit over 
  \[\mathcal{M}_{\phi(j)-\phi(i)/x} := \mathcal{M}_{\phi(j)-\phi(i)}
    \times_{\mathcal{M}_{j-i}} \mathcal{M}_{j-i/x},\]
  where the functor $\mathcal{M}_{\phi(j)-\phi(i)} \to
  \mathcal{M}_{j-i}$ arises from the cocartesian morphisms over the
  restriction of $\phi$ to an (active) morphism $[j-i] \cong
  \{i,i+1,\ldots,j\} \to \{\phi(i),\phi(i)+1,\ldots,\phi(j)\} \cong
  [\phi(j)-\phi(i)]$. Note in particular that this is the identity if
  $\phi$ is an inert morphism.
\end{remark}

\begin{remark}\label{rmk:d1desc}
  Continuing from \cref{rmk:hMXdesc}, let us spell out the description
  a little further in
  the case of $d_{1} \colon [1] \to [2]$: A functor $F \colon \mathcal{M}
  \times_{\Dop} \bbS^{1} \to \mathcal{X}$ assigns to every $x \in
  \mathcal{M}_{1}$ a span
  \[ F(d_{1,!}x) \from F(x) \to F(d_{0,!}x)\]
  in $\mathcal{X}$, while a functor $G \colon \mathcal{M}
  \times_{\Dop} \bbS^{2} \to \mathcal{X}$ assigns to $(x,y) \in
  \mathcal{M}_{2} \simeq \mathcal{M}_{1} \times_{\mathcal{M}_{0}}
  \mathcal{M}_{1}$ a diagram
  \[
    \begin{tikzcd}
      {} & & G(x,y) \arrow{dl} \arrow{dr} \\
      & G(x) \arrow{dl} \arrow{dr}& & G(y) \arrow{dl} \arrow{dr} \\
      G(d_{0,!}x) & & G(d_{1,!}x) \simeq G(d_{0,!}y) & & G(d_{1,!}y).
    \end{tikzcd}
  \]
  In the first step, this is taken to the functor
  $d_{1}^{*}(\mathcal{M} \times_{\Dop} \bbS^{2}) \to \mathcal{X}$ that
  to $(x,y)$ assigns the span
  \[ G(d_{0,!}x) \from G(x,y) \to G(d_{1,!}y),\]
  which then in the second step is taken to the functor $\mathcal{M}
  \times_{\Dop} \bbS^{1} \to \mathcal{X}$ that to $x \in
  \mathcal{M}_{1}$ assigns
  \[ G(d_{0,!}x) \from \colim_{(x',y') \in \mathcal{M}_{2/x}} G(x',y')
    \to G(d_{1,!}x),\]
  where the colimit is over the fibre product $\mathcal{M}_{2/x} :=
  \mathcal{M}_{2} \times_{\mathcal{M}_{1}} \mathcal{M}_{1/x}$ defined
  using $d_{1,!} \colon \mathcal{M}_{2}\to \mathcal{M}_{1}$.
\end{remark}

\begin{remark}
  Note that if $\phi \colon [m] \to [n]$ is an inert morphism in
  $\simp$, then $\phi^{*}(\mathcal{M} \times_{\Dop} \bbS^{n}) \to
  \mathcal{M} \times_{\Dop} \bbS^{m}$ is an equivalence, and so the
  functor
  \[ \Fun(\mathcal{M} \times_{\Dop} \bbS^{n}, \mathcal{X}) \to
    \Fun(\mathcal{M} \times_{\Dop} \bbS^{m}, \mathcal{X}) \]
  is simply given by restriction.
\end{remark}

We now define a subobject of $\widehat{\mathcal{M}}_{\mathcal{X}}^{+}$
that, in good cases, will be a double \icat{}:
\begin{defn}
  Let $\bbL^{n}$ be the full subcategory of $\bbS^{n}$ on the objects
  $(i,j)$ such that $j-i \leq 1$. We define $\bbL
  \mathcal{M}_{n}$ to be the pullback
  \[
    \nlcsquare{\bbL\mathcal{M}_{n}}{\bbSM_{n}}{\bbL^{n}}{\bbS^{n},}
  \]
  and write $\widehat{\mathcal{M}}_{\mathcal{X}}$ for the full
  subcategory of $\widehat{\mathcal{M}}_{\mathcal{X}}^{+}$ spanned by the
  functors $\bbSM_{n} \to \mathcal{X}$ that are right Kan
  extensions of their restrictions to $\bbL\mathcal{M}_{n}$, for all $n$.
\end{defn}

\begin{lemma}\label{lem:bbLM}
  A functor $F \colon \bbSM_{n} \to \mathcal{X}$ is a right Kan
  extension of its restriction to $\bbL\mathcal{M}_{n}$ \IFF{} for
  every object $((i,j) \in \bbS^{n}, x_{ij} \in \mathcal{M}_{j-i})$ and
  every integer $k$, $i \leq k \leq j$, the commutative square
  \[
\nlcsquare{F(x_{ij})}{F(x_{ik})}{F(x_{kj})}{F(x_{kk})}
\]
is cartesian, where $x_{ij} \to x_{i'j'}$ denotes the cocartesian
morphism over $(i,j) \to (i',j')$.
\end{lemma}
\begin{proof}
  For $\xi = ((i,j) \in \bbS^{n}, x_{ij} \in \mathcal{M}_{j-i})$, let
  $\bbL\mathcal{M}_{n,\xi/}'$ denote the full subcategory of
  $\bbL\mathcal{M}_{n,\xi/} := \bbL\mathcal{M}_{n} \times_{\bbSM_{n}}
  \bbSM_{n/\xi}$ spanned by the cocartesian morphisms. Then
  $\bbL\mathcal{M}_{n,\xi/}' \simeq \bbL^{n}_{(i,j)/} \simeq
  \bbL^{j-i}$ and the inclusion $\bbL\mathcal{M}_{n,\xi/}'
  \hookrightarrow \bbL\mathcal{M}_{n,\xi/}$ is coinitial. It follows
  that $F$ is a right Kan extension of its restriction to
  $\bbL\mathcal{M}_{n}$ \IFF{} $F(x_{ij})$ is the limit over
  $F(x_{i'j'})$ with $(i',j') \in \bbL^{n}_{(i,j)/}$. Depicting the
  category $\bbL^{n}_{(i,j)/}$ as
  \[
    \begin{tikzcd}
      {} & (i,i+1) \arrow{dl} \arrow{dr} & & \cdots \arrow{dl}
      \arrow{dr} & & (j-1,j) \arrow{dl} \arrow{dr} \\
      (i,i) & & (i+1,i+1) & & (j-1,j-1) & & (j,j),
    \end{tikzcd}
  \]
  we see that the condition is that the map
  \[ F(x_{ij})\to F(x_{i,i+1}) \times_{F(x_{(i+1)(i+1)})} \cdots
    \times_{F(x_{(j-1)(j-1)})} F(x_{(j-1)j}) \]
  must be an equivalence. Inducting on $j-i$ and using that limits
  commute, we see that this condition holds for all $(i,j)$ \IFF{} the
  given commutative squares are all cartesian.
\end{proof}

\begin{defn}
  Let $\mathcal{M}$ be a double \icat{} and $\mathcal{X}$ a cocomplete
  \icat{} with pullbacks. Fix an active morphism
  $\phi \colon [m] \to [n]$ in $\simp$, an object $x_{0m} \in
  \mathcal{M}_{m}$, and an integer $0 \leq k \leq m$.
 Suppose given functors
  \[ F_{0k} \colon \mathcal{M}_{\phi(k)/x_{0k}} \to \mathcal{X},\]
  \[ F_{kk} \colon \mathcal{M}_{0/x_{kk}} \to \mathcal{X},\]
  \[ F_{km} \colon \mathcal{M}_{n-\phi(k)/x_{km}} \to \mathcal{X},\]
  where $x_{0m} \to x_{ij}$ is the cocartesian morphism over
  $\{i,\ldots,j\} \hookrightarrow [m]$ and
  $\mathcal{M}_{\phi(j)-\phi(i)/x_{ij}}$ is defined using the
  restriction of $\phi$ to an active
  morphism
  \[[j-i] \cong \{i,i+1,\ldots,j\} \to
    \{\phi(j),\phi(j)+1,\ldots,\phi(i)\} \cong [\phi(j)-\phi(i)],\]
  together with natural transformations
  $F_{0k} \to F_{kk}|_{\mathcal{M}_{\phi(k)/x_{0k}}}$,
  $F_{km} \to F_{kk}|_{\mathcal{M}_{n-\phi(k)/x_{km}}}$. Then define
  $F_{0m} \colon \mathcal{M}_{n/x_{0m}} \to \mathcal{X}$ as the fibre
  product $F_{0k} \times_{F_{kk}} F_{km}$ using these natural
  transformations and the equivalence
  $\mathcal{M}_{n/x_{0m}} \simeq \mathcal{M}_{\phi(k)/x_{0k}}
  \times_{\mathcal{M}_{0/x_{kk}}} \mathcal{M}_{n-\phi(k)/x_{km}}$. We
  then have an equivalence
    \[ \colim_{\mathcal{M}_{n/x_{0m}}} F_{0m} \simeq
    \colim_{\mathcal{M}_{\phi(k)/x_{0k}} \times_{\mathcal{M}_{0/x_{kk}}}
      \mathcal{M}_{n-\phi(k)/x_{km}}} F_{0k} \times_{F_{kk}} F_{km}.\]
  This induces a canonical distributivity morphism
  \[ \colim_{\mathcal{M}_{n/x_{0m}}} F|_{0m} \to
    \left(\colim_{\mathcal{M}_{\phi(k)/x_{0k}}}
      F_{0k} \right)
    \times_{\left(\colim_{\mathcal{M}_{0/x_{kk}}}
        F_{kk} \right)}
    \left(\colim_{\mathcal{M}_{n-\phi(k)/x_{km}}}
      F_{km}\right).
  \]
  We say $\mathcal{M}$ is \emph{$\mathcal{X}$-admissible} if this
  morphism is always an equivalence, \ie{} if colimits over these slices
  of $\mathcal{M}$  distribute over pullbacks.
\end{defn}

\begin{remark}
  Given a functor $F \colon \bbSM_{n} \to \mathcal{X}$ that
  is right Kan extended from $\bbL\mathcal{M}_{n}$,
  the condition of $\mathcal{X}$-admissibility implies an equivalence  
  \[ \colim_{\mathcal{M}_{n/x_{0m}}} F|_{\mathcal{M}_{n/x_{0m}}} \to
    \left(\colim_{\mathcal{M}_{\phi(k)/x_{0k}}}
      F|_{\mathcal{M}_{\phi(k)/x_{0k}}} \right)
    \times_{\left(\colim_{\mathcal{M}_{0/x_{kk}}}
        F|_{\mathcal{M}_{0/x_{kk}}} \right)}
    \left(\colim_{\mathcal{M}_{n-\phi(k)/x_{km}}}
      F|_{\mathcal{M}_{n-\phi(k)/x_{km}}}\right).
  \]
\end{remark}

\begin{ex}\label{ex:gpdadm}
  If $\mathcal{M}_{0}$ is an $\infty$-groupoid (in which case we may
  view $\mathcal{M}$ as an $(\infty,2)$-category, in the sense of a
  (not necessarily complete) 2-fold Segal space), then
  $(\mathcal{M}_{0})_{/x_{kk}} \simeq *$, so $\mathcal{M}$ is
  $\mathcal{X}$-admissible provided pullbacks in $\mathcal{X}$
  preserve colimits, \ie{} colimits in $\mathcal{X}$ are
  universal. In particular, $\mathcal{M}$ is $\mathcal{X}$-admissible
  for any $\infty$-topos $\mathcal{X}$.
\end{ex}

\begin{propn}\label{propn:dayconvdouble}
  Suppose $\mathcal{M}$ is an $\mathcal{X}$-admissible double
  \icat{}. Then $\widehat{\mathcal{M}}_{\mathcal{X}}$ is a double
  \icat{}.
\end{propn}
\begin{proof}
  Using the description of $\widehat{\mathcal{M}}_{\mathcal{X}}$ in
  \cref{lem:bbLM} we see that the condition of
  $\mathcal{X}$-admissibility is precisely set up so that for
  $F \in \widehat{\mathcal{M}}_{\mathcal{X},n}$ and
  $\phi \colon [m] \to [n]$ in $\simp$, the cocartesian morphism
  $F \to \phi_{!}F$ in $\widehat{\mathcal{M}}_{\mathcal{X}}^{+}$ lies
  in $\widehat{\mathcal{M}}_{\mathcal{X},m}$. Hence
  $\widehat{\mathcal{M}}_{\mathcal{X}} \to \Dop$ is a cocartesian
  fibration. It remains to check that
  $\widehat{\mathcal{M}}_{\mathcal{X}}$ is a double \icat{}, \ie{}
  that the functor
  \[ \widehat{\mathcal{M}}_{\mathcal{X},n} \to
    \widehat{\mathcal{M}}_{\mathcal{X},1}\times_{\widehat{\mathcal{M}}_{\mathcal{X},0}}
    \cdots \times_{\widehat{\mathcal{M}}_{\mathcal{X},0}}
    \widehat{\mathcal{M}}_{\mathcal{X},1}\]
  is an equivalence. From the definition of
  $\widehat{\mathcal{M}}_{\mathcal{X}}$ it is immediate that we may
  identify $\widehat{\mathcal{M}}_{\mathcal{X},n}$ with
  $\Fun(\bbL\mathcal{M}_{n}, \mathcal{X})$, where
  $\bbL\mathcal{M}_{n}$ is equivalent to the pullback $\mathcal{M}
  \times_{\Dop} \bbL^{n}$. Under this equivalence our functor is 
  given by composition with
  \[ \bbL\mathcal{M}_{1} \amalg_{\bbL\mathcal{M}_{0}} \cdots
    \amalg_{\bbL\mathcal{M}_{0}} \bbL\mathcal{M}_{1} \to
    \bbL\mathcal{M}_{n}.\]
  Since $\mathcal{M} \to \Dop$ is a cocartesian fibration, pullback
  along it preserves colimits, hence this functor is equivalent to
  \[ \mathcal{M} \times_{\Dop} (\bbL^{1} \amalg_{\bbL^{0}} \cdots
    \amalg_{\bbL^{0}} \bbL^{1}) \to \mathcal{M} \times_{\Dop}
    \bbL^{n},\] which is an equivalence by \cite{spans}*{Proposition
    5.13}.
\end{proof}

In the case where $\mathcal{X}$ is $\mathcal{S}$, we can give a more
explicit description of the \icat{}
$\widehat{\mathcal{M}}_{\mathcal{S}}(F,G)$ of horizontal morphisms
from $F$ to $G$, \ie{} the fibre of
$\widehat{\mathcal{M}}_{\mathcal{S},1} \to
\widehat{\mathcal{M}}_{\mathcal{S},0}^{\times 2}$ at $(F,G)$:
\begin{notation}
  Given $F,G \colon \mathcal{M}_{0} \to \mathcal{S}$ with
  corresponding left fibrations $\mathcal{F}, \mathcal{G} \to
  \mathcal{M}_{0}$, let 
  $\mathcal{M}_{1,F,G} \to \mathcal{M}_{1}$ be the left fibration
  defined by the pullback
  \[
    \begin{tikzcd}
      \mathcal{M}_{1,F,G} \arrow{r} \arrow{d} & \mathcal{F} \times
      \mathcal{G} \arrow{d} \\
      \mathcal{M}_{1} \arrow{r} & \mathcal{M}_{0} \times \mathcal{M}_{0}.
    \end{tikzcd}
  \]
  This left fibration corresponds to the functor
  \[ \mathcal{M}_{1} \to \mathcal{M}_{0}^{\times 2} \xto{F \times G} \mathcal{S}.\]
\end{notation}

\begin{lemma}\label{lem:bbSM1po}
  The \icat{} $\bbSM_{1}$ is equivalent to the pushout
  $(\mathcal{M}_{1} \times \bbS^{1}) \amalg_{\mathcal{M}_{1} \amalg
    \mathcal{M}_{1}} (\mathcal{M}_{0} \amalg \mathcal{M}_{0})$.
\end{lemma}
\begin{proof}
  By definition $\bbSM_{1}$ is the pullback $\bbS^{1} \times_{\Dop}
  \mathcal{M}$. The category $\bbS^{1}$ can be written as a pushout
  $\Delta^{1} \amalg_{\{0\}} \Delta^{1}$, and since pullbacks along
  cocartesian fibrations preserve 
  colimits we get a decomposition $\bbSM_{1} \simeq
  (\Delta^{1}\times_{\Dop} \mathcal{M}) \amalg_{\mathcal{M}_{1}}
  (\Delta^{1} \times_{\Dop} \mathcal{M})$. By \cite{freepres}*{Lemma
    3.8} (which summarizes results of \cite{HTT}*{\S 3.2.2}) for any
  cocartesian fibration $\mathcal{E} \to \Delta^{1}$ there is a
  pushout $\mathcal{E} \simeq \mathcal{E}_{0}\times \Delta^{1}
  \amalg_{\mathcal{E}_{0} \times\{1\}} \mathcal{E}_{1}$. Applying this
  we get an equivalence
  \[ \bbSM_{1} \simeq (\mathcal{M}_{1} \times \Delta^{1}
    \amalg_{\mathcal{M}_{1} \times \{1\}} \mathcal{M}_{0})
    \amalg_{\mathcal{M}_{1} \times \{0\}} (\mathcal{M}_{1} \times \Delta^{1}
    \amalg_{\mathcal{M}_{1} \times \{1\}} \mathcal{M}_{0}),\]
  which we can rewrite as the desired expression.
\end{proof}

\begin{propn}\label{propn:M1FG}
  Given $F,G \colon \mathcal{M}_{0} \to \mathcal{S}$, the \icat{}
  $\widehat{\mathcal{M}}_{\mathcal{S}}(F,G)$ is equivalent to the
  functor \icat{} 
  $\Fun(\mathcal{M}_{1,F,G}, \mathcal{S})$.
\end{propn}
\begin{proof}
  By \cref{lem:bbSM1po} we have a pullback square
  \[
    \begin{tikzcd}
      \Fun(\bbSM_{1}, \mathcal{S}) \arrow{r} \arrow{d} &
      \Fun(\mathcal{M}_{1} \times \bbS^{1}, \mathcal{S}) \arrow{d} \\
      \Fun(\mathcal{M}_{0}, \mathcal{S})^{\times 2} \arrow{r} &
      \Fun(\mathcal{M}_{1}, \mathcal{S})^{\times 2}.
    \end{tikzcd}
  \]
  Now since $\bbS^{1} \simeq \{0,1\}^{\triangleleft}$, for any \icat{}
  $\mathcal{C}$ we can
  identify the fibre of 
  \[ \Fun(\bbS^{1}, \mathcal{C}) \to \mathcal{C}^{\times 2}\]
  at $x,y$ with $\mathcal{C}_{/x,y} := \mathcal{C}_{/p}$ for the
  diagram $p \colon \{0,1\} \to \mathcal{C}$ picking out $x$ and $y$,
  and if $\mathcal{C}$ has products then $\mathcal{C}_{/x,y} \simeq
  \mathcal{C}_{/x \times y}$ by the universal property of the
  limit. We therefore have an equivalence between the fibre of $\Fun(\mathcal{M}_{1}
  \times \bbS^{1}, \mathcal{S}) \to \Fun(\mathcal{M}_{1},
  \mathcal{S})^{\times 2}$ at $(\alpha, \beta)$ and
  $\Fun(\mathcal{M}_{1}, \mathcal{S})_{/\alpha \times \beta}$.
  Now \cite{freepres}*{Proposition 9.7} describes this as
  $\Fun(\mathcal{E}, \mathcal{S})$ where $\mathcal{E} \to
  \mathcal{M}_{1}$ is the left fibration for the functor $\alpha
  \times \beta$. Together with the pullback square above, this
  identifies the fibre of $\Fun(\bbSM_{1}, \mathcal{S}) \to
  \Fun(\mathcal{M}_{0}, \mathcal{S})^{\times 2}$ at $F, G$ with
  $\Fun(\mathcal{M}_{1,F,G}, \mathcal{S})$, as required.  
\end{proof}

\begin{remark}\label{rmk:horizref}
  Let us reformulate the description of the horizontal composition
  from \cref{rmk:d1desc} in terms of our description of horizontal
  morphisms: Suppose $\Phi$ is a horizontal morphism from $F$ to $G$
  and $\Psi$ is a horizontal morphism from $G$ to $H$, so that we may
  view $\Phi$ as a functor $\mathcal{M}_{1,F,G} \to \mathcal{S}$ and
  $\Psi$ as a functor $\mathcal{M}_{1,G,H} \to \mathcal{S}$, then
  their composite is the functor $\mathcal{M}_{1,F,H} \to \mathcal{S}$
  given by
  \[ (x, p \in F(x_{00}), q \in H(x_{11})) \mapsto \colim_{y \in
      \mathcal{M}_{2/x}} \colim_{p' \in F(y_{00})_{p}} \colim_{q' \in
      H(y_{22})_{q}} \colim_{r \in H(y_{11})} \Phi(y_{01}, p', r)
    \times \Psi(y_{12}, r, q'),\] where $y_{ij}$ and $x_{ij}$ denote
  the cocartesian pushforwards of $y$ and $x$ along the inert
  inclusion of $[j-i]$ as $\{i,i+1,\ldots,j\}$.

  Note that if the \icat{} $\mathcal{M}_{0}$ is an \igpd{}, the formula above
  simplifies to
  \[ (x, p \in F(x_{00}), q \in H(x_{11})) \mapsto \colim_{y \in
      \mathcal{M}_{2/x}} \colim_{r \in H(y_{11})} \Phi(y_{01}, p, r)
    \times \Psi(y_{12}, r, q).\]
\end{remark}

\begin{remark}\label{rmk:ncomp}
  More generally, we can write the composition of $n$ horizontal
  morphisms
  \[ \Phi_{i} \colon \mathcal{M}_{1,F_{i-1},F_{i}} \to
  \mathcal{S}, \qquad i = 1,\ldots,n,\] as the functor $\mathcal{M}_{1,F_{0},F_{n}} \to
  \mathcal{S}$ given by
  \[ (x, p \in F_{0}(x_{00}), q \in F_{n}(x_{11})) \mapsto \colim_{y \in
      \mathcal{M}_{n/x}} \colim_{(t_{0},\ldots,t_{n}) \in F(y)_{p,q}
    }
    \Phi_{1}(y_{01},t_{0},t_{1}) \times \cdots \times
    \Phi_{n}(y_{(n-1)n}, t_{n-1},t_{n}),\]
  where $F(y)_{p,q} := F(y_{00})_{p} \times F(y_{11}) \times \cdots \times
  F(y_{(n-1)(n-1)}) \times F(y_{nn})_{q}$. 
\end{remark}

\begin{remark}
  Let $\widehat{\mathcal{M}}^{\otimes}_{\mathcal{X},*}$ denote the full
  subcategory of $\widehat{\mathcal{M}}_{\mathcal{X}}$
  spanned by functors $F \colon \bbSM_{n} \to \mathcal{X}$ such that for all
  maps $[0] \to [n]$ in $\simp$ the composite $\bbSM_{0} \to \bbSM_{n}
  \to \mathcal{X}$ is constant at the terminal object of
  $\mathcal{X}$. For such $F$ we have an equivalence
    \[ \colim_{\mathcal{M}_{n/x_{0m}}} F|_{\mathcal{M}_{n/x_{0m}}} \simeq
    \colim_{\mathcal{M}_{\phi(k)/x_{0k}} \times_{\mathcal{M}_{0/x_{kk}}}
      \mathcal{M}_{n-\phi(k)/x_{km}}} F|_{\mathcal{M}_{\phi(k)/x_{0k}}}
    \times F|_{\mathcal{M}_{n-\phi(k)/x_{km}}},\]
  giving a canonical morphism
  \[ \colim_{\mathcal{M}_{n/x_{0m}}} F|_{\mathcal{M}_{n/x_{0m}}} \to
    \left(\colim_{\mathcal{M}_{\phi(k)/x_{0k}}}
      F|_{\mathcal{M}_{\phi(k)/x_{0k}}}\right)
    \times
    \left(\colim_{\mathcal{M}_{n-\phi(k)/x_{km}}}
      F|_{\mathcal{M}_{n-\phi(k)/x_{km}}}\right).
  \]
  $\widehat{\mathcal{M}}_{\mathcal{X},*}^{\otimes}$
  is a monoidal \icat{} under the weaker hypothesis that this morphism
  is an
  equivalence. This holds, in particular, if $\mathcal{X}$ has
  finite products that commute with colimits in each variable, and the
  functors
  \[\mathcal{M}_{n/x_{0m}} \to  \mathcal{M}_{\phi(k)/x_{0k}} \times
    \mathcal{M}_{n-\phi(k)/x_{km}}\] 
  are cofinal.
\end{remark}

\begin{remark}\label{rmk:Dayconvftr}
  From \cref{rmk:unfftr} we see that any morphism of double \icats{}
  $\mu \colon \mathcal{N} \to \mathcal{M}$ (\ie{} a functor over $\Dop$ that
  preserves cocartesian morphisms) induces a functor
  $\widehat{\mu} \colon \widehat{\mathcal{N}}^{+}_{\mathcal{X}} \to
  \widehat{\mathcal{M}}_{\mathcal{X}}^{+}$ (given by taking left Kan extensions). However, even if both
  $\mathcal{N}$ and $\mathcal{M}$ are $\mathcal{X}$-admissible this
  does not necessarily restrict to a functor
  $\widehat{\mathcal{N}}_{\mathcal{X}} \to
  \widehat{\mathcal{M}}_{\mathcal{X}}$. Using \cref{lem:cocartlke} we
  see that this happens precisely when for every $x \in
  \mathcal{M}_{2} \simeq \mathcal{M}_{1} \times_{\mathcal{M}_{0}}
  \mathcal{M}_{1}$ the natural distributivity morphism
\[ \colim_{\mathcal{N}_{1/x_{01}} 
    \times_{\mathcal{N}_{0/x_{11}}}
    \mathcal{N}_{1/x_{12}}} F_{01} \times_{F_{11}} F_{12} \to
  \colim_{\mathcal{N}_{1/x_{01}}} F_{01} 
  \times_{\colim_{\mathcal{N}_{0/x_{11}}} F_{11}}
  \colim_{\mathcal{N}_{1/x_{12}}} F_{12}\]
 is an equivalence for all functors $F \colon \mathcal{N}_{2} \to
 \mathcal{S}$ in $\widehat{\mathcal{N}}_{\mathcal{X}}$. This happens,
 for instance, if $\mathcal{N}_{0/x_{11}}$ is a contractible
 $\infty$-groupoid and colimits in $\mathcal{X}$ are universal, or if
 $\mathcal{X}$ is $\mathcal{S}$ and
 all \icats{} of the form $\mathcal{N}_{0/x_{11}}$ and
 $\mathcal{N}_{1/x_{01}}$ admit cofinal functors from \igpds{} (since
 colimits indexed by spaces distribute over limits by
 \cite{patterns1}*{Corollary 7.17}).
\end{remark}

\subsection{The Universal Property}\label{sec:dayunivp}
Suppose $\mathcal{M}$ is an $\mathcal{X}$-admissible double \icat{}
and $\mathcal{O}$ is a generalized non-symmetric \iopd{}. Our goal in
this subsection is to show that there is a natural equivalence
\[ \Alg_{\mathcal{O}}(\widehat{\mathcal{M}}_{\mathcal{X}}) \simeq
  \Seg_{\mathcal{O} \times_{\Dop} \mathcal{M}}(\mathcal{X}). \]

\begin{remark}
  We already know from \cref{cor:EXuniv} that
  $\widehat{\mathcal{M}}_{\mathcal{X}}^{+} \to \Dop$ has the universal
  property that there is a natural equivalence
  \[ \Map_{/\Dop}(\mathcal{I},
    \widehat{\mathcal{M}}_{\mathcal{X}}^{+}) \simeq \Map(\mathcal{I}
    \times_{\Dop} \bbSM, \mathcal{X}).\] Our first goal is to reduce
  the right-hand side to functors from
  $\mathcal{I} \times_{\Dop} \mathcal{M}$ by a further localization.
\end{remark}

\begin{remark}
  For $\mathcal{M} := \Dop$, the universal property we want was proved
  as \cite{spanalg}*{Corollary 3.11}, and our proof will build on the
  constructions made there. 
\end{remark}

\begin{notation}
  We recall some notation from \cite{spanalg}:
  \begin{itemize}
  \item The functor $\Pi \colon \hbbS \to \Dop$ has a section $\Psi
    \colon \Dop \to \hbbS$, taking $[n]$ to $([n], (0,n))$; there is
    also a natural transformation $\eta \colon \id_{\hbbS} \to
    \Psi\Pi$ given at $([n], (i,j))$ by the map $([n], (i,j)) \to
    ([j-i], (0,j-i))$ lying over the inert map $\rho_{ij} \colon
    [j-i] \to [n]$. Note that $\Pi \eta \simeq \id_{\Pi}$.
  \item We also have $p\Psi \cong \id_{\Dop}$, where $p$ is the
    Cartesian fibration $\hbbS \to \Dop$, and a natural transformation
    $\epsilon \colon \Psi p \to \id_{\hbbS}$ given at $([n], (i,j))$
    by the natural maps $([n], (0,n)) \to ([n], (i,j))$.
  \item We let $I$ denote the set of cartesian morphisms in $\hbbS$ that
    lie over inert morphisms in $\Dop$.
  \end{itemize}
\end{notation}

The functor $\Pi$ exhibits $\Dop$ as the localization of $\hbbS$ at $I$ by
\cite{spanalg}*{Proposition 3.8}. We can extend this as follows:
\begin{propn}
  The projection $\Pi_{\mathcal{M}} \colon \bbSM \to \mathcal{M}$ exhibits $\mathcal{M}$ as
  the localization of $\bbSM$ at the set $I_{\mathcal{M}}$ of cocartesian
  morphisms that lie over $I$.
\end{propn}
\begin{proof}
  Let $W$ be the class of morphisms in $\hbbS$ that are mapped to
  isomorphisms (\ie{} identities) by $\Pi$ and let $W_{\mathcal{M}}$
  be the class of morphisms in $\bbSM$ that are mapped to equivalences
  by $\Pi_{\mathcal{M}}$; then $W_{\mathcal{M}}$ is precisely the
  class of cocartesian morphisms over $W$.

  The section $\Psi$ pulls back to a section
  $\Psi_{\mathcal{M}} \colon \mathcal{M} \to \bbSM$, and since $\Pi
  \eta \simeq \id_{\Pi}$ the transformation $\eta$ pulls back to a
  natural transformation $\eta_{\mathcal{M}}\colon \id_{\bbSM} \to
  \Psi_{\mathcal{M}}\Pi_{\mathcal{M}}$. Then $\eta_{\mathcal{M}}$ is
  componentwise given by cocartesian morphisms in $\bbSM$ that lie
  over morphisms in $W$, and so this data becomes an equivalence of
  \icats{} after localizing at $W$. In particular,
   $\Pi_{\mathcal{M}}$ exhibits $\mathcal{M}$ as the
  localization of $\bbSM$ at $W_{\mathcal{M}}$. It thus only remains
  to see that the localization at $I_{\mathcal{M}}$ is the same as the
  localization at $W_{\mathcal{M}}$. Since the morphisms involved are
  cocartesian, this follows from the same 2-of-3 argument as in the
  proof of \cite{spanalg}*{Proposition 3.8}.  
\end{proof}

\begin{propn}\label{propn:inertloc}
  Suppose $f \colon \mathcal{I} \to \Dop$ is any functor such that
  $\mathcal{I}$ has $f$-cocartesian morphisms over inert maps in
  $\Dop$. Then there is a functor
  $\overline{\Pi} \colon \mathcal{I} \times_{\Dop} \bbSM \to
  \mathcal{I} \times_{\Dop} \mathcal{M}$ lying over $\Pi$, which
  exhibits $\mathcal{I} \times_{\Dop} \mathcal{M}$ as the localization
  at the set $I_{\mathcal{I},\mathcal{M}}$ of morphisms whose image in
  $\Dop$ is inert, whose image in $\mathcal{I}$ is cocartesian, and
  whose image in $\hbbS$ is cartesian (and whose image in
  $\mathcal{M}$ is therefore an equivalence).
\end{propn}
\begin{proof}
  As the proof of \cite{spanalg}*{Proposition 3.9}, using the lifts
  defined in the proof of the previous proposition.
\end{proof}

\begin{cor}\label{cor:M+eq}
  Suppose $f \colon \mathcal{I} \to \Dop$ is any functor such that
  $\mathcal{I}$ has $f$-cocartesian morphisms over inert maps in
  $\Dop$. Then there is an equivalence
  \[ \Map_{/\Dop}^{\txt{int}}(\mathcal{I},
    \widehat{\mathcal{M}}^{+}_{\mathcal{X}}) 
    \simeq \Map(\mathcal{I} \times_{\Dop} \mathcal{M}, \mathcal{X}),\]
  natural in $\mathcal{I}$, where $\Map_{/\Dop}^{\txt{int}}(\mathcal{I},
  \widehat{\mathcal{M}}^{+}_{\mathcal{X}})$ denotes the subspace of
  $\Map_{/\Dop}(\mathcal{I},
  \widehat{\mathcal{M}}^{+}_{\mathcal{X}})$ consisting of functors
  that preserve the cocartesian morphisms over inert maps in $\Dop$.
\end{cor}
\begin{proof}
  Using \cref{propn:inertloc} we may identify
  $\Map(\mathcal{I} \times_{\Dop} \mathcal{M}, \mathcal{X})$ with the
  subspace of
  $\Map(\mathcal{I} \times_{\Dop}\bbSM, \mathcal{X}) \simeq
  \Map_{/\Dop}(\mathcal{I}, \widehat{\mathcal{M}}^{+}_{\mathcal{X}})$
  consisting of functors that take morphisms in
  $I_{\mathcal{I},\mathcal{M}}$ to equivalences. Unwinding the
  definitions, we see that (as a cocartesian morphism in
  $\widehat{\mathcal{M}}^{+}_{\mathcal{X}}$ over an inert morphism
  does not involve a left Kan extension) these precisely correspond to
  the functors that preserve cocartesian morphisms over inert
  morphisms in $\Dop$.
\end{proof}

Replacing $\mathcal{I}$ by $\mathcal{I} \times \Delta^{\bullet}$ this induces an equivalence
\[ \Fun^{\txt{int}}_{/\Dop}(\mathcal{I},
  \widehat{\mathcal{M}}^{+}_{\mathcal{X}}) \simeq \Fun(\mathcal{I}
  \times_{\Dop} \mathcal{M}, \mathcal{X}).\]
Restricting further to functors with value in
$\widehat{\mathcal{M}}_{\mathcal{X}}$, we get:
\begin{cor}\label{cor:Meq}
  Suppose $\mathcal{O}$ is a \gnsiopd{}. Then there is an equivalence
  \[ \Fun_{/\Dop}^{\txt{int}}(\mathcal{O},
    \widehat{\mathcal{M}}_{\mathcal{X}}) 
    \simeq \Seg_{\mathcal{O} \times_{\Dop} \mathcal{M}}(\mathcal{X}),\]
  natural in $\mathcal{O}$. \qed
\end{cor}

\subsection{Day Convolution Monoidal Structures}\label{sec:daymon}
Our goal in this subsection is to show that
$\widehat{\mathcal{M}}_{\mathcal{X}}$ induces a family of monoidal
\icats{}, and in some cases (including for associative algebras)
algebras in $\widehat{\mathcal{M}}_{\mathcal{X}}$ are algebras in
these monoidal \icats{}. This will follow from a general observation
about \emph{framed} double \icats{}, which we will consider after
some simple observations about algebras in double \icats{}:
\begin{defn}
  For any \icat{} $\mathcal{C}$, let $\Dop_{\mathcal{C}} \to \Dop$ be
  the terminal double \icat{} with $\mathcal{C}$ as its fibre at
  $[0]$. This is defined as the cocartesian fibration for the functor
  $\Dop \to \CatI$ obtained as the right Kan extension along
  $\{[0]\} \hookrightarrow \Dop$ of the functor $\{[0]\} \to \CatI$
  with value $\mathcal{C}$. Thus
  $(\Dop_{\mathcal{C}})_{n} \simeq \mathcal{C}^{\times (n+1)}$ with
  cocartesian morphisms over face maps given by projections and those
  over degeneracies given by diagonals. Note that there is in
  particular a canonical functor
  $\mathcal{C} \times \Dop \to \Dop_{\mathcal{C}}$, given fibrewise by
  the diagonal $\mathcal{C} \to \mathcal{C}^{\times (n+1)}$.
\end{defn}

\begin{defn}
  Given a double \icat{} $\mathcal{M}$, define $\mathcal{M}^{\otimes}$
  as the pullback
  \[
    \begin{tikzcd}
      \mathcal{M}^{\otimes} \arrow{r} \arrow{d} & \mathcal{M}
      \arrow{d} \\
      \mathcal{M}_{0} \times \Dop \arrow{r} & \Dop_{\mathcal{M}_{0}}.
    \end{tikzcd}
  \]
  This is a pullback of cocartesian fibrations over $\Dop$ along
  functors that preserve cocartesian morphisms, hence
  $\mathcal{M}^{\otimes} \to \Dop$ is again a cocartesian fibration.
\end{defn}

\begin{propn}
  Suppose $\mathcal{O}$ is a \gnsiopd{} such that
  the inclusion
  $\mathcal{O}_{0} \to \mathcal{O}$ induces an equivalence
  $\mathcal{O}_{0} \isoto \mathcal{O}[I^{-1}]$ where $I$ is the class
  of inert morphisms in $\mathcal{O}$. Let $\mathcal{M}$
  be a double \icat{}. Then
  \[ \Alg_{\mathcal{O}}(\mathcal{M}^{\otimes}) \to
    \Alg_{\mathcal{O}}(\mathcal{M}) \]
  is an equivalence.
\end{propn}
\begin{proof}
  The double \icat{} $\mathcal{M}^{\otimes}$ is defined by a pullback
  square of generalized non-symmetric \iopds{}, so we have a pullback
  square
  \[
    \nlcsquare{\Alg_{\mathcal{O}}(\mathcal{M}^{\otimes})}{\Alg_{\mathcal{O}}(\mathcal{M})}{\Alg_{\mathcal{O}}(\mathcal{M}_{0}
      \times \Dop)}{\Alg_{\mathcal{O}}(\Dop_{\mathcal{M}_{0}}).}
  \]
  It therefore suffices to show that
  $\Alg_{\mathcal{O}}(\mathcal{M}_{0} \times \Dop) \to
  \Alg_{\mathcal{O}}(\Dop_{\mathcal{M}_{0}})$ is an equivalence if
  $\mathcal{O}$ satisfies the assumptions.

  By definition $\Alg_{\mathcal{O}}(\Dop_{\mathcal{M}_{0}})$ is
  equivalent to $\Fun(\mathcal{O}_{0}, \mathcal{M}_{0})$, while we may
  identify $\Alg_{\mathcal{O}}(\mathcal{M}_{0} \times \Dop)$ with the
  \icat{} of functors $\mathcal{O} \to \mathcal{M}_{0}$ that take
  inert morphisms to equivalences. We therefore have an equivalence if
  $\mathcal{O}_{0} \to \mathcal{O}[I^{-1}]$ is an equivalence.
\end{proof}
\begin{cor}
  Let $\mathcal{M}$
  be a double \icat{} and $\mathcal{O}$ a \nsiopd{} such that the \icat{}
  $\mathcal{O}$ is weakly contractible. Then
  \[ \Alg_{\mathcal{O}}(\mathcal{M}^{\otimes}) \to
    \Alg_{\mathcal{O}}(\mathcal{M}) \]
  is an equivalence.
\end{cor}
\begin{proof}
  We must show that $\mathcal{O}[I^{-1}]$ is contractible, where $I$
  is the class of inert morphisms. Since $\mathcal{O}$ is
  weakly contractible, it suffices to check that inverting the inert
  morphisms in $\mathcal{O}$ amounts to inverting all morphisms. To
  see this we first observe that for any map $\phi \colon [n] \to [m]$ in $\simp$, we
  have a commutative triangle
  \[
    \begin{tikzcd}
      {} & {[0]} \arrow{dl} \arrow{dr} \\
      {[n]} \arrow{rr}{\phi} & & {[m]}
    \end{tikzcd}
  \]
  where the maps from $[0]$ are inert. Given a morphism $X \to Y$ in
  $\mathcal{O}$ we therefore have a commutative triangle
  \[
    \begin{tikzcd}
      X \arrow{dr} \arrow{rr} & & Y \arrow{dl} \\ 
       & (),
    \end{tikzcd}
  \]
  where $()$ denotes the unique object of $\mathcal{O}_{0}$ and the
  diagonal morphisms are cocartesian morphisms over the (inert)
  morphisms from $[0]$ in the first triangle.
\end{proof}

Since $\Dop$ is weakly contractible, as a special case we have:
\begin{cor}
  Let $\mathcal{M}$
  be a double \icat{}. Then
  \[ \Alg_{\Dop}(\mathcal{M}^{\otimes}) \to
    \Alg_{\Dop}(\mathcal{M}) \]
  is an equivalence. \qed
\end{cor}

\begin{defn}
  A double \icat{} $\mathcal{M}$ is \emph{framed} if the functor
  $(d_{1,!},d_{0,!}) \colon \mathcal{M}_{1} \to
  \mathcal{M}_{0}^{\times 2}$ is a cartesian fibration.
\end{defn}
\begin{remark}
  By \cite{polynomial}*{Proposition A.4.4}, $(d_{1,!},d_{0,!})$ is a
  cartesian fibration \IFF{} it is a cocartesian fibration, and this
  is also equivalent to the existence of ``companions and conjoints''
  in (the homotopy double category of) $\mathcal{M}$.
\end{remark}

\begin{propn}\label{propn:Motimesmonoidal}
  Suppose $\mathcal{M}$ is a framed double \icat{}. Then
  $\mathcal{M}^{\otimes} \to \mathcal{M}_{0}$ is a cartesian
  fibration, and corresponds to a functor from $\mathcal{M}_{0}$ to
  monoidal \icats{} and lax monoidal functors.
\end{propn}
\begin{proof}
  We apply the dual of \cite{cois}*{Lemma A.1.10} to
  $\mathcal{M}^{\otimes} \to \mathcal{M}_{0} \times \Dop$ to conclude
  that $\mathcal{M}^{\otimes}\to \mathcal{M}_{0}$ is a cartesian
  fibration. We know that $\mathcal{M}^{\otimes} \to \Dop$ is a
  cocartesian fibration, so it remain to check that for $[n] \in \Dop$
  the functor $\mathcal{M}^{\otimes}_{n} \to \mathcal{M}_{0}$ is a
  cartesian fibration. This functor lives in a pullback square
  \[
    \begin{tikzcd}
      \mathcal{M}^{\otimes}_{n} \arrow{r} \arrow{d} & \mathcal{M}_{n}
      \arrow{d} \\
      \mathcal{M}_{0} \arrow{r} & \mathcal{M}_{0}^{\times (n+1)}.
    \end{tikzcd}
  \]
  Here the right vertical map is equivalent to the iterated fibre product
  \[ \mathcal{M}_{1} \times_{\mathcal{M}_{0}} \cdots
    \times_{\mathcal{M}_{0}} \mathcal{M}_{1} \to
    (\mathcal{M}_{0}^{\times 2}) \times_{\mathcal{M}_{0}} \cdots
    \times_{\mathcal{M}_{0}} (\mathcal{M}_{0}^{\times 2}),\]
  and so is a cartesian fibration since $\mathcal{M}$ is framed.
  It follows that $\mathcal{M}^{\otimes}_{n} \to \mathcal{M}_{0}$ is a
  cartesian fibration, and hence we can conclude that
  $\mathcal{M}^{\otimes} \to \mathcal{M}_{0}$ is a cartesian
  fibration and that cartesian morphisms lie over equivalences in $\Dop$.

  The fibre $\mathcal{M}^{\otimes}_{X} \to \Dop$ at $X \in
  \mathcal{M}_{0}$ is a monoidal \icat{}, so it remains to check that
  the functor $f^{*} \colon \mathcal{M}^{\otimes}_{Y} \to
  \mathcal{M}^{\otimes}_{X}$ corresponding to the cartesian morphisms
  over $f \colon X \to Y$ in $\mathcal{M}_{0}$ is lax monoidal, \ie{}
  preserves those cocartesian morphisms that lie over inert morphisms
  in $\Dop$.

  Since the cartesian morphisms lie over identities in $\Dop$, it is
  equivalent to check that for every inert morphism $\phi \colon [m]
  \to [n]$, the functor $\mathcal{M}^{\otimes}_{n} \to
  \mathcal{M}^{\otimes}_{m}$, given by cocartesian pushforward along
  $\phi$, preserves cartesian morphisms over $\mathcal{M}_{0}$. To see
  this it suffices to consider the outer face maps $[n-1] \to [n]$ (as
  any inert morphism is a composite of these). In this case we have a
  commutative diagram
  \[
    \begin{tikzcd}
      \mathcal{M}^{\otimes}_{n} \arrow{rr} \arrow{dr} \arrow{dd} & &
      \mathcal{M}_{n} \arrow{rr} \arrow{dr} \arrow{dd} & & \mathcal{M}_{1}
      \arrow{dr} \arrow{dd} \\
      & \mathcal{M}^{\otimes}_{n-1} \arrow{rr}  \arrow{dd} & &
      \mathcal{M}_{n-1} \arrow{rr} \arrow{dd} & & \mathcal{M}_{0}
      \arrow[equals]{dd} \\
      \mathcal{M}_{0}\arrow{rr} \arrow[equals]{dr} & &
      \mathcal{M}_{0}^{n+1} \arrow{rr} \arrow{dr} & &
      \mathcal{M}_{0}^{2} \arrow{dr} \\
      & \mathcal{M}_{0} \arrow{rr} & & \mathcal{M}_{0}^{n} \arrow{rr}
      & & \mathcal{M}_{0}.
    \end{tikzcd}
  \]
  Here all the vertical morphisms are cartesian fibrations, and in the
  left-hand cube the front and back faces are cartesian. It
  therefore suffices to show that
  $\mathcal{M}_{n}\to \mathcal{M}_{n-1}$ takes cartesian morphisms
  over $\mathcal{M}_{0}^{n+1}$ to cartesian morphisms over
  $\mathcal{M}_{0}^{n}$.
  
  In the right-hand cube the top and bottom faces are cartesian, \ie{}
  $\mathcal{M}_{n} \to \mathcal{M}_{0}^{n+1}$ is a fibre product of
  cartesian fibrations, and both morphisms to $\mathcal{M}_{0}$
  preserve cartesian morphisms (since all morphisms are cartesian for
  the identity functor). A morphism in $\mathcal{M}_{n}$ is hence
  cartesian \IFF{} its images in both $\mathcal{M}_{n-1}$ and
  $\mathcal{M}_{1}$ are cartesian, and in particular the functor to
  $\mathcal{M}_{n-1}$ preserves cartesian morphisms, as required.
\end{proof}

\begin{remark}
  The underlying \icat{} of the monoidal \icat{}
  $\mathcal{M}^{\otimes}_{X}$ is the \icat{} $\mathcal{M}(X,X)$ of
  horizontal endomorphisms of $X$, and the monoidal structure is
  given by horizontal composition. Since $\mathcal{M}$ is framed, a
  vertical morphism $f \colon X \to Y$ gives rise to two horizontal
  morphisms, say $f^{\triangleright} := (f,\id_{Y})^{*}\bbone_{Y}$
  from $X$ to $Y$ and $f^{\triangleleft} := (\id_{Y},f)^{*}\bbone_{Y}$ from $Y$ to $X$,
  and $f^{\triangleright}$ is left adjoint to $f^{\triangleleft}$. The underlying functor of
  the lax monoidal functor $\mathcal{M}^{\otimes}_{Y} \to
  \mathcal{M}^{\otimes}_{X}$ is the functor
  $\mathcal{M}(Y,Y) \to \mathcal{M}(X,X)$ given by
  $\Phi \mapsto f^{\triangleright} \odot_{Y} \Phi
  \odot_{Y}f^{\triangleleft}$, where we use $\odot_{Y}$ for horizontal
  composition over $Y$ as in \cref{not:odotcomp}. The lax monoidal
  structure comes from the unit transformation (note the non-standard
  order of composition)
  \[ f^{\triangleright} \odot_{Y} \Phi \odot_{Y}f^{\triangleleft} \odot_{X} f^{\triangleright} \odot_{Y}
    \Psi \odot_{Y}f^{\triangleleft} \to f^{\triangleright} \odot_{Y} \Phi \odot_{Y}
    \Psi \odot_{Y}f^{\triangleleft}.\]
\end{remark}

\begin{cor}\label{cor:framedalgfib}
  Suppose  $\mathcal{M}$
  is a framed double \icat{} and $\mathcal{O}$ is a \nsiopd{} such that
  the \icat{} $\mathcal{O}$ is weakly contractible.
  Then the restriction functor
  \[ \Alg_{\mathcal{O}}(\mathcal{M}) \to  \Fun(\mathcal{O}_{0},
    \mathcal{M}_{0}) \simeq \mathcal{M}_{0}\]
  is a cartesian fibration, corresponding to the functor
  $\mathcal{M}_{0}^{\op} \to \CatI$ taking $X \in \mathcal{M}_{0}$ to
  $\Alg_{\mathcal{O}}(\mathcal{M}^{\otimes}_{X})$.
\end{cor}
\begin{proof}
  Observe that we have a commutative diagram
  \[
    \begin{tikzcd}
      \Fun_{/\Dop}(\mathcal{O}, \mathcal{M}^{\otimes}) \arrow{r}
      \arrow{d} &
      \Fun(\mathcal{O}, \mathcal{M}^{\otimes})
       \arrow{d} \\
      \Fun(\mathcal{O}, \mathcal{M}_{0}) \arrow{r} \arrow{d} & \Fun(\mathcal{O},
      \mathcal{M}_{0} \times \Dop) \arrow{d} \\
      * \arrow{r} & \Fun(\mathcal{O}, \Dop),
    \end{tikzcd}
  \]
  where the bottom and outer squares are clearly cartesian. Hence the
  top square is also cartesian, and here \cite{HTT}*{Proposition
    3.1.2.1} implies that all but the top left vertex
  are cartesian fibrations over
  $\Fun(\mathcal{O}, \mathcal{M}_{0})$ and the morphisms to
  $\Fun(\mathcal{O}, \mathcal{M}_{0} \times \Dop)$
  preserve cartesian morphisms. Hence $\Fun_{/\Dop}(\mathcal{O},
  \mathcal{M}^{\otimes}) \to \Fun(\mathcal{O}, \mathcal{M}_{0})$ is a
  cartesian fibration.

  Since $\mathcal{O}$ is weakly contractible, the inclusion
  $\mathcal{M}_{0}\to \Fun(\mathcal{O}, \mathcal{M}_{0})$ of the
  constant functors is a full subcategory. Let
  $\Fun_{/\Dop}'(\mathcal{O}, \mathcal{M}^{\otimes}) \to
  \mathcal{M}_{0}$ denote the pullback of $\Fun_{/\Dop}(\mathcal{O},
  \mathcal{M}^{\otimes})$ along this inclusion; then
  $\Fun_{/\Dop}'(\mathcal{O}, \mathcal{M}^{\otimes})$ is a full
  subcategory of $\Fun_{/\Dop}(\mathcal{O}, \mathcal{M}^{\otimes})$
  that contains $\Alg_{\mathcal{O}}(\mathcal{M}^{\otimes}$. The
  projection $\Fun_{/\Dop}'(\mathcal{O}, \mathcal{M}^{\otimes}) \to
  \mathcal{M}_{0}$ is a cartesian fibration, so to show that
  $\Alg_{\mathcal{O}}(\mathcal{M}^{\otimes}) \to \mathcal{M}_{0}$ is a
  cartesian fibration it suffices to check that for every cartesian
  morphism in $\Fun_{/\Dop}'(\mathcal{O}, \mathcal{M}^{\otimes})$ whose
  target is an $\mathcal{O}$-algebra, its source is also an
  $\mathcal{O}$-algebra; this follows from
  \cref{propn:Motimesmonoidal}, which shows that cartesian morphisms
  preserve cocartesian morphisms over inert maps in $\Dop$. Since
  $\Alg_{\mathcal{O}}(\blank)$ preserves pullbacks, the fibre at $x
  \in \mathcal{M}_{0}$ can be identified with
  $\Alg_{\mathcal{O}}(\mathcal{M}_{x}^{\otimes})$, and from the
  description of the cartesian morphisms in \cite{HTT}*{Proposition
    3.1.2.1} it follows that for $f \colon y \to x$ in
  $\mathcal{M}_{0}$ the functor
$\Alg_{\mathcal{O}}(\mathcal{M}_{x}^{\otimes}) \to
\Alg_{\mathcal{O}}(\mathcal{M}_{y}^{\otimes})$ is given by composition
with the lax monoidal functor $\mathcal{M}_{x}^{\otimes} \to
\mathcal{M}_{y}^{\otimes}$ arising from the cartesian morphisms over
$f$ in $\mathcal{M}^{\otimes}$.
\end{proof}

\begin{lemma}
  If $\mathcal{M}$ is an $\mathcal{X}$-admissible double \icat{}
  where $\mathcal{X}$ is cocomplete, then
  $\widehat{\mathcal{M}}_{\mathcal{X}}$ is framed.
\end{lemma}
\begin{proof}
  It suffices to show that the source-and-target projection
  $\widehat{\mathcal{M}}_{\mathcal{X},1} \to
  \widehat{\mathcal{M}}_{\mathcal{X},0}^{\times 2}$ is a cocartesian
  fibration. This is the functor
  \[ i^{*} \colon \Fun(\bbSM_{1}, \mathcal{X}) \to \Fun(\mathcal{M}_{0},
    \mathcal{X})^{\times 2}\]
  given by composition with the inclusion $i \colon \mathcal{M}_{0} \amalg \mathcal{M}_{0}
  \to \mathcal{M} \times_{\Dop} \bbS^{1}$. To see that this is a
  cocartesian fibration we use the criterion of
  \cite{nmorita}*{Corollary 4.52}.

  First observe that $i^{*}$ has a left adjoint $i_{!}$,
  given by left Kan extension along $i$. Note that for $X \in
  \mathcal{M}_{1} \subseteq \bbSM_{1}$ the \icat{} $(\mathcal{M}_{0}
  \amalg \mathcal{M}_{0})_{/X}$ is empty, and so for $F \colon
  \mathcal{M}_{0}\amalg \mathcal{M}_{0} \to \mathcal{X}$ we have
  $i_{!}F(X) \simeq \emptyset$. Given $G \colon \bbSM_{1}\to
  \mathcal{X}$ and $\phi \colon i^{*}G \to F$ consider the pushout
  square
  \[
    \begin{tikzcd}
      i_{!}i^{*}G \arrow{d} \arrow{r} & i_{!}F \arrow{d} \\
      G \arrow{r} & G'.
    \end{tikzcd}
  \]
  Since pushouts in $\Fun(\bbSM_{1}, \mathcal{X})$ are computed
  pointwise, we see that $F \isoto i^{*}i_{!}F \isoto i^{*}G'$ is an
  equivalence, which is what we need in order to apply
  \cite{nmorita}*{Corollary 4.52} to conclude that $i^{*}$ is a
  cocartesian fibration (with $G \to G'$ being the cocartesian
  morphism over $i^{*}G \to F$).
\end{proof}

\begin{remark}\label{rmk:framedcart}
  Suppose $\mathcal{M}$ is an $\mathcal{S}$-admissible double
  \icat{}. Given natural transformations $\phi \colon F \to F'$,
  $\gamma \colon G \to G'$ of
  functors $\mathcal{M}_{0} \to \mathcal{S}$ and $\Phi \colon
  \mathcal{M}_{1,F,G} \to \mathcal{S}$, the description of the
  cocartesian pushforward of
  $\Phi$ along $(\phi,\gamma)$ above amounts to this being given by
  the left Kan extension along the induced functor
  $\mathcal{M}_{1,F,G} \to \mathcal{M}_{1,F',G'}$, \ie{}
  \[ (\phi,\gamma)_{!}\Phi(x,p',q') \simeq \colim_{p \in
      F(x_{00})_{p'}, q \in G(x_{11})_{q'}} \Phi(x,p,q).\]
  It follows that the cartesian pullback of $\Psi \colon
  \mathcal{M}_{1,F',G'}\to \mathcal{S}$ is given by composition with
  this functor.
\end{remark}

Applying \cref{cor:framedalgfib} to
$\widehat{\mathcal{M}}_{\mathcal{X}}$, we now get:
\begin{cor}\label{cor:dayconvmonalg}
  Let $\mathcal{O}$ be a non-symmetric \iopd{} such that the \icat{}
  $\mathcal{O}$ is weakly contractible and $\mathcal{M}$ an
  $\mathcal{X}$-admissible double \icat{} where $\mathcal{X}$ is
  cocomplete. Then
  \[\Alg_{\mathcal{O}}(\widehat{\mathcal{M}}_{\mathcal{X}}) \to
    \Fun(\mathcal{M}_{0}, \mathcal{S}) \] is a cartesian fibration,
  corresponding to the functor
  \[ X \in \Fun(\mathcal{M}_{0}, \mathcal{S}) \mapsto
    \Alg_{\mathcal{O}}(\widehat{\mathcal{M}}_{\mathcal{X}}(X,X)) \]
  arising from the family of monoidal \icats{}
  $\widehat{\mathcal{M}}_{\mathcal{X},X}^{\otimes}$. \qed
\end{cor}
\begin{remark}\label{rmk:Algfibeq}
  Suppose $\mathcal{M}$ is $\mathcal{S}$-admissible and $\mathcal{O}$
  is as above. Then we can
  use the equivalence
  \[\Alg_{\mathcal{O}}(\widehat{\mathcal{M}}_{\mathcal{X}}) \simeq
  \Seg_{\mathcal{O} \times_{\Dop} \mathcal{M}}(\mathcal{S})\] and the
description of the cartesian fibration from \cref{propn:Segfib} to
conclude that there is fibrewise a natural equivalence
\[ \Alg_{\mathcal{O}}(\widehat{\mathcal{M}}_{\mathcal{S}}(X,X))
  \simeq \Mon_{\mathcal{O} \times_{\Dop}
    \mathcal{M}_{X}}(\mathcal{S}) \simeq \Alg_{\mathcal{O}
    \times_{\Dop} \mathcal{M}_{X}}(\mathcal{S}).\]
\end{remark}

\subsection{Enriched Day Convolution}\label{sec:enrdc}
In this subsection we generalize our construction slightly by showing
that if $\mathcal{M}$ is a double \icat{} and $\mathcal{V}^{\otimes}$
is an $\mathcal{M}$-monoidal \icat{} then in good cases there exists a
double \icat{} $\widehat{\mathcal{M}}_{\mathcal{V}}$ such that we have
a natural equivalence
\[ \Alg_{\mathcal{O}}(\widehat{\mathcal{M}}_{\mathcal{V}}) \simeq
  \Algd_{\mathcal{O} \times_{\Dop} \mathcal{M}}(\mathcal{V}).\] More
generally, an object $\widehat{\mathcal{M}}_{\mathcal{V}}$ with this
property will exist as a \gnsiopd{} that is not a double \icat{}. If
$\mathcal{V}^{\otimes}$ is given by Day convolution, we can obtain
$\widehat{\mathcal{M}}_{\mathcal{V}}$ from the following observation:
\begin{propn}\label{propn:Mmonadm}
  Suppose $\mathcal{M}$ is an $\mathcal{X}$-admissible double \icat{}
  where pullbacks in $\mathcal{X}$ preserve colimits,
  and $\mathcal{U}^{\otimes} \to \mathcal{M}$ is a small
  $\mathcal{M}$-monoidal \icat{}. Then $\mathcal{U}^{\otimes}$ is also
  an $\mathcal{X}$-admissible double \icat{}.
\end{propn}
\begin{proof}
  Given an active morphism $\phi
  \colon [m] \to [n]$ in $\simp$ and an object $u_{0m} \in
  \mathcal{U}^{\otimes}_{m}$ over $x_{0m} \in \mathcal{M}_{m}$,
  the functor $\mathcal{U}^{\otimes}_{n/u_{0m}} \to
  \mathcal{M}_{n/x_{0m}}$ is a cocartesian fibration, whose fibre at
  $y_{0n}$ and $\bar{\phi} \colon y_{0n} \to x_{0m}$ over $\phi$ we
  can identify with the slice
  $(\mathcal{U}^{\otimes}_{n,y_{0n}})_{/u_{0m}}$, defined using the
  cocartesian morphisms over $\bar{\phi}$. For $0
  \leq k \leq m$ we can decompose this as a product
  \[ (\mathcal{U}^{\otimes}_{n,y_{0n}})_{/u_{0m}} \simeq
    (\mathcal{U}^{\otimes}_{\phi(k),y_{0\phi(k)}})_{/u_{0k}} \times
    (\mathcal{U}^{\otimes}_{n-\phi(k),y_{\phi(k)n}})_{/u_{km}}.\]
  Given functors
  \[ F_{0k} \colon \mathcal{U}^{\otimes}_{\phi(k)/u_{0k}} \to \mathcal{X},\]
  \[ F_{kk} \colon \mathcal{U}^{\otimes}_{0/u_{kk}} \simeq \mathcal{M}_{0/x_{kk}} \to \mathcal{X},\]
  \[ F_{km} \colon \mathcal{U}^{\otimes}_{n-\phi(k)/u_{km}} \to \mathcal{X},\]
  together with natural transformations
  $F_{0k} \to F_{kk}|_{\mathcal{U}^{\otimes}_{\phi(k)/u_{0k}}}$,
  $F_{km} \to F_{kk}|_{\mathcal{U}^{\otimes}_{n-\phi(k)/u_{km}}}$,
  we have equivalences
  \[
    \begin{split}
      \colim_{\mathcal{U}^{\otimes}_{n/u_{0m}}}
      F_{0k}\times_{F_{kk}}F_{km}& \simeq
      \colim_{y_{0n} \in \mathcal{M}_{n/x_{0m}}}
      \colim_{(\mathcal{U}^{\otimes}_{n,y_{0n}})_{/u_{0m}}}
      F_{0k}\times_{F_{kk}}F_{km} \\
      & \simeq \colim_{y_{0n} \in \mathcal{M}_{n/x_{0m}}} \left(
        \colim_{(\mathcal{U}^{\otimes}_{\phi(k),y_{0\phi(k)}})_{/u_{0k}}} F_{0k}\right)
      \times_{F_{kk}} \left(
        \colim_{(\mathcal{U}^{\otimes}_{n-\phi(k),y_{\phi(k)n}})_{/u_{km}}}
        F_{km}\right) \\
     & \simeq \left( \colim_{y_{0\phi(k)}\mathcal{M}_{\phi(k)/x_{0k}}}
      \colim_{(\mathcal{U}^{\otimes}_{\phi(k),y_{0\phi(k)}})_{/u_{0k}}}
      F_{0k} \right)  \\
     & \hphantom{\simeq}\times_{\left(\colim_{\mathcal{M}_{0/x_{kk}}}
        F_{kk} \right)}
    \left(\colim_{y_{\phi(k)n} \in \mathcal{M}_{n-\phi(k)/x_{km}}}
      \colim_{(\mathcal{U}^{\otimes}_{n-\phi(k),y_{\phi(k)n}})_{/u_{km}}}
      F_{km}\right)\\
    &   \simeq 
    \left(\colim_{\mathcal{U}^{\otimes}_{\phi(k)/u_{0k}}}
      F_{0k} \right)
    \times_{\left(\colim_{\mathcal{U}^{\otimes}_{0/u_{kk}}}
        F_{kk} \right)}
    \left(\colim_{\mathcal{U}^{\otimes}_{n-\phi(k)/u_{km}}}
      F_{km}\right),
    \end{split}
    \]
    where the second equivalence uses that fibre products in
    $\mathcal{X}$ preserve colimits in each variable and the third
    equivalence uses the $\mathcal{X}$-admissibility of $\mathcal{M}$.
\end{proof}

\begin{cor}\label{cor:AlgdUS}
  Let $\mathcal{M}$ be an $\mathcal{S}$-admissible double \icat{} and
  let $\mathcal{U}^{\otimes}$ be a small $\mathcal{M}$-monoidal
  \icat{}. Then there exists a double \icat{}
  $\widehat{\mathcal{U}^{\otimes}_{\mathcal{S}}}$ such that for any
  \gnsiopd{} $\mathcal{O}$ we have a natural equivalence
  \[ \Alg_{\mathcal{O}}(\widehat{\mathcal{U}^{\otimes}_{\mathcal{S}}})
    \simeq \Seg_{\mathcal{O} \times_{\Dop}
      \mathcal{U}^{\otimes}}(\mathcal{S}) \simeq \Algd_{\mathcal{O}
      \times_{\Dop} \mathcal{M}}(\mathcal{O}
    \times_{\Dop}\mathcal{U}^{\otimes}_{\mathcal{S}}).\]
  Moreover, any $\mathcal{M}$-monoidal functor $\mathcal{U}^{\otimes}
  \to \mathcal{V}^{\otimes}$ induces a morphism of double \icats{}
  $\widehat{\mathcal{U}^{\otimes}_{\mathcal{S}}} \to \widehat{\mathcal{V}^{\otimes}_{\mathcal{S}}}$.
\end{cor}
\begin{proof}
  The only part that is not an immediate consequence of
  \cref{propn:Mmonadm} and our results in the previous subsections is
  the claim about $\mathcal{M}$-monoidal functors, which follows using
  the criterion of \cref{rmk:Dayconvftr}, where we can compute the
  colimits by decomposing them as in the proof of \cref{propn:Mmonadm}.
\end{proof}

\begin{remark}
  Let us describe the double \icat{}
  $\widehat{\mathcal{U}^{\otimes}_{\mathcal{S}}}$ a bit more explicitly.
  The objects of $\widehat{\mathcal{U}^{\otimes}_{\mathcal{S}}}$ are
  functors $\mathcal{U}^{\otimes}_{0} \simeq \mathcal{M}_{0} \to
  \mathcal{S}$, and the vertical morphisms are transformations of such
  functors. If $F,G \colon \mathcal{M}_{0} \to \mathcal{S}$ are
  two objects, then by \cref{propn:M1FG} the \icat{}
  $\widehat{\mathcal{U}^{\otimes}_{\mathcal{S}}}(F,G)$ of horizontal
  morphisms can be identified with
  $\Fun(\mathcal{U}^{\otimes}_{1,F,G}, \mathcal{S})$; here
  $\mathcal{U}^{\otimes}_{1,F,G} \simeq \mathcal{U}^{\otimes}_{1}
  \times_{\mathcal{M}_{1}} \mathcal{M}_{1,F,G}$ where
  $\mathcal{U}^{\otimes}_{1}\to \mathcal{M}_{1}$ is a cocartesian
  fibration. Using \cref{rmk:dayconv1} and
  \cite{freepres}*{Proposition 7.3} we can identify this \icat{} with
  $\Fun_{/\mathcal{M}_{1}}(\mathcal{M}_{1,F,G},
  \mathcal{U}^{\otimes}_{\mathcal{S},1})$. Translating the formula for
  composition of horizontal morphisms from \cref{rmk:horizref}
  in these terms, the composite of $\Phi \colon \mathcal{M}_{1,F,G}
  \to \mathcal{U}^{\otimes}_{\mathcal{S},1}$ and $\Psi \colon
  \mathcal{M}_{1,G,H} \to \mathcal{U}^{\otimes}_{\mathcal{S},1}$ is
  the functor from $\mathcal{M}_{1,F,H}$ given by
  \[ (x, p \in F(x_{00}), q \in H(x_{11})) \mapsto \colim_{(\alpha
      \colon y \to x) \in \mathcal{M}_{2/x}} \colim_{p' \in
      F(y_{00})_{p}} \colim_{q' \in H(y_{22})_{q}} \colim_{r \in
      H(y_{11})} \alpha_{!}(\Phi(y_{01}, p', r), \Psi(y_{12}, r,
    q')),\] where
  $\alpha_{!} \colon \mathcal{U}^{\otimes}_{\mathcal{S},1,y_{01}}
  \times \mathcal{U}^{\otimes}_{\mathcal{S},1,y_{12}} \simeq
  \mathcal{U}^{\otimes}_{\mathcal{S},2,y} \to
  \mathcal{U}^{\otimes}_{\mathcal{S},1,x}$ is the cocartesian
  pushforward along $\alpha$ (given by a left Kan extension along the
  corresponding operation for $\mathcal{U}^{\otimes}$).
\end{remark}

\begin{remark}
  For an $\mathcal{M}$-monoidal \icat{} of the form
  $\mathcal{M} \times_{\Dop} \mathcal{C}^{\otimes}$, where
  $\mathcal{C}^{\otimes} \to \Dop$ is a monoidal \icat{}, the
  description above simplifies considerably: We can identify
  $(\mathcal{M}\times_{\Dop} \mathcal{C}^{\otimes})_{\mathcal{S}}$
  with the pullback $\mathcal{M} \times_{\Dop}
  \mathcal{C}^{\otimes}_{\mathcal{S}}$, so that horizontal morphisms
  reduce to functors $\mathcal{M}_{1,F,G} \to
  \Fun(\mathcal{C},\mathcal{S})$. The composition of $\Phi \colon
  \mathcal{M}_{1,F,G} \to \Fun(\mathcal{C},\mathcal{S})$ and $\Psi
  \colon \mathcal{M}_{1,G,H} \to \Fun(\mathcal{C},\mathcal{S})$ is
  then given by the formula
  \[ (x, p \in F(x_{00}), q \in H(x_{11})) \mapsto \colim_{y \in
      \mathcal{M}_{2/x}} \colim_{p' \in
      F(y_{00})_{p}} \colim_{q' \in H(y_{22})_{q}} \colim_{r \in
      H(y_{11})} \Phi(y_{01}, p', r) \otimes \Psi(y_{12}, r, q'),\]
  where $\otimes$ denotes the Day convolution.
\end{remark}

More generally, we can obtain $\widehat{\mathcal{M}}_{\mathcal{V}}$ by
passing to a larger universe:
\begin{defn}
  Let $\widehat{\mathcal{S}}$ denote the (very large) \icat{} of large
  spaces. If $\mathcal{V}^{\otimes}$ is a locally small (but potentially
  large) $\mathcal{M}$-monoidal \icat{}, let
  $\widehat{\mathcal{M}}_{\mathcal{V}}$ denote the full subcategory of
  $\widehat{\mathcal{V}^{\op,\otimes}_{\widehat{\mathcal{S}}}}$ spanned by
  the objects whose
  \begin{itemize}
  \item inert restrictions to the fibre at $0$ are functors
    $\mathcal{M}_{0} \to \widehat{\mathcal{S}}$ that factor through
    the full subcategory $\mathcal{S}$,
  \item inert restrictions to the fibre at $1$ correspond to functors
    of the form
    \[ \mathcal{M}_{1,F,G} \to
      \mathcal{V}^{\op,\otimes}_{\widehat{\mathcal{S}},1}\]
    that factor through the full subcategory
    $\mathcal{V}^{\otimes}_{1}$ (via the Yoneda embedding
    $\mathcal{V}^{\otimes} \hookrightarrow
    \mathcal{V}^{\op,\otimes}_{\widehat{\mathcal{S}}}$).
  \end{itemize}
\end{defn}

\begin{propn}\label{propn:enrDCopd}
  Let $\mathcal{M}$ be an $\mathcal{S}$-admissible double \icat{} and
  $\mathcal{V}^{\otimes}$ a locally small $\mathcal{M}$-monoidal \icat{}.
  \begin{enumerate}[(i)]
  \item $\widehat{\mathcal{M}}_{\mathcal{V}}$ is a
    \gnsiopd{}.
  \item For any \gnsiopd{} $\mathcal{O}$ we have a natural
    equivalence
    \[ \Alg_{\mathcal{O}}(\widehat{\mathcal{M}}_{\mathcal{V}}) \simeq
      \Algd_{\mathcal{O} \times_{\Dop} \mathcal{M}}(\mathcal{O}
      \times_{\Dop} \mathcal{V}^{\otimes}).\]
  \item An $\mathcal{M}$-monoidal functor $\mathcal{U}^{\otimes} \to
    \mathcal{V}^{\otimes}$ induces a natural morphism of \gnsiopds{}
    $\widehat{\mathcal{M}}_{\mathcal{U}} \to
    \widehat{\mathcal{M}}_{\mathcal{V}}$.
  \item $\widehat{\mathcal{M}}_{\mathcal{V},1}\to
    \widehat{\mathcal{M}}_{\mathcal{V},0}^{\times 2}\simeq
    \Fun(\mathcal{M}_{0}, \mathcal{S})^{\times 2}$ is a cartesian
    fibration, and the inclusion 
    $\widehat{\mathcal{M}}_{\mathcal{V},1} \to
    \widehat{\mathcal{V}^{\op,\otimes}_{\widehat{\mathcal{S}}}}$
    preserves cartesian morphisms.    
  \end{enumerate}
\end{propn}
\begin{proof}
  Part (i) follows from \cref{lem:subgnsiopd}, which also identifies
  $\Alg_{\mathcal{O}}(\widehat{\mathcal{M}}_{\mathcal{V}})$ with the
  full subcategory of $\mathcal{O}$-algebras in
  $\widehat{\mathcal{V}^{\op,\otimes}_{\widehat{\mathcal{S}}}}$ whose
  restrictions to $\mathcal{O}_{0}$ and $\mathcal{O}_{1}$ factor
  through $\mathcal{S} \subseteq \widehat{\mathcal{S}}$ and
  $\mathcal{V}^{\otimes}_{1} \subseteq
  (\widehat{\mathcal{V}^{\op,\otimes}_{\widehat{\mathcal{S}}}})_{1}$. Under
  the equivalence between
  $\Alg_{\mathcal{O}}(\widehat{\mathcal{V}^{\op,\otimes}_{\widehat{\mathcal{S}}}})$ and
  $\widehat{\Algd}_{\mathcal{O} \times_{\Dop} \mathcal{M}}(\mathcal{O}
  \times_{\Dop} \mathcal{V}^{\op,\otimes}_{\widehat{\mathcal{S}}})$
  from \cref{cor:AlgdUS} (where the latter denotes the \icat{} of $\mathcal{O}$-algebroids defined
  using $\widehat{\mathcal{S}}$), this full subcategory corresponds to
  that of $\mathcal{O} \times_{\Dop} \mathcal{M}$-algebroids
  $(\mathcal{O} \times_{\Dop} \mathcal{M})_{X} \to
  \mathcal{V}^{\op,\otimes}_{\widehat{\mathcal{S}}}$ such that $X$ is
  a functor $\mathcal{O}_{0} \times \mathcal{M}_{0} \to \mathcal{S}$ and the functor
  $(\mathcal{O} \times_{\Dop} \mathcal{M})_{X,1} \to
  \mathcal{V}^{\op,\otimes}_{\widehat{\mathcal{S}},1}$ factors through
  $\mathcal{V}^{\otimes}_{1}$; since the Yoneda embedding is a fully
  faithful $\mathcal{O} \times_{\Dop}\mathcal{M}$-monoidal functor,
  this full subcategory is precisely $\Algd_{\mathcal{O} \times_{\Dop}
  \mathcal{M}}(\mathcal{O} \times_{\Dop} \mathcal{V}^{\otimes})$,
which gives (ii). From \cref{cor:AlgdUS} we also know that an
$\mathcal{M}$-monoidal functor induces a morphism of double \icats{} 
$\widehat{\mathcal{U}^{\op,\otimes}_{\widehat{\mathcal{S}}}} \to
\widehat{\mathcal{V}^{\op,\otimes}_{\widehat{\mathcal{S}}}}$; this
evidently takes the full subcategory
$\widehat{\mathcal{M}}_{\mathcal{U}}$ into
$\widehat{\mathcal{M}}_{\mathcal{V}}$, which proves (iii).

We know that
$\widehat{\mathcal{V}^{\op,\otimes}_{\widehat{\mathcal{S}}}}$ is a
framed double \icat{}, so that
$\widehat{\mathcal{V}^{\op,\otimes}_{\widehat{\mathcal{S}},1}} \to
(\widehat{\mathcal{V}^{\op,\otimes}_{\widehat{\mathcal{S}},0}})^{\times
2}$ is a cartesian and cocartesian fibration. From the description of
the cartesian morphisms in \cref{rmk:framedcart} it follows that these
restrict to $\widehat{\mathcal{M}}_{\mathcal{V}}$,
which proves (iv).
\end{proof}

\begin{propn}\label{propn:enrDCloccoc}
  Let $\mathcal{M}$ be an $\mathcal{S}$-admissible double \icat{}.
  \begin{enumerate}[(i)]
  \item   Suppose that
  $\mathcal{V}^{\otimes}$ is a locally small $\mathcal{M}$-monoidal
  \icat{} such that:
  \begin{enumerate}[(1)]
  \item For every $x \in \mathcal{M}_{1}$ the \icat{}
    $\mathcal{V}_{x}$ has colimits indexed by \igpds{} and by the
    \icats{} $\mathcal{M}_{n/y}$ for $y \in \mathcal{M}_{1}$,
  \item These colimits are preserved by the functors
    $f_{!} \colon \mathcal{V}_{x}\to \mathcal{V}_{x'}$ induced by the
    cocartesian morphisms over $f \colon x \to x'$ in
    $\mathcal{M}_{1}$.
  \end{enumerate}
  Then  $\widehat{\mathcal{M}}_{\mathcal{V}} \to \Dop$ is a locally
    cocartesian fibration.
  \item Suppose $\mathcal{V}^{\otimes}$ and $\mathcal{U}^{\otimes}$
    are locally small $\mathcal{M}$-monoidal \icats{} satisfying
    conditions (1) and (2) in (i), and 
    $\mathcal{U}^{\otimes} \to \mathcal{V}^{\otimes}$ is an
    $\mathcal{M}$-monoidal functor such that:
    \begin{itemize}
    \item[(3)] For all $x$ the functor
    $\mathcal{U}_{x} \to \mathcal{V}_{x}$ preserves colimits indexed
    by \igpds{} and by the \icats{} $\mathcal{M}_{n/y}$.
    \end{itemize}
    Then the induced morphism
    $\widehat{\mathcal{M}}_{\mathcal{U}} \to
    \widehat{\mathcal{M}}_{\mathcal{V}}$ of \gnsiopds{} from
    \cref{propn:enrDCopd}(iii) preserves locally cocartesian
    morphisms.
  \end{enumerate}
\end{propn}
\begin{proof}
  We already know  from
  \cref{propn:enrDCopd} that $\widehat{\mathcal{M}}_{\mathcal{V}}$ is a
\gnsiopd{}, so to prove (i) it suffices by \cref{lem:loccoc} to show
there are locally cocartesian morphisms over the active morphisms to
$[1]$. Let $\alpha_{n} \colon [1] \to [n]$ denote the unique active
morphism to $[n]$ in $\simp$. Given an object of
$\widehat{\mathcal{M}}_{\mathcal{V}}$, which we can identify with
$(X_{1},\ldots,X_{n})$ with $X_{i}$ in
$\widehat{\mathcal{M}}_{\mathcal{V}}(F_{i-1},F_{i})$ for
$F_{i} \colon \mathcal{M}_{0}\to \mathcal{S}$, as well as $Y$ in
$\widehat{\mathcal{M}}_{\mathcal{V}}(G,H)$, we have
\[
  \Map_{\widehat{\mathcal{M}}_{\mathcal{V}}}^{\alpha_{n}}((X_{1},\ldots,X_{n}),
  Y) \simeq
  \Map^{\alpha_{n}}_{\widehat{\mathcal{V}^{\op,\otimes}_{\widehat{\mathcal{S}}}}}((X_{1},\ldots,X_{n}),
  Y) \simeq
  \Map_{\widehat{\mathcal{V}^{\op,\otimes}_{\widehat{\mathcal{S}},1}}}(X_{1}\odot_{F_{1}}\cdots
  \odot_{F_{n-1}} X_{n}, Y) \]
Since $\widehat{\mathcal{V}^{\op,\otimes}_{\widehat{\mathcal{S}}}}$ is
  a framed double \icat{}, if we take the fibre of the right-hand side
  over maps $g \colon F_{0}\to G$, $h \colon F_{n} \to G$, we get
  \[ \Map_{\widehat{\mathcal{V}^{\op,\otimes}_{\widehat{\mathcal{S}},1}}}(X_{1}\odot_{F_{1}}\cdots
    \odot_{F_{n-1}} X_{n}, Y)_{(g,h)} \simeq
    \Map_{\Fun(\mathcal{M}_{1,F_{0},F_{n}}, \mathcal{V}^{\op,\otimes}_{\widehat{\mathcal{S}},1})}(X_{1}\odot_{F_{1}}\cdots
    \odot_{F_{n-1}} X_{n}, (g,h)^{*}Y).\]
  Using \cite{freepres}*{Proposition 5.1} and the notation of
  \cref{rmk:ncomp}, we can expand this out as the limit over $(x,p,q) \to (x',p',q') \in
  \Tw^{r}(\mathcal{M}_{1,F_{0},F_{n}})$ of
  \[
    \Map_{\mathcal{V}^{\op,\otimes}_{\widehat{\mathcal{S}},1}}\left(
      \colim_{\alpha \colon y \to x \in \mathcal{M}_{n/x}}
      \colim_{(t_{0},\ldots,t_{n}) \in F(y)_{p,q}}
      \alpha_{!}(X_{1}(y_{01},t_{0},t_{1}),\ldots,X_{n}(y_{(n-1)n},
      t_{n-1},t_{n})), \,\, (g,h)^{*}Y(x',p',q')\right)\]
  \[ \simeq \lim_{\alpha \colon y \to x \in \mathcal{M}_{n/x}^{\op}}
    \lim_{(t_{0},\ldots,t_{n}) \in F(y)_{p,q}}
    \Map_{\mathcal{V}^{\otimes}_{1}}\left(\alpha_{!}(X_{1}(y_{01},t_{0},t_{1}),\ldots,X_{n}(y_{(n-1)n},
      t_{n-1},t_{n})), \,\,  (g,h)^{*}Y(x',p',q')\right)\]
  \[ \simeq \Map_{\mathcal{V}^{\otimes}_{1}}\left(\colim_{\alpha \colon y \to x \in \mathcal{M}_{n/x}}
      \colim_{(t_{0},\ldots,t_{n}) \in F(y)_{p,q}}
      \alpha_{!}(X_{1}(y_{01},t_{0},t_{1}),\ldots,X_{n}(y_{(n-1)n},
      t_{n-1},t_{n})), \,\, (g,h)^{*}Y(x',p',q')\right),\]
  under our assumptions. Let $\alpha_{n,!}(X_{1},\ldots,X_{n})$ denote
  the functor $\mathcal{M}_{1,F_{0},F_{n}} \to
  \mathcal{V}^{\otimes}_{1}$ given by
  \[ (x,p,q) \mapsto \colim_{\alpha \colon y \to x \in
      \mathcal{M}_{n/x}} \colim_{(t_{0},\ldots,t_{n}) \in F(y)_{p,q}}
    \alpha_{!}(X_{1}(y_{01},t_{0},t_{1}),\ldots,X_{n}(y_{(n-1)n},
    t_{n-1},t_{n})),\] then we see using \cref{propn:enrDCopd}(iv)
  that the mapping space we started with is equivalent to
  $\Map_{\widehat{\mathcal{M}}_{\mathcal{V},1}}(\alpha_{n,!}(X_{1},\ldots,X_{n}),
  Y)$, so that
  $(X_{1},\ldots,X_{n}) \to \alpha_{n,!}(X_{1},\ldots,X_{n})$ is
  indeed a locally cocartesian morphism over $\alpha_{n}$, as
  required. Moreover, part (ii) also follows immediately from this
  description of the locally cocartesian morphisms.
\end{proof}

For an arbitrary $\mathcal{S}$-admissible double \icat{} $\mathcal{M}$
it seems extremely awkward to formulate a condition on $\mathcal{V}$
such that $\widehat{\mathcal{M}}_{\mathcal{V}}$ is a double
\icat{}. We therefore content ourselves with the following observation:
\begin{propn}\label{propn:enrDCgpd}
  Let $\mathcal{M}$ be a double \icat{} such that $\mathcal{M}_{0}$ is
  an \igpd{}.
  \begin{enumerate}[(i)]
  \item If $\mathcal{V}^{\otimes}$ is an $\mathcal{M}$-monoidal
    \icat{} that is compatible with colimits indexed by \igpds{} and by
    the \icats{} $\mathcal{M}_{n/x}$ for $x \in \mathcal{M}_{1}$, then
    $\widehat{\mathcal{M}}_{\mathcal{V}}$ is a framed double \icat{}.
  \item If $\mathcal{U}^{\otimes}\to \mathcal{V}^{\otimes}$ is an
    $\mathcal{M}$-monoidal functor between $\mathcal{M}$-monoidal
    \icats{} with colimits as in (i) such that each functor
    $\mathcal{U}_{x} \to \mathcal{V}_{x}$ preserves these colimits,
    then the natural morphism of \gnsiopds{}
    $\widehat{\mathcal{M}}_{\mathcal{U}} \to
    \widehat{\mathcal{M}}_{\mathcal{V}}$ preserves cocartesian
    morphisms.
  \end{enumerate}
\end{propn}

\begin{proof}
  We know from \cref{propn:enrDCloccoc}(i) that
  $\widehat{\mathcal{M}}_{\mathcal{V}} \to \Dop$ in (i) is a locally
  cocartesian fibration. If we can show this is actually a cocartesian
  fibration, then (ii) also follows since \cref{propn:enrDCloccoc}(ii)
  implies the functor preserves locally cocartesian morphisms.

  By \cref{lem:loccoc} to show we have a cocartesian fibration it
  suffices to check that for every active map
  $\phi \colon [2] \to [n]$ and for
  $\Phi \in \widehat{\mathcal{M}}_{\mathcal{V},n}$ the canonical map
  \[  \alpha_{n,!}\Phi = (\phi d_{1})_{!}\Phi \to d_{1,!}\phi_{!}\Phi\]
  is an equivalence.

  For $\Phi$ over $F_{0},\ldots,F_{n} \colon \mathcal{M}_{0} \to
  \mathcal{S}$, the object $\alpha_{n,!}\Phi$ is given at $(x,p,q)$ by 
  \[ \colim_{y \xto{\gamma} x \in \mathcal{M}_{n/x}}
    \colim_{(t_{0},\ldots,t_{n})) \in F(y)_{p,q}}
  \gamma_{!}(\Phi_{1}(y_{01},t_{0},t_{1}), \ldots, 
  \Phi_{n}(y_{(n-1)n},t_{n-1},t_{n})),\]
while if $\phi_{ij}$ denotes the active part of $\phi \circ
\rho_{ij}$, we have
\[
d_{1,!}\phi_{!}\Phi(x,p,q) \simeq   \colim_{z \xto{\zeta} x \in \mathcal{M}_{2/x}} \colim_{t \in
    F_{\phi(1)}(z_{11})} \zeta_{!}( \phi_{01,!}\Phi_{0\phi(1)}(z_{01}, p,
  t), \phi_{12,!}\Phi_{\phi(1)n}(z_{12}, t, q)),\]
where (setting $\ell := \phi(1)$)
\[ \phi_{01,!}\Phi_{0\ell}(z_{01}, p,
  t) \simeq \colim_{u \xto{\delta} z_{01} \in
    \mathcal{M}_{\ell/z_{01}}} \colim_{(t_{0},\ldots,t_{\ell})
      \in F_{0\ell}(u)}
    \delta_{!}(\Phi_{1}(u_{01},t_{0},t_{1}),\ldots,
    \Phi_{\ell}(u_{(\ell-1)\ell}, t_{\ell-1},
    t_{\ell})), \]
\[ \phi_{12,!}\Phi_{\ell n}(z_{12}, t, q) \simeq   
\colim_{v \xto{\epsilon} z_{12} \in
    \mathcal{M}_{n-\ell/z_{12}}} \colim_{(t_{\ell},\ldots,t_{n})
      \in F_{\ell n}(u)}
   \epsilon_{!}(\Phi_{\ell+1}(v_{\ell(\ell+1)},t_{\ell},t_{\ell+1}),\ldots,
    \Phi_{n}(v_{(n-1)n}, t_{n-1}, t_{n})).
\]
 Since $\mathcal{V}^{\otimes}$ is compatible with these colimits, we
 can pass these colimits past $\zeta_{!}$ in the expression for
 $d_{1,!}\phi_{!}\Phi(x,p,q)$, obtaining an expression for this object
 as an iterated colimit of terms of the form
 \[ \zeta_{!}(\delta_{!}(\Phi_{1}(u_{01},t_{0},t_{1}),\ldots,
    \Phi_{\ell}(u_{(\ell-1)\ell}, t_{\ell-1},
    t_{\ell})), 
    \epsilon_{!}(\Phi_{\ell+1}(v_{\ell(\ell+1)},t_{\ell},t_{\ell+1}),\ldots,
    \Phi_{n}(v_{(n-1)n}, t_{n-1}, t_{n}))).\]
  Note that we can rewrite this as
  \[ (\zeta \circ (\delta,\epsilon))_{!}(\Phi_{1}(u_{01},t_{0},t_{1}),
    \ldots, \Phi_{n}(v_{(n-1)n}, t_{n-1}, t_{n}))\] since
  $\mathcal{V}^{\otimes}$ is cocartesian. Thus we can rewrite the
  expression for $d_{1,!}\phi_{!}\Phi(x,p,q)$ as
\[  \colim_{\zeta \in
    \mathcal{M}_{2/x}} \colim_{(\delta, \epsilon) \in \mathcal{M}_{\ell/z_{01}} \times
    \mathcal{M}_{n-\ell/z_{12}}} \colim_{(t_{0},\ldots,t_{n}) \in
    F_{0\ell}(u) \times_{F_{\ell}(z)} F_{\ell n}(v)} (\zeta \circ (\delta,\epsilon))_{!}(\Phi_{1}(u_{01},t_{0},t_{1}),
    \ldots, \Phi_{n}(v_{(n-1)n}, t_{n-1}, t_{n}))
\]
  On the other hand, we can evaluate the colimit over
  $\mathcal{M}_{n/x}$ in the expression for $\alpha_{n,!}\Phi$ by
  first taking a left Kan extension along the functor
  $\phi_{!} \colon \mathcal{M}_{n/x}\to \mathcal{M}_{2/x}$ given by the
  cocartesian morphisms over $\phi$. For a functor $f$ out of
  $\mathcal{M}_{n/x}$ this gives
  \[ \colim_{y \to x \in \mathcal{M}_{n/x}} f \simeq \colim_{z \to x
      \in \mathcal{M}_{2/x}} \colim_{(\mathcal{M}_{n/x})_{/z}} f,\]
  where $(\mathcal{M}_{n/x})_{/z} \simeq \mathcal{M}_{n/z} \simeq
  \mathcal{M}_{\phi(1)/z_{01}}\times_{\mathcal{M}_{0/z_{11}}}\mathcal{M}_{n-\phi(1)/z_{12}}$,
  which is equivalent to $\mathcal{M}_{\phi(1)/z_{01}} \times
  \mathcal{M}_{n-\phi(1)/z_{12}}$ since $\mathcal{M}_{0}$ is an
  \igpd{}. Rewriting our expression for $\alpha_{n,!}\Phi$ using this,
  we get exactly our last formula for $d_{1,!}\phi_{!}\Phi$, as required.
\end{proof}

\begin{remark}\label{rmk:hatMVdesc}
  We can describe the double \icat{}
  $\widehat{\mathcal{M}}_{\mathcal{V}}$ as follows:
  \begin{itemize}
  \item its objects are functors $\mathcal{M}_{0} \to \mathcal{S}$,
    and its vertical morphisms are natural transformations of these,
  \item its horizontal morphisms from $F$ to $G$ are functors
    $\mathcal{M}_{1,F,G} \to \mathcal{V}^{\otimes}_{1}$ over
    $\mathcal{M}_{1}$,
  \item the composite of horizontal morphisms $\Phi \colon
    \mathcal{M}_{1,F,G}  \to \mathcal{V}^{\otimes}_{1}$ and $\Psi
    \colon \mathcal{M}_{1,G,H} \to \mathcal{V}^{\otimes}_{1}$ is the
    functor $\mathcal{M}_{1,F,H} \to \mathcal{V}^{\otimes}_{1}$ given
    by 
    \[ (x \in \mathcal{M}_{1}, p \in F(x_{00}), q\in F(x_{11}))
      \mapsto \colim_{\alpha \colon y \to x \in \mathcal{M}_{2/x}}
      \colim_{t \in H(y_{11})} \alpha_{!}(\Phi(y_{01}, p, t), \Psi(y_{12},t,q)).\]
  \end{itemize}
\end{remark}

\subsection{Example: Enriched $\infty$-Categories as Associative
  Algebras}\label{subsec:enrcat}

In this subsection we illustrate our results on Day convolution by
considering a simple example of our construction: we will give a
description of enriched \icats{} as associative algebras in a family
of monoidal \icats{}. An alternative construction of these monoidal
\icats{} is given in \cite{HinichYoneda}, where this perspective on
enriched \icats{} is developed extensively.

\begin{remark}
  Our construction will extend the following description of ordinary
  enriched categories: If $S$ is a set and $\mathbf{V}$ is a monoidal
  category where the tensor product preserves coproducts in each
  variable, then there is a monoidal structure on
  $\Fun(S \times S, \mathbf{V})$ given by
\[(F \otimes G)(i,k) \cong \coprod_{j \in S} F(i,j) \otimes G(j,k),\]
with unit $\bbone$ the functor \[\bbone(i,j) =
\begin{cases}
  \bbone,& i = j\\
  \emptyset, & i \neq j.
\end{cases}
\]
This is sometimes known as the ``matrix multiplication'' tensor
product, since the formula is a ``categorified'' version of that for
multiplication of matrices. An associative algebra $A$ in the category
$\Fun(S \times S,
\mathbf{V})$ with this tensor product is the same thing as a $\mathbf{V}$-category with
set of objects $S$:
\begin{itemize}
\item The multiplication map $A \otimes A \to A$ supplies composition
  maps $A(i,j) \otimes A(j,k) \to A(i,k)$.
\item The unit map $\bbone \to A$ supplies identity maps
  $\bbone \to A(i,i)$.
\end{itemize}
\end{remark}

Let us consider first the result of applying Day convolution to the
simplest double \icat{}, namely $\Dop$. This is trivially
$\mathcal{X}$-admissible for any \icat{} $\mathcal{X}$ with pullbacks,
and so by \cref{propn:dayconvdouble} there is a double \icat{}
$\widehat{\simp}^{\op}_{\mathcal{X}}$ which by \cref{cor:Meq} has the
universal property that for any \gnsiopd{} $\mathcal{O}$ there is a
natural equivalence
\[ \Alg_{\mathcal{O}}(\widehat{\simp}^{\op}_{\mathcal{X}}) \simeq
  \Seg_{\mathcal{O}}(\mathcal{X}).\]
By construction, the fibre $\widehat{\simp}^{\op}_{\mathcal{X},n}$
is the \icat{} of functors $\bbS^{n} \to \mathcal{X}$ that are right
Kan extended from $\bbL^{n}$. We thus see that:
\begin{itemize}
\item objects of $\widehat{\simp}^{\op}_{\mathcal{X}}$ are objects
  of $\mathcal{X}$,
\item vertical morphisms (morphisms in
  $\widehat{\simp}^{\op}_{\mathcal{X},0}$) are morphisms in
  $\mathcal{X}$,
\item horizontal morphisms (objects in
  $\widehat{\simp}^{\op}_{\mathcal{X},1}$) are \emph{spans} in
  $\mathcal{X}$, \ie{} diagrams of shape
  \[ \bullet \from \bullet \to \bullet, \]
\item squares (morphisms in $\widehat{\simp}^{\op}_{\mathcal{X},1}$)
  are diagrams of shape
  \[
    \begin{tikzcd}
      \bullet \arrow{d} & \bullet \arrow{l} \arrow{r} \arrow{d}& \bullet\arrow{d} \\
      \bullet & \bullet \arrow{l} \arrow{r} & \bullet,
    \end{tikzcd}
  \]
\item composition of horizontal morphisms is given by taking pullbacks.
\end{itemize}
Indeed, the double \icat{}
$\widehat{\simp}^{\op}_{\mathcal{X},n}$ is precisely the double
\icat{} $\txt{SPAN}^{+}(\mathcal{X})$ of spans constructed in
\cite{spans}, and its universal property is that established in
\cite{spanalg}. In particular, we have an equivalence
\[ \Alg_{\Dop}(\widehat{\simp}^{\op}_{\mathcal{X}}) \simeq
  \Seg_{\Dop}(\mathcal{X}),\]
identifying associative algebras in the double \icat{}
$\widehat{\simp}^{\op}_{\mathcal{X}}$ with category objects in
$\mathcal{X}$. Specializing to the \icat{}
$\mathcal{S}$ of spaces, this says that associative algebras in
$\widehat{\simp}^{\op}_{\mathcal{S}}$ are equivalent to \emph{Segal
  spaces}, which describe the algebraic structure of
\icats{}.

By \cref{cor:dayconvmonalg}, the restriction
$\Alg_{\Dop}(\widehat{\simp}^{\op}_{\mathcal{S}}) \to \mathcal{S}$ is
a cartesian fibration, with fibre at a space $X$ given by
$\Alg_{\Dop}(\widehat{\simp}^{\op}_{\mathcal{S}}(X,X))$. Here
$\widehat{\simp}^{\op}_{\mathcal{S}}(X,X)$ is
equivalent to $\mathcal{S}_{/X \times X} \simeq \Fun(X \times X,
\mathcal{S})$. The monoidal structure is given by pullback of spans,
which in terms of functors to $\mathcal{S}$ admits the following description:
\begin{propn}
  For any space $X$ there is a monoidal structure on the \icat{}
  $\Fun(X \times X, \mathcal{S})$ such that
  \begin{enumerate}[(i)]
  \item the tensor product of $F$ and $G$ is given by 
    \[(F \otimes G)(x,x') \simeq \colim_{y \in X} F(x,y) \times
    G(y,x').\]
  \item the unit $\bbone$ is given by 
    \[ \bbone(x,y) \simeq \Map_{X}(x,y) \simeq 
    \begin{cases}
      \emptyset, & x \not\simeq y\\
      \Omega_{x}X, & x \simeq y,
    \end{cases}
    \]
    where $\Map_{X}(x,y)$ is the mapping space in the \igpd{} $X$,
    \ie{} the space of paths from $x$ to $y$ in $X$.
  \item we have
    $\Alg_{\Dop}(\Fun(X \times X, \mathcal{S}))
    \simeq \Seg(\mathcal{S})_{X}$.
  \end{enumerate}
\end{propn}
In other words, \icats{} with space of objects $X$ are associative
algebras in $\Fun(X \times X, \mathcal{S})$ with this monoidal structure.

Now we want to consider the analogue of this result for
\emph{enriched} \icats{}. Propositions~\ref{propn:enrDCopd} and
\ref{propn:enrDCgpd} specialize to give
the following:
\begin{propn}
  Let $\mathcal{C}^{\otimes} \to \Dop$ be a monoidal \icat{}
  compatible with colimits indexed by \igpds{}. Then there is a framed
  double
  \icat{} $\widehat{\Dop_{\mathcal{C}}}$ such that for any \gnsiopd{}
  $\mathcal{O}$ there is an equivalence
  \[\Alg_{\mathcal{O}}(\widehat{\Dop_{\mathcal{C}}})  \simeq
    \Algd_{\mathcal{O}}(\mathcal{O}\times_{\Dop}
    \mathcal{C}^{\otimes}).\]
  A monoidal functor $\mathcal{C}^{\otimes} \to \mathcal{D}^{\otimes}$
  induces a natural morphism of \gnsiopds{}
  $\widehat{\Dop_{\mathcal{C}}} \to \widehat{\Dop_{\mathcal{D}}}$, and
  this preserves cocartesian morphisms if the monoidal functor
  preserves colimits indexed by \igpds{}.
\end{propn}

In particular, we get an equivalence
\[\Alg_{\Dop}(\widehat{\Dop_{\mathcal{C}}})  \simeq
    \Algd_{\Dop}(\mathcal{C}^{\otimes}),\]
where as in \cref{ex:enrcat} the right-hand side is the model of enriched \icats{} considered
in \cite{enr}. Specializing \cref{rmk:hatMVdesc} gives the following
description of the double \icat{} $\widehat{\Dop_{\mathcal{C}}}$:
\begin{itemize}
\item its objects are spaces, and its vertical morphisms are morphisms
  of spaces,
\item a horizontal morphism from $X$ to $Y$ is a functor $X \times Y
  \to \mathcal{C}$,
\item the composite of the horizontal morphisms $\Phi \colon X \times
  Y \to \mathcal{C}$ and $\Psi \colon Y \times Z \to \mathcal{C}$ is
  the functor $X \times Z \to \mathcal{C}$ given by
  \[ (x,z) \mapsto \colim_{y \in Y} \Phi(x,y) \otimes \Psi(y,z).\]
\end{itemize}
From \cref{cor:framedalgfib} we know that the restriction
$\Alg_{\Dop}(\widehat{\Dop_{\mathcal{C}}}) \to \mathcal{S}$ is
a cartesian fibration, with fibre at a space $X$ given by
$\Alg_{\Dop}((\widehat{\Dop_{\mathcal{C}}})(X,X)) \simeq
\Alg_{\Dop_{X}}(\mathcal{C})$. Here the \icat{} 
$(\widehat{\Dop_{\mathcal{C}}})(X,X)$ is equivalent to $\Fun(X \times
X, \mathcal{C})$, giving:
\begin{cor}
  Let $\mathcal{C}$ be a monoidal \icat{} compatible with colimits
  indexed by \igpds{}. Then there is a monoidal structure on the \icat{}
  $\Fun(X \times X, \mathcal{C})$ such that
  \begin{enumerate}[(i)]
  \item the tensor product of $F$ and $G$ is given by 
    \[(F \otimes G)(p,q) \simeq \colim_{x \in X} F(p,x) \otimes
      G(x,q),\]
  \item the unit $\bbone$ is given by 
    \[ \bbone(p,q) \simeq \Map_{X}(p,q) \otimes \bbone \simeq 
    \begin{cases}
      \emptyset, & p \not\simeq q\\
      \Omega_{p}X \otimes \bbone, & p \simeq q,
    \end{cases}
    \]
    where $\bbone$ is the unit of $\mathcal{C}$,
  \item we have
    $\Alg_{\Dop}(\Fun(X \times X, \mathcal{C})) \simeq \Alg_{\Dop_{X}}(\mathcal{C})$.
  \end{enumerate}
\end{cor}

\section{The Composition Product and $\infty$-Operads}\label{sec:appl}
In this section we apply our results on Day convolutions to describe
\iopds{} as associative algebras in double \icats{}.  We first
consider ordinary \iopds{} (in spaces) in \S\ref{subsec:opd}, and then
enriched \iopds{} in \S\ref{subsec:enropd}. We also briefly observe,
in \S\ref{subsec:bar}, that a version of the bar-cobar adjunction
between \iopds{} and \icoopds{} follows from this description of
\iopds{}.

\subsection{$\infty$-Operads as Associative
  Algebras}\label{subsec:opd}
In this subsection we will see that \iopds{} are given by
associative algebras in a double \icat{} of symmetric collections (or
coloured symmetric sequences) in $\mathcal{S}$. For this we use
Barwick's model of \iopds{} from \cite{BarwickOpCat}; this is known to
be equivalent to other models of \iopds{} thanks to the results of
\cite{BarwickOpCat,iopdcomp,CisinskiMoerdijkDendSeg,CisinskiMoerdijkSimplOpd}.
Before we recall Barwick's definition we first introduce some
notation:

\begin{defn}
  Write $\xF$ for a skeleton of the category of finite sets,
  with objects $\mathbf{k} := \{1,\ldots,k\}$, $k = 0,1,\ldots$. Let
  $\DF$ be the category with objects pairs
  $([n], f \colon [n] \to \xF)$ with a morphism
  $([n], f) \to ([m], g)$ given by a morphism
  $\phi \colon [n] \to [m]$ in $\simp$ and a natural transformation
  $\eta \colon f \to g \circ \phi$ such that
  \begin{enumerate}[(i)]
  \item the map $\eta_{i} \colon f(i) \to g(\phi(i))$ is injective for
    all $i = 0,\ldots,m$,
  \item the commutative square
    \csquare{f(i)}{g(\phi(i))}{f(j)}{g(\phi(j))}{\eta_{i}}{}{}{\eta_{j}}
    is a pullback square for all $0 \leq i \leq j \leq m$.
  \end{enumerate}
\end{defn}

\begin{notation}
  For $I = ([n], f)$ in $\DF$, we write
  $I|_{ij} := ([j-i], f|_{\{i,i+1,\ldots,j\}})$ for
  $0 \leq i < j \leq n$, and $I|_{i} := ([0], f(i))$. Moreover, for
  $x \in f(n)$ we write $I_{x} := ([n], f_{x})$ where $f_{x}$ is
  obtained by taking fibres at $x$.
\end{notation}

\begin{defn}\label{defn:DFSeg}
  A presheaf $F \colon \DF^{\op} \to \mathcal{S}$ is a
\emph{Segal operad} if it satisfies the following three ``Segal conditions'':
\begin{enumerate}[(1)]
\item for every object $I = ([n], f)$ of $\DF$, the natural map
  \[F(I) \to F(I|_{01})
  \times_{F(I|_{1})} \cdots \times_{F(I|_{n-1})} F(I|_{(n-1)n})\]
is an equivalence,
\item for every object $I = ([1], \mathbf{k} \to \mathbf{l})$, the natural map
\[ F(I) \to \prod_{x \in \mathbf{l}} F(I_{x})\] is an
equivalence,
\item for every object $I = ([0], \mathbf{k})$, the natural map
  \[ F(I) \to \prod_{x \in \mathbf{k}} F(I_{x})\]
  is an equivalence.
\end{enumerate}
We write $\Seg_{\DFop}^{\opd}(\mathcal{S})$ for the full subcategory of $\mathcal{P}(\DF)$
spanned by the Segal operads.
\end{defn}
\begin{remark}
  In the presence of condition (1), conditions (2) and (3) can be
  replaced by the following more general version:
  \begin{quotation}
    For every object $I = ([n], f)$ of $\DF$, the natural map
    \[ F(I) \to \prod_{x \in f(n)} F(I_{x})\]
    is an equivalence.
  \end{quotation}
\end{remark}

Segal presheaves on $\DF$ describe the algebraic structure of
\iopds{}: If we write $\mathfrak{e} := ([0], \mathbf{1})$ and
$\mathfrak{c}_{n} := ([1], \mathbf{n} \to \mathbf{1})$, then the Segal
conditions describe how $F(I)$ decomposes as a limit of
$F(\mathfrak{e})$ and $F(\mathfrak{c}_{n})$. We can think of an
objects of $\DF$ as a forest with levels; then $\mathfrak{e}$
correponds to a plain edge and $\mathfrak{c}_{n}$ to a corolla with
$n$ leaves, while the Segal condition corresponds to the decomposition
of a forest into its edges and corollas. If $F$ is viewed as an
\iopd{}, the value $F(\mathfrak{e})$ is the space of objects of $F$,
while $F(\mathfrak{c}_{n})$ is the space of $n$-ary morphisms.

The following is the  starting point for our construction of the composition product on
symmetric sequences:
\begin{propn}\label{propn:DFadm}
  The projection $\DF^{\op} \to \Dop$ is an $\mathcal{X}$-admissible
  double \icat{} for any cocomplete \icat{} $\mathcal{X}$ with
  pullbacks where colimits are universal.
\end{propn}

For the proof we need some notation and a lemma:
\begin{defn}
  Suppose $\phi \colon [n] \to [m]$ is an active map in $\simp$ and 
  $A = (\mathbf{a}_{0} \to \cdots \to \mathbf{a}_{n})$ is an object of
  $(\DF)_{[n]}$. An object of the slice $(\DF)_{[m], A/}$, defined
  using $\phi$, is an object $(\mathbf{b}_{0} \to \cdots
  \to \mathbf{b}_{m})$ of $(\DF)_{[m]}$ together with injective maps
  $\mathbf{a}_{i} \to \mathbf{b}_{\phi(i)}$ such that the squares
  \nolabelcsquare{\mathbf{a}_i}{\mathbf{a}_{i+1}}{\mathbf{b}_{\phi(i)}}{\mathbf{b}_{\phi(i+1)}}
  are cartesian. Let $(\DF)_{[m], A/}^{\txt{iso}}$ denote the full
  subcategory of $(\DF)_{[m], A/}$ containing those objects where the
  map $\mathbf{a}_{n} \to \mathbf{b}_{\phi(m)}$ is an
  isomorphism.
\end{defn}

\begin{lemma}\label{lem:DFisocinit}
  Let $\phi \colon [n] \to [m]$ be an active morphism in $\simp$, and
  let $A = (\mathbf{a}_{0} \to \cdots \to \mathbf{a}_{n})$ be an
  object of $(\DF)_{[n]}$. For $0 \leq i \leq j \leq n$ let $A_{ij} :=
  (\mathbf{a}_{i} \to \mathbf{a}_{i+1} \to \cdots \to \mathbf{a}_{j})$.
  \begin{enumerate}[(i)]
  \item The inclusion
    $(\DF)_{[m],A/}^{\txt{iso}} \hookrightarrow (\DF)_{[m], A/}$ is
    coinitial.
  \item For every $k$, $0\leq k \leq n$, the functor $(\DF)_{[m],A/}^{\txt{iso}} \to
    (\DF)_{[\phi(k)],A_{0k}/}^{\txt{iso}} \times (\DF)_{[m-\phi(k)],A_{kn}/}^{\txt{iso}}$
    is an equivalence.
  \end{enumerate}
\end{lemma}
\begin{proof}
  By \cite{HTT}*{Theorem 4.1.3.1} for part (i) it suffices to check that for all
  $B \in (\DF)_{[m], A/}$ the category
  $((\DF)_{[m], A/}^{\txt{iso}})_{/B}$ is weakly contractible.
  Observe that the projection $(\DF)_{[m]} \to \xF_{\txt{inj}}$
  given by evaluation at $[m]$ is a cartesian fibration, where
  $\xF_{\txt{inj}}$ denotes the subcategory of $\xF$
  containing only the injective maps.  The category
  $((\DF)_{[m], A/}^{\txt{iso}})_{/B}$ therefore has a terminal
  object, given by the cartesian morphism $B' \to B$ over the map
  $A_{n} \to B_{\phi(m)}$, which implies that it is weakly
  contractible. (Indeed this is the unique object of $((\DF)_{[m],
    A/}^{\txt{iso}})_{/B}$, which is actually a contractible \igpd{}.)
 Part (ii) is immediate from the definition.
\end{proof}

\begin{proof}[Proof of \cref{propn:DFadm}]
  The functor $\DF \to \simp$ is a cartesian fibration, and the
  corresponding functor $\Dop \to \CatI$ takes $[n]$ to the category
  $(\DF)_{[n]}$ where
  \begin{itemize}
  \item an object is a sequence $\mathbf{a}_{0} \to \cdots \to
    \mathbf{a}_{n}$ of morphisms in $\xF$,
  \item a morphism is a commutative diagram
    \[
    \begin{tikzcd}
    \mathbf{a}_{0} \arrow{d} \arrow{r} \arrow[phantom]{dr}[very near
start]{\lrcorner} & \mathbf{a}_{1} \arrow{r} \arrow{d} \arrow[phantom]{dr}[very near
start]{\lrcorner} &
    \cdots \arrow{r} & \mathbf{a}_{n-1} \arrow{d} \arrow{r}\arrow[phantom]{dr}[very near
start]{\lrcorner} &
    \mathbf{a}_{n} \arrow{d} \\
    \mathbf{b}_{0} \arrow{r} & \mathbf{b}_{1} \arrow{r}  &
    \cdots \arrow{r} & \mathbf{b}_{n-1} \arrow{r} &
    \mathbf{b}_{n}
    \end{tikzcd}
    \]
    where the squares are cartesian and the maps $\mathbf{a}_{i} \to
    \mathbf{b}_{i}$ are injective.
  \end{itemize}
  This clearly satisfies the Segal condition, \ie{}
  $(\DF)_{[n]} \simeq (\DF)_{[1]} \times_{(\DF)_{[0]}} \cdots
  \times_{(\DF)_{[0]}} (\DF)_{[1]}$. It follows that $\DF^{\op} \to
  \Dop$ is a cocartesian fibration corresponding to the functor
  $\Dop \to \CatI$ taking $[n]$ to $(\DF)_{[n]}^{\op}$ and so is also
  a double \icat{}.
  
  Suppose $\phi \colon [n] \to [m]$ is an active map in $\simp$. If
  $A = (\mathbf{a}_{0} \to \cdots \to \mathbf{a}_{n})$ and
  $A' = (\mathbf{a}_{0} \to \cdots \to \mathbf{a}_{k})$ and $A'' =
  (\mathbf{a}_{k} \to \cdots \to \mathbf{a}_{n})$ then we must
  show that the natural map
  \[ \colim_{((\DF)_{[m],A/})^{\op}} F \to
    \colim_{((\DF)_{[\phi(k)],A'/})^{\op}} F
    \times_{\colim_{((\DF)_{[0],\mathbf{a}_{k}/})^{\op}} F} 
    \colim_{((\DF)_{[n-\phi(k)],A''/})^{\op}} F\]
  is an equivalence for any appropriate functor $F$.
  
  We have a commutative square
  \[
    \begin{tikzcd}
      \colim_{((\DF)^{\txt{iso}}_{[m],A/})^{\op}} F \arrow{r} \arrow{d}{\sim}& 
      \colim_{((\DF)^{\txt{iso}}_{[\phi(k)],A'/})^{\op}} F
      \times_{\colim_{((\DF)^{\txt{iso}}_{[0],\mathbf{a}_{k}/})^{\op}} F} 
      \colim_{((\DF)^{\txt{iso}}_{[n-\phi(k)],A''/})^{\op}} F \arrow{d}{\sim} \\     
      \colim_{((\DF)_{[m],A/})^{\op}} F \arrow{r} & 
      \colim_{((\DF)_{[\phi(k)],A'/})^{\op}} F
      \times_{\colim_{((\DF)_{[0],\mathbf{a}_{k}/})^{\op}} F} 
      \colim_{((\DF)_{[n-\phi(k)],A''/})^{\op}} F,
    \end{tikzcd}
  \]
  where the vertical maps are equivalences by \cref{lem:DFisocinit}(i). To
  see that the bottom horizontal map is an equivalence it hence suffices
  to show the top horizontal map is an equivalence.
  
  Here $(\DF)^{\txt{iso}}_{[0],\mathbf{a}_{k}/}$ is contractible,
  since it only contains the identity map of $\mathbf{a}_{k}$, while 
  the functor $(\DF)^{\txt{iso}}_{[m],A/} \to
  (\DF)^{\txt{iso}}_{[\phi(k)],A'/} \times
  (\DF)^{\txt{iso}}_{[n-\phi(k)],A''/}$ is an equivalence by \cref{lem:DFisocinit}(ii). Thus the
  top horizontal functor is
    \begin{multline*}
\colim_{(X,Y) \in ((\DF)^{\txt{iso}}_{[\phi(k)],A'/})^{\op} \times
      ((\DF)^{\txt{iso}}_{[n-\phi(k)],A''/})^{\op}} F(X)
    \times_{F([0],\mathbf{a}_{k})} F(Y)\\
     \to \colim_{X \in
      ((\DF)^{\txt{iso}}_{[\phi(k)],A'/})^{\op}} F(X)
    \times_{F([0],\mathbf{a}_{k})} \colim_{Y \in
      ((\DF)^{\txt{iso}}_{[n-\phi(k)],A''/})^{\op}} F(Y), 
  \end{multline*}
 which is an
  equivalence since colimits in $\mathcal{X}$ are universal.
\end{proof}

\begin{notation}
  Given a morphism $f \colon \mathbf{a} \to \mathbf{b}$ of finite
  sets, we write $\txt{Fact}(f)$ for the groupoid
  $((\DF)^{\txt{iso}}_{[2],\mathbf{a}\to\mathbf{b}/})^{\op}$ of
  factorizations of $f$.
\end{notation}

Applying \cref{propn:dayconvdouble} and \cref{cor:Meq} we get:
\begin{cor}\label{cor:DFDayCon}
  There is a double \icat{} $\hDFSop$ with the universal property that
  for any \gnsiopd{} $\mathcal{O}$ there is a natural equivalence
  \[ \Alg_{\mathcal{O}}(\hDFSop) \simeq \Seg_{\mathcal{O} \times_{\Dop}
      \DF^{\op}}(\mathcal{S}).\]
  In particular, $\Alg_{\Dop}(\hDFSop) \simeq
  \Seg_{\DF^{\op}}(\mathcal{S})$.
\end{cor}
Here $\Seg_{\DF^{\op}}(\mathcal{S})$ is the \icat{} of presheaves on
$\DF$ that satisfy condition (1) in \cref{defn:DFSeg}. The double
\icat{} $\hDFSop$ can be described as follows:
\begin{itemize}
\item Objects are functors $\Finj^{\op} \simeq \simp_{\xF,[0]}^{\op}
  \to \mathcal{S}$,
\item Vertical morphisms are natural transformations of such functors.
\item A horizontal morphism $\Phi$ from $F$ to $G \colon \Finj^{\op} \to
  \mathcal{S}$ assigns to $([1], \mathbf{a} \to \mathbf{b})$ a span
  \[ F(\mathbf{a}) \from \Phi(\mathbf{a} \to \mathbf{b}) \to
    G(\mathbf{b}).\]
\item Squares are natural transformations of such diagrams.
\item If $\Phi$ is a horizontal morphism from $F$ to $G$ and $\Psi$ is
  a horizontal morphism from $G$ to $H$ then their composite assigns
  to $([1], \mathbf{a} \xto{f} \mathbf{b})$ the space over $F(\mathbf{a})$
  and $H(\mathbf{b})$ given by
  \[ \colim_{\mathbf{a} \to \mathbf{x} \to \mathbf{b} \in \Fact(f)} \Phi(\mathbf{a}
    \to \mathbf{x}) \times_{G(\mathbf{x})} \Psi(\mathbf{x} \to
    \mathbf{b}).\]
\end{itemize}

\begin{defn}
  Let $\COLL(\mathcal{S})$ denote the full subcategory of
  $\hDFSop$ spanned by the functors $(\bbS\DF^{\op})_{n} \to
  \mathcal{S}$ such that their inert restrictions to
  $(\bbS\DF^{\op})_{0} \simeq \Finj^{\op}$ take coproducts of finite
  sets to products, and their inert restrictions to
  $(\bbS\DF^{\op})_{1}$ moreover satisfy condition (2) in
  \cref{defn:DFSeg} when restricted to $(\DF^{\op})_{[1]}$.
\end{defn}

\begin{lemma}\label{lem:COLLSdouble}
  $\COLL(\mathcal{S})$ is a sub-double \icat{} of $\hDFSop$.
\end{lemma}
\begin{proof}
  It follows from Lemma~\ref{lem:subgnsiopd} that $\COLL(\mathcal{S})$
  is a \gnsiopd{}, so it only remains to check that the cocartesian
  morphisms restrict to $\COLL(\mathcal{S})$. To see this it suffices
  to check that the horizontal morphisms in $\COLL(\mathcal{S})$ are
  closed under composition. If $\Phi$ is a horizontal morphism from
  $F$ to $G$ and $\Psi$ is one from $G$ to $H$, and both lie in
  $\COLL(\mathcal{S})$, then we have
  \[
    \begin{split}
      (\Phi \odot_{G} \Psi)(\mathbf{a} \to \mathbf{b}) & \simeq
      \colim_{\mathbf{a} \to \mathbf{x} \to \mathbf{b}}
      \Phi(\mathbf{a} \to \mathbf{x}) \times_{G(\mathbf{x})}
      \Psi(\mathbf{x} \to \mathbf{b}) \\
      & \simeq \colim_{(\mathbf{a}_{i} \to \mathbf{x}_{i} \to
        \mathbf{1}) \in \prod_{i \in \mathbf{b}} \Fact(\mathbf{a}_{i}
        \to \mathbf{1})} \left( \prod_{i \in \mathbf{b}}
        \Phi(\mathbf{a}_{i} \to \mathbf{x}_{i})\right ) \times_{ \left(
          \prod_{i \in \mathbf{b}} G(\mathbf{x}_{i}) \right)}
      \left( \prod_{i \in \mathbf{b}}
        \Psi(\mathbf{x}_{i} \to \mathbf{1})\right ) \\
      & \simeq \colim_{(\mathbf{a}_{i} \to \mathbf{x}_{i} \to
        \mathbf{1}) \in \prod_{i \in \mathbf{b}} \Fact(\mathbf{a}_{i}
        \to \mathbf{1})} \prod_{i \in \mathbf{b}} \Phi(\mathbf{a}_{i}
      \to \mathbf{x}_{i}) \times_{G(\mathbf{x}_{i})}
      \Psi(\mathbf{x}_{i} \to \mathbf{1}) \\
      & \simeq \prod_{i \in \mathbf{b}} 
\colim_{\mathbf{a}_{i} \to \mathbf{x}_{i} \to
        \mathbf{1} \in  \Fact(\mathbf{a}_{i}
        \to \mathbf{1})} \Phi(\mathbf{a}_{i}
      \to \mathbf{x}_{i}) \times_{G(\mathbf{x}_{i})}
      \Psi(\mathbf{x}_{i} \to \mathbf{1}) \\
      & \simeq \prod_{i \in \mathbf{b}} (\Phi \odot_{G}
      \Psi)(\mathbf{a}_{i} \to \mathbf{1}),
    \end{split}
  \]
  \ie{} $\Phi \odot_{G}\Psi$ also lies in $\COLL(\mathcal{S})$, as required.
\end{proof}

The double \icat{} $\COLL(\mathcal{S})$ can be described as follows:
\begin{itemize}
\item Its objects can be identified with spaces (since functors $\Finj^{\op} \to
  \mathcal{S}$ in $\COLL(\mathcal{S})_{0}$ are determined by their
  value at $\mathbf{1}$).
\item Its vertical morphisms are maps of spaces.
\item A horizontal morphism $\Phi$ from $X$ to $Y$ is determined by assigning to $([1],
  \mathbf{n} \to \mathbf{1})$ a span
  \[ X^{\times n} \from \Phi(n) \to Y,\]
  where $\Phi(n)$ has a $\Sigma_{n}$-action compatible with permuting
  the factors of $X^{\times n}$.
\item A square is a natural transformation of such diagrams.
\item If $\Phi$ is a horizontal morphism from $X$ to $Y$ and $\Psi$ is
  one from $Y$ to $Z$, then their composite assigns to $([1],
  \mathbf{n} \to \mathbf{1})$ the space over $X^{\times n} \times Z$
  given by
  \[ \colim_{\mathbf{n} \to \mathbf{m} \to \mathbf{1}} \Phi(\mathbf{n}
    \to \mathbf{m}) \times_{Y^{\times m}} \Psi(m),\]
  where $\Phi(\mathbf{n} \to \mathbf{m}) \simeq \prod_{i = 1}^{m} \Phi(n_{i})$.
\end{itemize}

Restricting \cref{cor:DFDayCon} to $\COLL(\mathcal{S})$, we get:
\begin{cor}
  There is an equivalence $\Alg_{\Dop}(\COLL(\mathcal{S})) \simeq
  \Seg^{\opd}_{\DFop}(\mathcal{S})$.\qed
\end{cor}
In other words, \iopds{} can be described as associative algebras in
the double \icat{} $\COLL(\mathcal{S})$. Moreover, applying
\cref{cor:dayconvmonalg} we see that the restriction
$\Alg_{\Dop}(\COLL(\mathcal{S})) \to \mathcal{S}$ is a cartesian
fibration. If we write
\[ \Coll_{X}(\mathcal{S}) := \COLL(\mathcal{S})(X,X) \]
for the \icat{} of $X$-collections in $\mathcal{S}$, the fibre at
 $X \in \mathcal{S}$ is given by
$\Alg_{\Dop}(\Coll_{X}(\mathcal{S}))$. To describe the monoidal
structure on $\Coll_{X}(\mathcal{S})$ we first need to introduce some notation:

\begin{notation}\label{not:xFXY}
  Let $\xF^{\simeq}$ denote the maximal subgroupoid of $\xF$,
  and write $j \colon \xFeq \to (\DF^{\op})_{[1]}$ for the fully
  faithful functor taking $\mathbf{n}$ to $\mathbf{n} \to
  \mathbf{1}$. Given $X \in \mathcal{S}$,  we write
  $(\DF^{\op})_{[0],X} \to (\DF^{\op})_{[0]}$ for the 
  left fibration corresponding to the
  unique product-preserving functor
  $(\DF^{\op})_{[0]} \simeq \Finj^{\op} \to \mathcal{S}$  that takes $\mathbf{1}$ to $X$.
  Moreover, for $X,Y \in \mathcal{S}$ let $(\DF^{\op})_{[1],X,Y}$
  denote the pullback $(\DF^{\op})_{[1]}
  \times_{(\DF^{\op})_{[0]}^{\times 2}} \left((\DF^{\op})_{[0],X}
    \times (\DF^{\op})_{[0],Y}\right)$.
  If we define $\xFeq_{X} := \coprod_{n=0}^{\times} X^{\times
    n}_{h\Sigma_{n}}$ to be the free commutative monoid on the \igpd{} $X$, then
  we have a pullback
  \[
    \begin{tikzcd}
      \xFeq_{X}\times Y \arrow{r}{j_{X,Y}} \arrow{d} & (\DF^{\op})_{[1],X,Y}
      \arrow{d} \\
      \xF^{\simeq} \arrow{r}{j} & (\DF^{\op})_{[1]}.
    \end{tikzcd}
  \]
\end{notation}

\begin{lemma}\label{lem:COLLSXY}
  For any $X \in \mathcal{S}$, the \icat{} 
  $\COLL(\mathcal{S})(X,Y)$ of horizontal morphisms from $X$ to $Y$ is
  equivalent to the full subcategory of
  $\Fun((\DF^{\op})_{[1],X,Y}, \mathcal{S})$ spanned by functors that are
  right Kan extensions along $j_{X,Y}$, so that
  \[ \COLL(\mathcal{S})(X,Y) \simeq \Fun(\xF_{X}^{\simeq} \times Y,
    \mathcal{S}).\]
\end{lemma}
\begin{proof}
  By \cref{propn:M1FG} we may identify $\hDFSop(F,G)$ with
  $\Fun((\DF^{\op})_{[1],F,G}, \mathcal{S})$ for any functors $F,G
  \colon (\DF^{\op})_{[0]} \to \mathcal{S}$. For the objects that lie
  in $\COLL(\mathcal{S})$ these are functors $(\DF^{\op})_{[1],X,Y}
  \to \mathcal{S}$, and under this identification it is easy to see
  that the functors that lie in $\COLL(\mathcal{S})(X,Y)$ are
  precisely those that are right Kan extended from
  $\xFeq_{X}\times Y$.
\end{proof}

\begin{remark}
  In particular, we may identify the \icat{} $\Coll_{*}(\mathcal{S})$
  of horizontal endomorphisms of the point with the \icat{}
  $\Fun(\xF^{\simeq}, \mathcal{S})$ of symmetric sequences in
  $\mathcal{S}$. More generally, the \icat{} $\Coll_{X}(\mathcal{S})$ is
  equivalent to $\Fun(\xF^{\simeq}_{X} \times X, \mathcal{S})$, the \icat{} of
 $X$-collections, or $X$-coloured symmetric sequences, in
  $\mathcal{S}$.
\end{remark}

\begin{notation}
  For a functor $F \colon \xFeq_{X}\times Y \to \mathcal{C}$, we will denote
  its value at $((x_{1},\ldots,x_{n}), y) \in X^{\times n}_{h
    \Sigma_{n}}\times Y$ by $F\binom{x_{1},\ldots,x_{n}}{y}$.
\end{notation}

\begin{cor}\label{cor:symmseqprodX}
  The \icat{} $\Fun(\xFeq_{X} \times X, \mathcal{S})$ has a monoidal
  structure such that
  \begin{enumerate}[(i)]
  \item the tensor product of $F$ and $G$ is given by
    \[(F \circ G)\binom{x_{1},\ldots,x_{n}}{z}  \simeq 
      \colim_{\substack{\mathbf{n} \to \mathbf{m} \to
          \mathbf{1}\\(y_{i}) \in X^{m}}}
        \prod_{i \in
      \mathbf{m}} F\binom{x_{k}: k \in \mathbf{n}_{i}}{y_{i}}
     \times G\binom{y_{1},\ldots,y_{m}}{z},\]
where the colimit is over $((\DFX)_{[2],
        (\mathbf{n} \to \mathbf{1}, (x_{i}), z)/}^{\txt{iso}})^{\op}$,
  \item the unit $\bbone_{X}$ is given by
    \[ \bbone_{X}\binom{x_{1},\ldots,x_{n}}{y} \simeq
    \begin{cases}
      \emptyset, & \mathbf{n} \not\cong \mathbf{1},\\
      \Map_{X}(x_{1},y), & \mathbf{n} \cong \mathbf{1},
    \end{cases}
    \]
  \item we have
    $\Alg_{\Dop}(\Fun(\xFeq_{X} \times X, \mathcal{S})) \cong
    \Seg^{\opd}_{\DFop}(\mathcal{S})_{X}$.
  \end{enumerate}
\end{cor}

\begin{remark}
  In particular, the \icat{} $\Fun(\xF^{\simeq}, \mathcal{S})$ of
  symmetric sequences has a monoidal structure with tensor product
  given by
      \[(F \circ G)(\mathbf{n}) \simeq 
\colim_{(\mathbf{n} \to \mathbf{m} \to \mathbf{1}) \in
  \Fact(\mathbf{n} \to \mathbf{1})} \prod_{i \in
      \mathbf{m}} F(\mathbf{n}_{i}) \times G(\mathbf{m}).\]
  This formula is easily seen to agree with the usual
  formula for the composition product of symmetric sequences by
  expanding out $\Fact(\mathbf{n} \to \mathbf{1})$ as a coproduct of
  its components, \cf{}
  \cite{DwyerHess}*{Lemma A.4}.
\end{remark}

\begin{remark}
  In \cite{BarwickOpCat}, Barwick defines $\Phi$-\iopds{} for
  \emph{operator categories} $\Phi$ as Segal presheaves on categories
  $\simp_{\Phi}$, of which $\simp_{\xF}$ is a special case. The
  proof of \cref{cor:symmseqprodX} works for any operator
  category, giving a monoidal
  structure on the \icat{} $\Fun(\Phi^{\simeq}_{X}, \mathcal{S})$ of
  ``$X$-coloured $\Phi$-symmetric sequences'' where associative algebras are
  $\Phi$-\iopds{} with $X$ as their space of objects. In particular,
  replacing $\xF$ with the category $\mathbb{O}$ of ordered finite
  sets we obtain the analogous results for non-symmetric \iopds{}.
\end{remark}

\subsection{Enriched  $\infty$-Operads as Associative Algebras}\label{subsec:enropd}
In this subsection we extend the results of the previous subsection to
\iopds{} enriched in a symmetric monoidal \icat{}. The starting point
is the following analogue of \cref{propn:enrDCgpd} for
$\DFop$-monoidal \icats{}:
\begin{propn}
  Let $\mathcal{U}^{\otimes}$ be a $\DFop$-monoidal \icat{} that is
  compatible with colimits indexed by \igpds{}. Then
  $\hDFUop$ is a framed double \icat{}. If
  $\mathcal{U}^{\otimes}\to \mathcal{V}^{\otimes}$ is a
  $\DFop$-monoidal functor between such $\DFop$-monoidal \icats{} such
  that each functor $\mathcal{U}_{X} \to \mathcal{V}_{X}$ for
  $X \in (\DF)_{[1]}$ preserves colimits indexed by \igpds{}, then the
  natural morphism of \gnsiopds{}
  $\hDFUop \to \hDFVop$
  preserves cocartesian morphisms.
\end{propn}
\begin{proof}
  Follows as in the proof of \cref{propn:enrDCgpd}, using
  \cref{lem:DFisocinit} to restrict to colimits indexed by \igpds{}.
\end{proof}
We now recall some definitions from
\cite{enropd}; we refer the reader there for motivation for these
definitions.
\begin{defn}
  Let $V \colon \DF^{\op} \to \xF_{*}$ be the functor of
  \cite{enropd}*{Definition 2.2.11}, taking
  $([n], \mathbf{a}_{0} \to \cdots \to \mathbf{a}_{n})$ to
  $(\coprod_{i = 1}^{n} \mathbf{a}_{i})_{+}$, and a morphism
  $([n], \mathbf{a}_{0}\to \cdots \to \mathbf{a}_{n})\to ([m],
  \mathbf{b}_{0} \to \cdots \to \mathbf{b}_{m})$ over
  $\phi \colon [n] \to [m]$ in $\Dop$ to the map
  $(\coprod_{i=1}^m \mathbf{a}_{i})_{+} \to (\coprod_{j=1}^n
  \mathbf{b}_{j})_{+}$ given on the component $\mathbf{a}_{i}$ by the
  map $\mathbf{a}_{i} \to (\coprod_{j=1}^n \mathbf{b}_{j})_{+}$ taking
  $x \in \mathbf{a}_{i}$ to an object $y \in \mathbf{b}_{j}$ if
  $\phi(j-1)< i \leq \phi(j)$ and the map
  $\mathbf{a}_{i} \to \mathbf{a}_{\phi(j)}$ takes $x$ to the image of
  $y$ under the map $\mathbf{b}_{j} \to \mathbf{a}_{\phi(j)}$, and to
  the base point $*$ otherwise. The functor $V$ assigns to a forest
  its set of vertices with an added basepoint. Note that $V$ assigns
  every morphism in $\DFop$ that lies over an identity morphism in
  $\Dop$ to an inert morphism in $\xF_{*}$.
\end{defn}

\begin{defn}
  If $\mathcal{V}^{\otimes} \to \xF_{*}$
  is a symmetric monoidal \icat{}, and $\mathcal{V}_{\otimes} \to
  \xF_{*}^{\op}$ is the corresponding cartesian fibration, then we define the
  \icat{} $\DFV$ by the pullback square
  \csquare{\DFV}{\mathcal{V}_{\otimes}}{\DF}{\xF_{*}^{\op}.}{}{}{}{V^{\op}}
  Note that $V
  \colon \DF^{\op} \to \xF_{*}$ satisfies
  \[ V([n], \mathbf{a}_{0} \to
    \cdots \to \mathbf{a}_{n}) \cong V([1], \mathbf{a}_{0} \to
    \mathbf{a}_{1}) + \cdots + V([1], \mathbf{a}_{n-1} \to
    \mathbf{a}_{n}),\]
  which implies that $\DFVop$ is a $\DF$-monoidal \icat{}.
\end{defn}

\begin{remark}
  Let $\mathcal{V}$ be a symmetric monoidal \icat{} compatible with
  colimits indexed by \igpds{}. The double \icat{}
  $\widehat{\simp}^{\op}_{\xF,\DFVop}$ can be described as follows:
  \begin{itemize}
  \item The objects are functors $(\DF)_{[0]}^{\op} \to \mathcal{S}$,
    and the vertical morphisms are natural transformations of such
    functors.
  \item A horizontal morphism from $F$ to $G$ is a functor
    \[ \Phi \colon  (\DF)_{[1],F,G}^{\op} \to (\DF)_{[1]}^{\op} \times_{\xF_{*}}
      \mathcal{V}^{\otimes} \]
    over $(\DF)_{[1]}^{\op}$. This thus assigns to an object $(\mathbf{n}
    \xto{f} \mathbf{m}, p \in F(\mathbf{n}), q \in G(\mathbf{m}))$ an
    object $(\Phi(f,p,q)_{i})_{i\in \mathbf{m}}$ of $\mathcal{V}^{\times |\mathbf{m}|}$.
  \item If $\Phi$ is a horizontal morphism from $F$ to $G$ and $\Psi$
    is one from $G$ to $H$, then their composite is the functor from
    $(\DF)_{[1],F,H}^{\op}$ given by
    \[ (\mathbf{n}
      \xto{f} \mathbf{m}, p \in F(\mathbf{n}), q \in H(\mathbf{m}))
      \mapsto \left( \colim_{(\mathbf{n} \xto{g} \mathbf{x} \xto{h} \mathbf{m}) \in
        \Fact(f)} \colim_{t \in G(\mathbf{x})} \bigotimes_{j \in
        h^{-1}(i)} \Phi(g,p,t)_{j}\otimes \Psi(h,t,q)_{i} \right)_{i
      \in \mathbf{m}}
    \]
  \end{itemize}
\end{remark}

\begin{defn}
  Let $\mathcal{V}$ be a symmetric monoidal \icat{} compatible with
  colimits indexed by \igpds{}. We denote by
  $\COLL(\mathcal{V})$ the full subcategory of
  $\widehat{\simp}^{\op}_{\xF,\DFVop}$ spanned by the objects
  \begin{enumerate}[(1)]
  \item   whose
  inert restrictions to $[0]$ are given by functors
  $(\DF)_{[0]}^{\op} \simeq \Finj^{\op} \to \mathcal{S}$
  that take coproducts of finite sets to products,
\item whose inert restrictions to $[1]$ correspond to functors
  \[(\DF)_{[1],F,G}^{\op} \to (\DF)_{[1]}^{\op} \times_{\xF_{*}}
    \mathcal{V}^{\otimes} \]
  that send all morphisms to cocartesian morphisms in the target.
\end{enumerate}
This is a sub-double \icat{} of $\widehat{\simp}^{\op}_{\xF,\DFVop}$
by a variant of the proof of \cref{lem:COLLSdouble}.
\end{defn}

\begin{remark}
  A functor $F \colon \Finj^{\op} \to \mathcal{S}$ that satisfies
  condition (1) is the right Kan extension of its restriction to the
  object $\mathbf{1}$. Thus the objects of $\COLL(\mathcal{V})$ can
  equivalently be described as spaces. Since the restriction of the
  functor $V$ to $(\DF)^{\op}_{[1]}$ sends all
  morphisms to inert morphisms in $\xF_{*}$, the functor
  $(\DF)^{\op}_{[1]} \to \CatI$ corresponding to the cocartesian
  fibration $(\DF)_{[1]}^{\op} \times_{\xF_{*}}
  \mathcal{V}^{\otimes}$ is also a right Kan extension of its restriction
  to the full subcategory $\xFeq$, and this restriction is the constant
  functor with value $\mathcal{V}$. Thus a horizontal morphism from
  $X$ to $Y$ in $\COLL(\mathcal{V})$ is uniquely determined by its
  restriction to a functor $\xFeq_{X}\times Y \to \mathcal{V}$.
\end{remark}

\begin{notation}
  We say a morphism in $\DFop$ is \emph{operadic inert} if it lies
  over an inert morphism in $\Dop$. 
  Let $\Alg^{\opd}_{\DFXop}(\mathcal{V})$ denote the full subcategory
  of $\Alg_{\DFXop/\DFop}(V^{*}\mathcal{V})$ spanned by morphisms that
  take operadic inert morphisms to cocartesian morphisms; we call such
  objects \emph{operadic $\DFXop$-algebras}.
  We then write $\Algd_{\DFop}^{\opd}(\mathcal{V})$ for the full
  subcategory of $\Algd_{\DFop}(\mathcal{V})$ corresponding to
  operadic algebras for all $\DFXop$, $X \in \mathcal{S}$. Similarly,
  we can define operadic algebras and algebroids for $\mathcal{O}
  \times_{\Dop} \DFop$ where $\mathcal{O}$ is any \gnsiopd{}, by
  taking the operadic inert morphisms in $\mathcal{O} \times_{\Dop}
  \DFop$ to be those that lie over inert morphisms in $\mathcal{O}$
  and operadic inert morphisms in $\DFop$.
\end{notation}

Restricting the equivalence from \cref{propn:enrDCopd} to
$\COLL(\mathcal{V})$, we get:
\begin{cor}
  Let $\mathcal{V}$ be a symmetric monoidal \icat{} compatible with
  colimits indexed by \igpds{}. There is a framed double \icat{}
  $\COLL(\mathcal{V})$ where:
  \begin{itemize}
  \item objects are spaces and vertical morphisms are morphisms of
    spaces,
  \item horizontal morphisms from $X$ to $Y$ are functors
    $\xFeq_{X}\times Y
    \to \mathcal{V}$,
  \item if $\Phi \colon \xFeq_{X}\times Y \to \mathcal{V}$ and $\Psi \colon
    \xFeq_{Y}\times Z \to \mathcal{V}$ are two horizontal morphisms, their
    composite is the functor $\xFeq_{X}\times Z \to \mathcal{V}$ given by
    \[ \binom{x_{1},\ldots,x_{n}}{z}  \mapsto
      \colim_{\mathbf{n} \xto{f} \mathbf{m} \to \mathbf{1}}
      \colim_{(y_{j}) \in Y^{m}}  \bigotimes_{j} \Phi\binom{x_{i} : i \in
        f^{-1}(j)}{y_{j}} \otimes \Psi\binom{y_{1},\ldots,y_{m}}{z}.\]
  \end{itemize}
  We have $\Alg_{\mathcal{O}}(\COLL(\mathcal{V})) \simeq
  \Algd_{\mathcal{O} \times_{\Dop}
    \DFop}^{\txt{opd}}(\mathcal{V})$. Moreover, if $\phi \colon
  \mathcal{V} \to \mathcal{W}$ is a symmetric monoidal functor that
  preserves colimits indexed by \igpds{}, then $\phi$ induces a
  morphism of double \icats{} $\COLL(\mathcal{V}) \to
  \COLL(\mathcal{W})$ given on horizontal morphisms by composition
  with $\phi$.
\end{cor}

Let $\Coll_{X}(\mathcal{V}) := \COLL(\mathcal{V})(X,X)$; then 
 $\Coll_{X}(\mathcal{V})$ is equivalent to the
\icat{} $\Fun(\xFeq_{X} \times X, \mathcal{V})$ of symmetric $X$-collections in $\mathcal{V}$. The monoidal structure on $\Coll_{X}(\mathcal{V})$ has
the following description:
\begin{cor}\label{cor:Vcompprod}
  Let $\mathcal{V}$ be a symmetric monoidal \icat{} compatible with
  colimits indexed by \igpds{}. The
  \icat{} $\Fun(\xFeq_{X} \times X , \mathcal{V})$ has a monoidal
  structure such that
  \begin{enumerate}[(i)]
  \item the tensor product of $F$ and $G$ is given by
        \[(F \circ G)\binom{x_{1},\ldots,x_{n}}{z}  \simeq 
      \colim_{\substack{\mathbf{n} \to \mathbf{m} \to
          \mathbf{1}\\y_{i} \in X, i \in \mathbf{m}}}
        \bigotimes_{i \in
      \mathbf{m}} F\binom{x_{k}: k \in \mathbf{n}_{i}}{y_{i}}
     \otimes G\binom{y_{1},\ldots,y_{k}}{z},\]
where the colimit is over $((\DFX)_{[2],
        (\mathbf{n} \to \mathbf{1}, (x_{i}), z)/}^{\txt{iso}})^{\op}$
    \item the unit $\bbone_{X}$ is given by
    \[ \bbone_{X}\binom{x_{1},\ldots,x_{n}}{y} \simeq            
    \begin{cases}
      \emptyset, & \mathbf{n} \not\cong \mathbf{1},\\
      \Map_{X}(x_{1},y) \otimes \bbone, & \mathbf{n} \cong \mathbf{1},
    \end{cases}
    \]
  \item we have
    $\Alg_{\Dop}(\Fun(\xFeq_{X} \times X, \mathcal{V})) \simeq
    \Alg_{\DFX^{\op}}(\mathcal{V})$.
  \end{enumerate}
  Moreover, if $\phi \colon \mathcal{V} \to \mathcal{W}$ is a
  symmetric monoidal functor that preserves colimits indexed by
  \igpds{}, then composition with $\phi$ gives a monoidal functor
  $\Fun(\xFeq_{X} \times X, \mathcal{V}) \to \Fun(\xFeq_{X} \times X, \mathcal{W})$.
\end{cor}

\subsection{$\infty$-Cooperads and a Bar-Cobar
  Adjunction}\label{subsec:bar}
In this subsection we will apply Lurie's bar-cobar adjunction for
associative algebras \cite{HA}*{\S 5.2.2} to obtain a version of the bar-cobar
adjunction between \iopds{} and \icoopds{} with a fixed space of
objects. We first spell out the variant of \icoopds{} that this
applies to:
\begin{defn}
  For $X \in \mathcal{S}$ and $\mathcal{V}$ a symmetric monoidal
  \icat{} compatible with colimits indexed by \igpds{}, a \emph{$\mathcal{V}$-enriched
    $\infty$-cooperad} with space of objects $X$ is a coassociative
  coalgebra in $\Fun(\xFeq_{X} \times X, \mathcal{V})$, equipped
  with the monoidal structure of \cref{cor:Vcompprod}. We write
  \[\Coopd_{X}(\mathcal{V}):= \Alg_{\Dop}(\Fun(\xFeq_{X} \times X,
  \mathcal{V})^{\op})^{\op}\]
  for the \icat{} of $\mathcal{V}$-\icoopds{} with space of objects
  $X$, and 
  \[\CoopdA_{X}(\mathcal{V}):=
  \Coopd_{X}(\mathcal{V})_{\bbone_{X}/}\]
  for the \icat{} of coaugmented $\mathcal{V}$-\icoopds{}.
  Similarly, we write \[\Opd_{X}(\mathcal{V}) :=
    \Alg_{\Dop}(\Fun(\xFeq_{X} \times X,
  \mathcal{V}))\]
  and
  \[ \OpdA_{X}(\mathcal{V}) :=
  \Opd_{X}(\mathcal{V})_{/\bbone_{X}}.\]
\end{defn}

\begin{cor}
  Let $\mathcal{V}$ be a symmetric monoidal \icat{} compatible with
  small colimits. There is an adjunction
  \[ \txt{Bar} : \txt{Opd}^{\txt{aug}}_{X}(\mathcal{V})
  \rightleftarrows \Coopd^{\txt{coaug}}_{X}(\mathcal{V}) :
  \txt{Cobar},\]
  where on underlying symmetric sequences $\txt{Bar}(\mathcal{O})$ is
  given by
  \[ 
\bbone_{X} \circ_{\mathcal{O}} \bbone_{X} \simeq \colim_{\Dop} \left(
  \begin{tikzcd}
    \bbone \arrow{r} & \mathcal{O} 
    \arrow[yshift=0.7ex]{l}   
    \arrow[yshift=-0.7ex]{l} 
    \arrow[yshift=0.7ex]{r}   
    \arrow[yshift=-0.7ex]{r}
    & \mathcal{O} \circ \mathcal{O}  \arrow{l}
    \arrow[yshift=1.4ex]{l} \arrow[yshift=-1.4ex]{l}
\arrow{r}
    \arrow[yshift=1.4ex]{r} \arrow[yshift=-1.4ex]{r}      & \cdots
    \arrow[yshift=0.7ex]{l} 
    \arrow[yshift=-0.7ex]{l} 
    \arrow[yshift=2.1ex]{l} 
    \arrow[yshift=-2.1ex]{l} 
  \end{tikzcd}
\right)
\]
  and $\txt{Cobar}(\mathcal{Q})$ is given by
\[
\lim_{\simp} \left(
  \begin{tikzcd}
    \bbone 
    \arrow[yshift=0.7ex]{r}   
    \arrow[yshift=-0.7ex]{r} 
& \mathcal{Q}
\arrow{l}
    \arrow[yshift=1.4ex]{r} \arrow[yshift=-1.4ex]{r}
\arrow{r}
    & \mathcal{Q} \circ \mathcal{Q}  
    \arrow[yshift=0.7ex]{l} 
    \arrow[yshift=-0.7ex]{l} 
    \arrow[yshift=0.7ex]{r} 
    \arrow[yshift=-0.7ex]{r} 
    \arrow[yshift=2.1ex]{r} 
    \arrow[yshift=-2.1ex]{r} 
  & \cdots
    \arrow[yshift=1.4ex]{l} \arrow[yshift=-1.4ex]{l} \arrow{l}
  \end{tikzcd}
\right).
\]
\end{cor}
\begin{proof}
  Apply \cite{HA}*{Theorem 5.2.2.17} to the monoidal \icat{}
$\Fun(\xFeq_{X} \times X, \mathcal{V})_{\bbone_{X}//\bbone_{X}}$.
\end{proof}

\begin{remark}
  Here we have defined \icoopds{} as coalgebras in symmetric sequences,
  following the definition proposed in, for instance,
  \cite{FrancisGaitsgory}. However, the notion of cooperad in 
  $\mathbf{V}$ that is relevant in bar-cobar duality for operads
  often seems to be that of operads enriched in $\mathbf{V}^{\op}$
  (as for example used by Ching to define the bar--cobar adjunction
  for operads in spectra~\cite{ChingThesis}).
  In general these two
  versions of cooperads are quite different: an $\infty$-cooperad $\mathcal{O}$
  with one object in
  our sense has a comultiplication
  $\mathcal{O} \to \mathcal{O} \circ \mathcal{O}$, which is given by
  morphisms
  \[ \mathcal{O}(n) \to \coprod_{k=0}^{\infty}
  \left(\left(\coprod_{i_{1}+\cdots+i_{k} = n}
      \txt{Ind}_{\Sigma_{i_{1}} \times \cdots \times
        \Sigma_{i_{k}}}^{\Sigma_{n}}(\mathcal{O}(i_{1}) \otimes \cdots
      \otimes \mathcal{O}(i_{k})) \right) \otimes
    \mathcal{O}(k)\right)_{h \Sigma_{k}},\]
  while an \iopd{} enriched in $\mathcal{V}^{\op}$ would be given by
  morphisms
\[ \mathcal{O}(n) \to \prod_{k=0}^{\infty}
\left(\left(\prod_{i_{1}+\cdots+i_{k} = n} \txt{CoInd}_{\Sigma_{i_{1}}
      \times \cdots \times
      \Sigma_{i_{k}}}^{\Sigma_{n}}(\mathcal{O}(i_{1}) \otimes \cdots
    \otimes \mathcal{O}(i_{k}))\right) \otimes
  \mathcal{O}(k)\right)^{h \Sigma_{k}};\]
here
$\txt{Ind}_{\Sigma_{i_{1}} \times \cdots \times
  \Sigma_{i_{k}}}^{\Sigma_{n}}$
and
$\txt{CoInd}_{\Sigma_{i_{1}} \times \cdots \times
  \Sigma_{i_{k}}}^{\Sigma_{n}}$
denote induction and coinduction, respectively, or in other words left
and right Kan extension along the functor
$B(\Sigma_{i_{1}} \times \cdots \times \Sigma_{i_{k}}) \to
B\Sigma_{n}$. 

However, if we make some assumptions on both the \iopds{} we consider
and on the \icats{} we enrich in, then the two notions do agree: First
suppose $\mathcal{V}$ is a \emph{semiadditive} \icat{}, meaning it has
a zero object and finite biproducts (\ie{} finite products and
coproducts coincide). (For example, this holds in any stable \icat{},
such as those of spectra or chain complexes.) If we then restrict
ourselves to consider only \emph{reduced} \iopds{}
$\mathcal{O} \in \Opd_{*}(\mathcal{V})$, meaning \iopds{} such that
$\mathcal{O}(0) \simeq 0$, then the coproducts in
$\mathcal{O} \circ \mathcal{O}$ are finite and hence are equivalent to
products. Moreover, for such reduced symmetric sequences we can
rewrite the formula for the composition product without taking any
homotopy orbits:\footnote{I thank Gijs Heuts for pointing this out,
  and thereby correcting a serious misconception about the bar-cobar
  adjunction in the previous version of this paper.}
\[ (\mathcal{O} \circ \mathcal{O})(n) \simeq \bigoplus_{\mathbf{n}
    \twoheadrightarrow \mathbf{k} \twoheadrightarrow *}
  \mathcal{O}(i_{1}) \otimes \cdots \otimes
  \mathcal{O}(i_{k}) \otimes \mathcal{O}(k),\]
where $i_{j} = |\mathbf{n}_{j}|$. 
This is easy to see in the coordinate-free description as discussed in
  \S\ref{subsec:overview}: passing to reduced symmetric sequences
  means only surjective maps of sets appear, and these have no
  automorphisms in $\xF^{[2],\simeq}$. 

Therefore for reduced $\mathcal{O}$ a comultiplication
$\mathcal{O} \to \mathcal{O} \circ \mathcal{O}$ is equivalently
described by $\Sigma_{n}$-equivariant maps
\[ \mathcal{O}(n) \to \mathcal{O}(i_{1}) \otimes \cdots \otimes
  \mathcal{O}(i_{k}) \otimes \mathcal{O}(k)\] where
$n = i_{1} + \cdots + i_{k}$.  This is precisely the structure of
an \iopd{} enriched in $\mathcal{V}^{\op}$ with the same $n$-ary
operations as $\mathcal{O}$.

For reduced \iopds{} enriched in semiadditive \icats{}, we therefore expect
that the bar--cobar adjunction arising from the composition product is
the correct one for understanding bar--cobar duality for enriched \iopds{}. One might wonder
if there exists some more general version of a bar-cobar adjunction
without these restrictions, but this setting does in fact seem to
cover all the cases of bar-cobar duality for operads in the literature
that we are aware of.
\end{remark}

\end{document}